\documentclass[a4paper,11pt,reqno]{amsart}
\usepackage{amsfonts}
\usepackage{amsmath}
\usepackage{logicproof}
\usepackage{amssymb, amsthm,bm}
\usepackage{tabularx}
\usepackage{caption}
\usepackage{booktabs}
\usepackage{multirow,dsfont}
\usepackage{array}
\usepackage{floatrow}
\usepackage{float}
\usepackage{floatpag}
\usepackage{centernot}
\usepackage{enumitem}
\newcolumntype{L}[1]{>{\raggedright\let\newline\\\arraybackslash\hspace{0pt}}m{#1}}
\newcolumntype{C}[1]{>{\centering\let\newline\\\arraybackslash\hspace{0pt}}m{#1}}
\newcolumntype{R}[1]{>{\raggedleft\let\newline\\\arraybackslash\hspace{0pt}}m{#1}}
\usepackage[symbol]{footmisc}
\usepackage[labelfont=bf,format=plain,justification=raggedright,singlelinecheck=false]{caption}
\usepackage{mathrsfs}
\usepackage[]{mdframed}
\usepackage[colorlinks]{hyperref}
\usepackage{tcolorbox}
\usepackage{adjustbox}
\tcbuselibrary{theorems}
\newtcbtheorem[number within=section]{mytheo}{Note}
{colback=white!5,colframe=black!35!black,fonttitle=\bfseries}{th}
\DeclareMathOperator*{\esssup}{ess\,sup}

\numberwithin{equation}{section}

\usepackage{graphicx}
\usepackage{textgreek}

 {
      \theoremstyle{plain}
      \newtheorem{assumption}{Assumption}
  }
\newtheorem{theorem}{Theorem}[section]

\newtheorem{remark}[theorem]{Remark}
\newtheorem{corollary}[theorem]{Corollary}
\newtheorem{prop}[theorem]{Proposition}
\newtheorem{lemma}[theorem]{Lemma}

\begin{document}
	
	\title[Regularized Mean and Covariance Estimation]{Minimax Optimal Estimation of Mean and Covariance Functions with Spectral Regularization}
    \author[N. Gupta]{Naveen Gupta}
\address[N. Gupta]{Department of Statistics, Pennsylvania State University, USA}
\email{ngupta.maths@gmail.com}
\author[B. K. Sriperumbudur]{Bharath K. Sriperumbudur}
\address[B. K. Sriperumbudur]{Department of Statistics, Pennsylvania State University, USA}
\email{bks18@psu.edu}

\begin{abstract}
Estimation of the mean and covariance functions is a fundamental problem in functional data analysis, particularly for discretely observed functional data. In this work, we study a regularization-based framework for estimating the mean and the covariance functions within a reproducing kernel Hilbert space (RKHS) setting. Our approach utilizes a spectral regularization technique under H\"{o}lder-type source conditions, allowing for a broad class of regularization schemes and accommodating a wide range of smoothness assumptions on the target functions. In contrast to RKHS formulations that assume the target belongs to the underlying RKHS, our source-condition framework also accommodates misspecified targets. Convergence rates for the proposed estimators are derived, and we derive corresponding minimax lower bounds and identify regimes in which the upper bounds are optimal, or optimal up to logarithmic factors.
\end{abstract}

	\maketitle

\section{Introduction}
Functional data analysis (FDA) provides a framework for the statistical analysis of data whose observational units are functions and are therefore naturally viewed as infinite-dimensional objects \cite{ramsay1991some,ramsay2002afda,Theoreticalfoundationshsing2015,kokoszka2017}. Such data arise in a wide range of applications, including biomedical and longitudinal studies, environmental monitoring, finance, and data generated by modern sensing technologies. Two fundamental objects in FDA are the mean and covariance functions, which characterize, respectively, the first- and second-order structure of the underlying stochastic process. Their estimation is important not only for describing the process itself, but also because they serve as basic building blocks for many subsequent statistical procedures, including functional principal component analysis, functional regression, classification, clustering, and dimension reduction \cite{muller2005generalized,clustering_2003_fd, wang2016functional}.

Formally, let $X(\cdot)$ be a square integrable stochastic process defined on a compact domain $T \subset \mathbb{R}$ with $\{X_{i},1 \leq i \leq n\}$ being its i.i.d.~observations. For simplicity we assume that $T$ has Lebesgue measure 1. Then the mean and the covariance functions are defined as $\mu_0(t) = \mathbb{E}[X(t)],~\forall~t \in T$ and
$$C_{0}(s,t) = \mathbb{E}[(X(s)-\mu_{0}(s))(X(t)-\mu_{0}(t))],~\forall~s,t\in T,$$ respectively. When the trajectories are fully observed, one can simply estimate the mean function $\mu_{0}$ and the covariance function $C_{0}$ by sample mean $\hat{\mu}= \frac{1}{n}\sum_{i=1}^{n}X_{i}$ and sample covariance $\hat{C}(s,t) = \frac{1}{n}\sum_{i=1}^{n}(X_{i}(s)-\hat{\mu}(s))(X_{i}(t)-\hat{\mu}(t))$ respectively. Under suitable moment conditions, these estimators attain the usual parametric rate $n^{-1/2}$ in the corresponding $L^2$ norms. In practice, however, complete trajectories are rarely observed. Instead, each trajectory is typically recorded at only finitely many, possibly irregular or random, locations and is often contaminated by measurement error. Consequently, the sample mean and covariance based on the complete trajectories are not directly available, motivating the following discrete observation model.

Suppose that the independent realizations $X_1,\ldots,X_n$ of $X$ are observed at discrete locations and contaminated by measurement error according to
\begin{equation}\label{model}
Y_{ij}=X_i(t_{ij})+\epsilon_{ij},
\qquad 1\leq j\leq m_i,\quad 1\leq i\leq n.
\end{equation}
Two sampling schemes are commonly considered for the observation locations: a \emph{fixed design} \cite{cuevas2002,Medina_fixed_design}, in which the locations are deterministic and prespecified, and a \emph{random design} \cite{cai2010nonparametric, Wong&Zhand_2019_covariance}, in which they are randomly sampled. In this work, we adopt the random-design framework and assume that $\{t_{ij}:1\leq j\leq m_i,1\leq i\leq n\}$ are independently and uniformly distributed on $T$. We further assume that the measurement errors $\{\epsilon_{ij}\}$ are independent and identically distributed with mean zero and finite variance $\sigma_0^2$, and that the collections $\{X_i\}$, $\{t_{ij}\}$, and $\{\epsilon_{ij}\}$ are mutually independent.\\

\textbf{Existing works.}
For model~\eqref{model}, the existing literature on estimation of the mean and covariance functions has largely developed along two broad methodological directions. One prominent line of work relies on smoothing-based methods, in which the mean and covariance functions are estimated by applying nonparametric smoothing procedures to the pooled observations or their empirical cross-products \cite{yao2005sparsetodense,hall_2006_principle_compo,zhang_2007_statstical_inference,hsing_2010_non_parametric,Wang&zhang_2016_sparse_to_dense,wang&zhang_2018_longitudinal_and_functional}. Yao, Müller and Wang \cite{yao2005sparsetodense} developed smoothing-based methodology for sparse longitudinal functional data, while Hall, Müller and Wang studied the resulting estimation of functional principal components and identified sampling regimes in which discretization does not affect first-order rates. Li and Hsing \cite{hsing_2010_non_parametric} subsequently developed a unified local linear smoothing framework for both sparse and dense functional data and established uniform convergence rates for estimators of the mean and covariance functions. Their results allow the number of observations to vary across curves and demonstrate explicitly how the convergence rates depend jointly on the number of curves and the within-curve sampling frequencies; in particular, the rates reduce to the usual nonparametric regression rates in the sparse regime, whereas root-$n$ rates can be attained in sufficiently dense regimes. More recently, Zhang and Wang \cite{Wang&zhang_2016_sparse_to_dense,wang&zhang_2018_longitudinal_and_functional} further investigated the transition from sparse to dense functional data and the role of different weighting schemes in mean and covariance estimation.

When the numbers of observations vary across subjects, the choice of weighting scheme becomes particularly important. Two commonly used schemes are the observation-weighted (OBS) scheme, which assigns equal weight to each observation and consequently gives greater total weight to subjects with more observations, and the subject-weighted (SUBJ) scheme, which assigns equal total weight to each subject by normalizing its contribution according to its number of observations. Zhang and Wang \cite{Wang&zhang_2016_sparse_to_dense} studied these schemes within a general weighting framework and showed that their relative performance depends on the sampling regime and the heterogeneity of the numbers of observations across subjects. In particular, the OBS scheme is preferable in certain non-dense regimes, whereas the SUBJ scheme can be more efficient in sufficiently dense regimes. They also introduced weighting schemes obtained by combining the OBS and SUBJ weights. Subsequently, Zhang and Wang \cite{wang&zhang_2018_longitudinal_and_functional} derived optimal weighting schemes for mean and covariance estimation by directly minimizing the corresponding $L^2$ convergence rates, yielding procedures that adapt to the underlying sampling design.

Although these smoothing-based procedures have demonstrated strong theoretical and empirical performance, they are typically constructed through local smoothing and aggregation rather than as solutions to a single global optimization problem. An alternative line of work formulates mean and covariance estimation as regularized optimization problems in reproducing kernel Hilbert spaces (RKHSs), providing a global variational framework for estimation. For mean estimation, Cai and Yuan \cite{Cai_2011_mean} considered a smoothing-spline estimator over a Sobolev RKHS and established minimax-optimal convergence rates under both common and independent designs. Their analysis revealed a phase transition in the sampling frequency, beyond which the optimal rate coincides with that attainable from fully observed trajectories. For covariance estimation, Cai and Yuan \cite{cai2010nonparametric} developed a regularization method over a tensor-product RKHS that accounts for the intrinsically bivariate structure of the covariance function. Despite this bivariate structure, they showed that the resulting estimator can attain convergence rates, up to logarithmic factors, comparable to those arising in univariate nonparametric function estimation, and identified a corresponding phase transition between sparse and dense sampling regimes. Both analyses impose smoothness conditions through the assumption that the sample paths of the underlying process belong to an appropriate Sobolev space. Moreover, although the covariance estimator of \cite{cai2010nonparametric} is formulated over a tensor-product RKHS, positive semidefiniteness is not explicitly enforced, and hence the resulting estimator need not itself be a valid covariance function.

More recently, Wong and Zhang \cite{Wong&Zhand_2019_covariance} developed an operator-regularized framework for covariance estimation in tensor-product RKHSs. In contrast to earlier unconstrained formulations, their estimator is obtained by optimizing over the cone of symmetric positive semidefinite operators, thereby guaranteeing that the resulting estimator defines a valid covariance function. Their framework accommodates a broad class of operator regularization penalties, including the Hilbert--Schmidt and trace-norm penalties, with the latter additionally encouraging low-rank covariance estimation. They established a representer theorem that reduces the infinite-dimensional optimization problem to a finite-dimensional one and derived convergence rates under both fixed and random sampling designs. This operator-theoretic framework was subsequently extended by Wang et al. \cite{wang_2022_low_rank} to covariance estimation for multidimensional functional data, incorporating multilinear low-rank structures to account for the additional structural complexity arising from multidimensional functional domains.

As in smoothing-based methods, an important distinction among the existing RKHS-based approaches concerns the treatment of heterogeneous sampling frequencies across subjects. For mean estimation, Cai and Yuan \cite{Cai_2011_mean} allow the numbers of observations $m_i$ to vary across subjects and employ SUBJ weighting scheme. For covariance estimation, Wong and Zhang \cite{Wong&Zhand_2019_covariance} derive their theoretical results under a balanced design with $m_1=\cdots=m_n=m$, for which the distinction between observation- and subject-level weighting disappears. In contrast, we allow the sampling frequencies $m_i$ to be unequal and formulate both mean and covariance estimation under general weighting schemes. This framework encompasses both OBS and SUBJ weighting schemes as special cases and enables us to characterize how the choice of weighting interacts with the heterogeneity of the sampling frequencies and the sparse-to-dense sampling regime.

Although there is a substantial literature on mean and covariance estimation from discretely observed functional data, several important issues remain unresolved, particularly within the regularization-based framework. Existing RKHS-based theories are primarily developed in Sobolev-type spaces and impose smoothness assumptions at the level of the sample paths of the underlying stochastic process. For example, Cai and Yuan \cite{Cai_2011_mean,cai2010nonparametric} assume that the sample paths belong almost surely to an appropriate Sobolev space. Such assumptions can be restrictive, since regularity of the population mean and covariance functions does not in general require individual sample paths to possess the same degree of smoothness. In contrast, it is desirable to formulate regularity directly in terms of the population quantities being estimated, thereby allowing the stochastic process itself to belong to a substantially broader class.

Furthermore, a unified treatment through the spectral-filter framework of statistical inverse problems has not been developed. Where Cai and Yuan \cite{cai2010nonparametric, Cai_2011_mean} focus on Tikhonov regularization, Wong et al. \cite{Wong&Zhand_2019_covariance} have considered a more general penalty term. Working with general regularization allows a broad class of regularization methods to be studied simultaneously and, through Hölder source conditions with unrestricted exponent, naturally accommodates both well-specified and misspecified settings in which the target function need not belong to the chosen RKHS. For covariance estimation, Cai and Yuan \cite{cai2010nonparametric} focus on a particular penalized least-squares regularization scheme and develop their theoretical analysis under specific Sobolev smoothness assumptions, with the underlying process assumed to have zero mean. Moreover, positive semidefiniteness is not imposed in their optimization problem, so the resulting estimator is not guaranteed to define a valid covariance function. These considerations motivate the framework developed in the present work, in which mean and covariance estimation are analyzed through general spectral regularization under source conditions imposed directly on the respective population quantities.

There are two further issues that are particularly relevant for covariance estimation from discretely observed functional data. First, when the mean function is unknown and estimated from the same observations used for covariance estimation, the dependence between the mean estimator and the covariance estimating equation introduces additional stochastic terms that complicate the theoretical analysis. To circumvent this difficulty, we employ a sample-splitting strategy: one subsample is used to estimate the mean function, while an independent subsample is used for covariance estimation. The resulting independence allows the effect of mean estimation on covariance estimation to be isolated and controlled explicitly. Sample splitting has been widely used as a general device for decoupling dependent estimated quantities in statistical inference and has also appeared in recent work on discretely observed functional data \cite{Wang&zhang_2016_sparse_to_dense}. For theoretical convenience, we employ sample splitting, using independent subsamples for mean and covariance estimation. Since the two subsamples are taken to have comparable sizes, we use $n$ to denote the sample size of each subsample for notational simplicity.

Second, existing analyses have either been restricted to balanced designs or, in the case of unbalanced designs with varying $m_i$'s, have focused exclusively on the SUBJ weighting scheme. Rather than restricting the analysis to a particular weighting scheme, we formulate our estimators using a general class of subject-specific weights. This formulation encompasses both SUBJ and OBS as special cases and allows us to characterize explicitly how the convergence rates depend on the sampling frequencies $\{m_i\}_{i=1}^n$ and the choice of weights.

Our analysis also revisits the minimax theory for mean estimation under independent random sampling. Cai and Yuan \cite{Cai_2011_mean} established minimax rates over a class of stochastic processes whose sample paths satisfy prescribed Sobolev smoothness conditions. We derive minimax lower bounds without imposing a balanced sampling design, thereby allowing the numbers of observations $m_i$ to vary across subjects. In the balanced case $m_1=\cdots=m_n=m$, our lower bounds recover those of Cai and Yuan \cite{Cai_2011_mean}. Under heterogeneous sampling, however, the lower bounds depend on the arithmetic mean of the sampling frequencies, whereas the convergence rates obtained under the SUBJ weighting scheme depend on their harmonic mean. Consequently, although the two quantities coincide under balanced sampling, they can differ substantially when the $m_i$'s are heterogeneous, showing that SUBJ weighting need not be rate-optimal in such settings. 

In this work, we develop a unified RKHS-based spectral-regularization framework for estimating mean and covariance functions from discretely and noisily observed functional data. The framework accommodates heterogeneous sampling frequencies, general subject-specific weighting schemes, and both well-specified and misspecified regimes, while imposing regularity conditions directly on the population targets rather than on individual sample paths. We establish upper and minimax lower bounds for mean and covariance estimation, characterize sparse-to-dense transitions, and quantify the effects of sampling heterogeneity and weighting on the resulting rates. For covariance estimation, we further account for estimation of an unknown mean function and show that positive semidefiniteness can be enforced without deteriorating the convergence rate. Our main contributions are summarized below.

\subsection{Contributions}
\emph{(i)} We develop a unified spectral-regularization framework for mean and covariance function estimation from discretely and noisily observed functional data. Our framework accommodates a broad class of spectral regularization methods, including Tikhonov regularization, spectral cut-off, Landweber iteration, and Showalter's method. The regularity of the population targets is characterized through Hölder-type source conditions, with no a priori restriction on the corresponding smoothness exponent. This allows us to study a wide range of target regularities within a common framework and to cover both well-specified and misspecified settings, i.e., when the target function does not belong to the pre-specified RKHS.\vspace{1mm}\\

\emph{(ii)} Our assumptions are imposed directly on the population quantities of interest rather than on the individual sample paths of the underlying stochastic process. In particular, unlike approaches that require the sample paths to belong almost surely to a prescribed Sobolev or RKHS class \cite{cai2010nonparametric}, our analysis only requires suitable moment conditions on the process together with Hölder source conditions on the mean and covariance functions. This permits stochastic processes whose sample paths need not belong to the prescribed RKHS.\vspace{1mm}\\
\emph{(iii)} We allow the numbers of observations $m_i$ to vary across subjects and formulate both mean and covariance estimators using general subject-specific weights. The resulting framework includes OBS and SUBJ schemes as special cases. For mean estimation, we establish upper and minimax lower bounds that explicitly characterize the effects of the sampling frequencies and weighting scheme; see Theorems~\ref{thm:mean} and~\ref{thm:mean-lower}. Our lower bounds recover the classical minimax rates in the balanced case $m_1=\cdots=m_n=m$, while revealing a different phenomenon under heterogeneous sampling: the fundamental lower-bound scale is governed by the arithmetic mean of the sampling frequencies, whereas the SUBJ upper bound is governed by their harmonic mean. Consequently, SUBJ weighting need not be rate-optimal when these two effective sampling frequencies differ substantially.\vspace{1mm}\\

\emph{(iv)} For covariance estimation, we do not assume that the mean function is known or identically zero. Instead, we estimate the mean function and incorporate it into the covariance estimator in \eqref{covariance_variational}. To decouple the mean-estimation error from the covariance estimating equation, we employ sample splitting, using independent subsamples for mean and covariance estimation. We explicitly quantify the contribution of mean estimation to the covariance error and establish convergence rates in both the well-specified and misspecified regimes; see Theorem~\ref{thm:covariance}. '
We also derive lower bounds for covariance estimation in Theorem~\ref{thm:cov-lower-1} and Theorem~\ref{thm:lower_covariance} to provide comparison of our achieved rates.\vspace{1mm}\\

\emph{(v)} Since a covariance function must induce a positive semidefinite operator, we further construct a positive-semidefinite version of the proposed covariance estimator. We show that projecting the unconstrained spectral estimator onto the positive-semidefinite cone does not deteriorate its $L^2$ convergence rate. This provides a theoretically justified way of enforcing the defining structural constraint of a covariance function while retaining the statistical guarantees of the original estimator.\vspace{1mm}\\

\emph{(vi)} Finally, our results characterize the transition between sparse and dense functional data regimes. When the within-curve sampling frequencies are sufficiently large, the discretization error becomes asymptotically negligible and the convergence rates are determined solely by the number of independent curves $n$, recovering the parametric rate $n^{-1/2}$. In the sparse regime, the rates depend jointly on the number of curves, the within-curve sampling frequencies, the regularity of the target function, and the weighting scheme. Our results therefore provide a unified description of how regularization, smoothness, sampling density, and sampling heterogeneity interact in mean and covariance estimation.\vspace{1mm}\\

\emph{(vii)} Our analysis reveals a finer phase transition in covariance estimation by distinguishing between very sparse, sparse, and dense sampling regimes. In particular, under the SUBJ weighting scheme, the rate in the very sparse regime coincides in order with the existing rate of Cai and Yuan~\cite{cai2010nonparametric} when the number of observations per subject remains bounded, while improving upon it when the sampling frequency grows with the sample size. In the sparse regime, our estimator attains the minimax-optimal rate without the additional logarithmic factor appearing in the existing result. Moreover, the transition to the parametric rate $n^{-1/2}$ occurs at a lower sampling frequency, demonstrating that the proposed estimator can achieve the fully observed-data rate under a less stringent sampling requirement. See Remark~\ref{cov:subj_comparision} for a detailed comparison with Cai and Yuan~\cite{cai2010nonparametric}

\subsection{Organization}
The mean function and covariance function estimators, their upper and lower convergence analysis are presented in Sections~\ref{Sec:mean} and \ref{sec:cov}. The missing proofs are captured in Section~\ref{sec:proof}. Supplementary results related to the mean function and covariance function estimation are collected in Appendices~\ref{app:mean} and \ref{app:cov}, while some technical results needed to prove all these main results are collected in Appendix~\ref{app:tech}.

\subsection{Definitions and Notation}
A Hilbert space $(\mathcal{H}, \langle \cdot, \cdot \rangle_{\mathcal{H}})$ of functions from $D$ to $\mathbb{R}$ is called an RKHS if the pointwise evaluation functional $\delta_{x} : \mathcal{H} \to \mathbb{R}~(f \mapsto f(x))$ is continuous for each point $x \in D$. It is well known that there is a one-to-one correspondence between an RKHS and a reproducing kernel. Let $k$ be the reproducing kernel (r.k.) associated with the RKHS $\mathcal{H}$. Then, we have that for each $f \in \mathcal{H}$ and $x \in D$,~ $f(x)= \langle f, k(x, \cdot) \rangle_{\mathcal{H}}$. Throughout this paper, we assume that all r.k.'s are continuous and bounded, i.e., 
$\sup_{x \in T}k(x,x) \leq \kappa^2 < \infty$ for some constant $\kappa$. For more details on RKHS, see \cite{paulsen2016rkhs}. 

For non-negative sequences $(a_n)_n$ and $(b_n)_n$, we say $a_n\lesssim b_n$ for all $n$, if there exists a universal constant $c$ not depending on $n$ such that $a_n\le cb_n$ for all $n$. For a random variable $\chi\in \mathbb{R}$
with distribution $P$ and a constant $e$, we use $\chi 
\lesssim_p e$ to denote the fact that for any $\delta > 0$, there
exists a positive constant $K_\delta<\infty$ such that $P(\chi \le K_\delta e) \ge 1-\delta$. $\Vert A\Vert_{\text{op}}$ denotes the operator norm of an operator $A$. For any positive integer $m$, $[m]$ denotes the set $\{1,2,\ldots,m\}.$ $\mathcal{R}(A)$ denotes the range space of operator $A$. We denote $m_{H} = n\left(\sum_{i=1}^{n}\frac{1}{m_{i}}\right)^{-1}$ as the harmonic mean and $m_{A}  = \frac{1}{n}\sum_{i=1}^{n}m_{i}$ as the arithmetic mean. We denote $[n]:=\{1,\ldots,n\}.$ For an integer $m > r$ we denote $m^{\underline{r}} = \sum_{j=1}^{r-1}m(m-1)\cdots (m-(r-1))$.

\section{Mean Estimation} \label{Sec:mean}
We begin with the estimation of the mean function $\mu_0(t)=\mathbb{E}\{X(t)\}$. Besides being a fundamental problem in its own, mean estimation serves as an essential building block for the covariance estimation procedure developed subsequently. We consider the setting in which the sample trajectories are observed only through noisy measurements at subject-specific design points, with the number of observations allowed to vary across subjects. This formulation accommodates both sparsely and densely sampled functional data within a unified framework. We construct the mean estimator using spectral regularization in a reproducing kernel Hilbert space and derive its convergence properties.

Let \(\mathcal{H}_{1}\) be an RKHS over \(T\) with reproducing kernel \(k:T\times T\to\mathbb{R}\). Under the observation model \eqref{model}, conditional on the design points \((t_{ij})\), we have, for every \(i\in[n]\) and \(j\in[m_i]\),
$$
\mathbb{E}\!\left[Y_{ij}\mid t_{ij}\right]
=
\mathbb{E}\!\left[X_i(t_{ij})\mid t_{ij}\right]
+
\mathbb{E}\!\left[\epsilon_{ij}\mid t_{ij}\right]
=
\mu_0(t_{ij}).
$$
This motivates the following RKHS-based estimator of the mean function:
\begin{align}
\hat{\mu} := \arg\min_{\mu\in\mathcal{H}_{1}}
\frac{1}{n}\sum_{i=1}^{n}
w_i\sum_{j=1}^{m_i}
[Y_{ij}-\mu(t_{ij})]^{2},
\label{Eq:mean-est}
\end{align}
where the nonnegative weights \(w_1,\ldots,w_n\) are normalized so that
$$
\frac{1}{n}\sum_{i=1}^{n}w_i m_i=1.
$$
This general weighting scheme encompasses, in particular, subject-level (SUBJ) weighting, corresponding to \(w_i=m_i^{-1}\), and observation-level (OBS) weighting, corresponding to \(w_i=m_A^{-1}\), where \(m_A\) denotes the arithmetic mean of \(m_1,\ldots,m_n\). We next derive a convenient operator representation of \(\hat{\mu}\), which will also facilitate the subsequent theoretical analysis.

By the reproducing property of $\mathcal{H}_1$, the estimator in \eqref{Eq:mean-est} can be equivalently expressed as
\begin{align}
\hat{\mu} &=
\arg\min_{\mu\in\mathcal{H}_1}
\frac{1}{n}\sum_{i=1}^{n}w_i\sum_{j=1}^{m_i}
[Y_{ij} -\left\langle \mu,k(t_{ij},\cdot)\right\rangle_{\mathcal{H}_1}
]^{2}\\
&=
\arg\min_{\mu\in\mathcal{H}_1}
\frac{1}{n}\sum_{i=1}^{n}
\left\|\overline{Y}_i-S_i\mu\right\|_{\mathbb{R}^{m_i}}^{2},
\label{Eq:mean-est-1}
\end{align}
where the sampling operator \(S_i:\mathcal{H}_1\to\mathbb{R}^{m_i}\) is defined by

$$
S_i f
=
\left\{
\sqrt{w_i}\,
\left\langle f,k(t_{ij},\cdot)\right\rangle_{\mathcal{H}_1}
\right\}_{j\in[m_i]},
$$
and
$\overline{Y}_i
=
\left\{\sqrt{w_i}Y_{ij}\right\}_{j\in[m_i]}.
$ The adjoint \(S_i^*:\mathbb{R}^{m_i}\to\mathcal{H}_1\) is given by
$$
S_i^*\boldsymbol{\alpha}
=
\sum_{j=1}^{m_i}
\sqrt{w_i}\,\alpha_j k(\cdot,t_{ij}),
\qquad
\boldsymbol{\alpha}=(\alpha_1,\ldots,\alpha_{m_i})^\top.
$$
Consequently, any minimizer of \eqref{Eq:mean-est-1} satisfies the normal equation
$$
\left(\frac{1}{n}\sum_{i=1}^{n}S_i^*S_i\right)\hat{\mu}
=
\frac{1}{n}\sum_{i=1}^{n}S_i^*\overline{Y}_i.
$$
For notational convenience, we define
$$
\hat S_n
:=
\frac{1}{n}\sum_{i=1}^{n}S_i^*S_i,
\qquad
V
:=
\frac{1}{n}\sum_{i=1}^{n}S_i^*\overline{Y}_i.
$$
The empirical operator \(\hat S_n:\mathcal{H}_1\to\mathcal{H}_1\) is positive and self-adjoint, but is generally not invertible. Direct inversion of the normal equation is therefore ill-posed. To obtain a stable solution, we replace the generalized inverse of \(\hat S_n\) by a spectral regularization and define
\begin{equation}
\hat{\mu}_{\lambda}
:=
g_{\lambda}(\hat S_n)V,
\label{mean_estimator}
\end{equation}

where \(g_{\lambda}:[0,a]\to\mathbb{R}\), \(0<\lambda\leq a\), denotes a family of regularization functions, satisfying the following conditions. Let
$$
r_{\lambda}(\sigma):=1-\sigma g_{\lambda}(\sigma), \qquad \sigma\in[0,a],
$$
denote the associated residual function. We assume that there exist constants \(a_1,a_2,a_3>0\), independent of \(\lambda\), such that
\begin{align}
\sup_{0<\sigma\leq a}\left|\sigma g_{\lambda}(\sigma)\right|
&\leq a_1, \label{regularization_property_1}~~~~\
\sup_{0<\sigma\leq a}\left|g_{\lambda}(\sigma)\right| \leq \frac{a_2}{\lambda},~~~~\
\sup_{0<\sigma\leq a}\left|r_{\lambda}(\sigma)\right| \leq a_3.
\end{align}
The qualification of the regularization family \(g_{\lambda}\) is the largest \(p>0\) for which there exists a constant \(\omega_p>0\), independent of \(\lambda\), such that
\begin{equation}
\sup_{0<\sigma\leq a}
\left|r_{\lambda}(\sigma)\right|\sigma^p
\leq
\omega_p\lambda^p.
\label{qualification_property}
\end{equation}
In this sense, \(g_{\lambda}\) acts as a regularized approximation of the reciprocal map \(\sigma\mapsto\sigma^{-1}\), while \(r_{\lambda}\) measures the corresponding approximation error. The qualification characterizes the extent to which this error can be controlled as \(\lambda\to0\), and consequently determines the degree of smoothness that can be exploited by the regularization scheme.

Prominent examples include Tikhonov regularization,
$
g_{\lambda}(\sigma)=(\sigma+\lambda)^{-1},
$
spectral cut-off,
$
g_{\lambda}(\sigma)=\sigma^{-1}\mathds{1}_{\{\sigma\geq\lambda\}},
$
and Showalter regularization,
$
g_{\lambda}(\sigma)
=
\sigma^{-1}\bigl(1-e^{-\sigma/\lambda}\bigr)\mathds{1}_{\{\sigma>0\}}
+\lambda^{-1}\mathds{1}_{\{\sigma=0\}},
$
as well as Landweber iteration, for which the regularization parameter is identified with the reciprocal of the iteration index. Tikhonov regularization has qualification \(1\), whereas spectral cut-off, Showalter regularization, and Landweber iteration have infinite qualification. We refer to \cite{sergei2013,englmartinbook} for a comprehensive treatment of these and related spectral regularization methods.

Let $A:\mathcal H\to\mathcal H$ be a compact, self-adjoint, non-negative operator with spectral representation
$$
A=\sum_{j\geq1}\tau_j(e_j\otimes e_j),
$$
where $\{\tau_j\}_{j\geq1}$ are the strictly positive eigenvalues of $A$ and $\{e_j\}_{j\geq1}$ are the corresponding orthonormal eigenfunctions. For a spectral regularizer $g_\lambda:[0,a]\to\mathbb R,~ a > \|A\|_{op}$, we define
$$
g_\lambda(A)
:=
\sum_{j\geq1}g_\lambda(\tau_j)(e_j\otimes e_j)
+
g_\lambda(0)
\left(
I-\sum_{j\geq1}e_j\otimes e_j
\right).
$$
Equivalently, the second term is $g_\lambda(0)P_{\ker(A)}$, where $P_{\ker(A)}$ denotes the orthogonal projection onto $\ker(A)$. When $A$ is injective, $\ker(A)={0}$ and the second term vanishes.

Notice that $g_\lambda(A)$ is well-defined through the above functional calculus only when $A$ is a non-negative self-adjoint operator, since $g_\lambda$ is defined on $[0,\infty)$. In particular, this notation is not applicable to a self-adjoint operator having negative eigenvalues.\\

\noindent
The following result provides an alternative representation of $\hat{\mu}_{\lambda}$ that will be instrumental in establishing its convergence properties. Before stating the result, we introduce some notation that will be used throughout the paper.

Let $ J_1:\mathcal{H}_1\to L^2(T),\qquad f\mapsto f,$ denote the canonical inclusion operator, and let $J_1^*:L^2(T)\to\mathcal{H}_1$ be its adjoint. By the reproducing property,

$$ J_1^*f = \int_T k(s,\cdot)f(s)\,ds,
\qquad f\in L^2(T).$$

\noindent
We define the integral operator
$\Lambda_1:=J_1J_1^*:L^2(T)\to L^2(T),$
so that
$$ (\Lambda_1f)(\cdot)=
\int_T k(s,\cdot)f(s)\,ds, \qquad f\in L^2(T).
$$

\noindent
Let \(\{(\lambda_\ell,\psi_\ell)\}_{\ell\in\mathbb N}\) denote the eigenvalue--eigenfunction pairs of \(\Lambda_1\), and define

$$
k^{1/2}(t,\cdot)
:=
\sum_{\ell\geq1}
\lambda_\ell^{1/2}\psi_\ell(t)\psi_\ell(\cdot),
\qquad t\in T.
$$

\noindent
The following proposition transfers the representation of \(\hat{\mu}_\lambda\) from \(\mathcal H_1\) to \(L^2(T)\), a formulation that will be convenient for the subsequent analysis.

\begin{prop}\label{mean_alternate_estimator}
Suppose that \(k:T\times T\to\mathbb R\) is bounded and continuous. Then

$$ \hat{\mu}_\lambda = g_\lambda(\hat S_n)V = \Lambda_1^{1/2}g_\lambda(\hat A_n)V_1, $$

where \(\hat A_n:L^2(T)\to L^2(T)\) is the empirical operator defined by

$$ \hat A_n f = \frac{1}{n}\sum_{i=1}^n w_i\sum_{j=1}^{m_i} \left\langle f,k^{1/2}(t_{ij},\cdot) \right\rangle_{L^2(T)} k^{1/2}(t_{ij},\cdot), \qquad f\in L^2(T), $$

and
$$ V_1 = \frac{1}{n}\sum_{i=1}^n w_i\sum_{j=1}^{m_i} Y_{ij}k^{1/2}(t_{ij},\cdot).
$$
\end{prop}

\noindent
The proof of Proposition~\ref{mean_alternate_estimator} is deferred to Section~\ref{subsec:prop}.\\

While the representation in \eqref{mean_estimator} is convenient for numerical implementation, the alternative representation in Proposition~\ref{mean_alternate_estimator} is better suited to the theoretical analysis. We now state the assumptions used to establish the convergence rate of \(\hat{\mu}_{\lambda}\); their roles and implications are discussed in Remark~\ref{rmk_ass}.

\begin{assumption}\label{mean_SRC}
    $\mu_{0} \in \mathcal{F}_{\alpha, R} = \{f = \Lambda_{1}^{\alpha}h,~\|h\|_{L^2(T)} \leq R\},~ \alpha >0$, i.e., there exists an $h \in L^2(T)$ with $\|h\|_{L^2(T)} \leq R$ such that $\mu_{0} = \Lambda_{1}^{\alpha}h$.
\end{assumption}
\begin{assumption}\label{mean_moment}
    $\sup_{s \in T}\mathbb{E}[X^2(s)] < \infty$, $\mathbb{E}[X^4(s)] \leq C (\mathbb{E}[X^2(s)])^2$ for all $s \in T$ and 
    \begin{equation*}
        \mathbb{E}\left[\langle X,f \rangle^4_{L^2(T)}\right] \leq C_{1}\left(\mathbb{E}\left[\langle X,f \rangle^2_{L^2(T)}\right]\right)^2,~ \forall~ f \in L^2(T),
    \end{equation*}
    where $C$ and $C_1$ are some positive universal constants. 
\end{assumption}
\begin{assumption}\label{mean_eigenvalue_decay}
    There exists constants $b >1$ such that
    \begin{equation*}
         i^{-b} \lesssim \lambda_{i} \lesssim i^{-b},\,\,\forall\,i \in \mathbb{N},
    \end{equation*}
    where $\{\lambda_i\}_{i \in \mathbb{N}}$ are the eigenvalues of $\Lambda_1$.
\end{assumption}

\noindent
For $\alpha \geq 0$, let us define the $\alpha$-power space as
\begin{equation*}
    [\mathcal{H}_{1}]^{\alpha} := \left\{\sum_{i}a_{i}\lambda_{i}^{\alpha}\psi_{i}~:~(a_{i})_{i \in \mathbb{N}} \in \ell^2(\mathbb{N})\right\} \subseteq L^2(T) 
\end{equation*}
equipped with the $\alpha$-power norm
\begin{equation*}
\left\|\sum_{i}a_{i}\lambda_{i}^{\alpha}\psi_{i}\right\|^2_{[\mathcal{H}_{1}]^{\alpha}} = \|(a_{i})_{i \in \mathbb{N}}\|^2_{\ell^2(\mathbb{N})} = \sum_{i \in \mathbb{N}}a_{i}^2.
\end{equation*}
Moreover, $\{\lambda_{i}^{\alpha}\psi_{i}\}_{i \in \mathbb{N}}$ forms an ONB for $[\mathcal{H}_{1}]^{\alpha}$ and consequently $[\mathcal{H}_{1}]^{\alpha}$ is a separable Hilbert space \cite{Steinwart_2020_sobolev_norm}.  
\begin{assumption}\label{mean_embedding}
For $ 0 < \alpha \leq \frac{1}{2}$, there exists a constant $Z>0$ such that
$$\|[\mathcal{H}_{{1}}]^{\alpha} \hookrightarrow L_{\infty}(T)\|_{\emph{op}} \leq Z.$$
\end{assumption}
\begin{remark}\label{rmk_ass}
(i) Assumption~\ref{mean_SRC} deals with the smoothness of the unknown mean function, which is the standard source condition in RKHS methods \cite{devito}. While in \cite{Cai_2011_mean} the authors assume that the mean function lies in the RKHS $(\alpha = \frac{1}{2})$ as their source condition, our analysis does not impose such a restriction on the choice of $\alpha$. For $0<\alpha<\frac{1}{2}$, the condition implies that $\mu_0$ lies in a real interpolation space, $[\mathcal{H}_1]^\alpha$ that interpolates $\mathcal{H}_1$ and $L^2(T)$, i.e., $\mu_0$ may lie outside the RKHS, and for $\alpha>\frac{1}{2}$, $\mu_0$ lies in smoother subspaces of $\mathcal{H}_1$ where the smoothness increases with $\alpha$.
  \vspace{1mm}\\
(ii) Assumption~\ref{mean_moment} imposes uniform moment conditions on the underlying stochastic process $X$. The boundedness of $\sup_{s \in T}\mathbb{E}[X^2(s)]$ ensures that the pointwise variance of the process is uniformly controlled over the domain $T$, while the fourth-moment condition $\mathbb{E}[X^4(s)] \leq C (\mathbb{E}[X^2(s)])^2$ rules out excessively heavy-tailed behavior at each location. In addition, the moment bound on the random linear functionals $\langle X,f\rangle_{L^2(T)}$ guarantees that fourth moments of integrated projections of the process are uniformly controlled by their second moments for all $f \in L^2(T)$. All these conditions have been consistently used in the FDA literature \cite{cai2010nonparametric} and play a crucial role in the analysis to achieve optimal rates. 
\vspace{1mm}\\

(iii) Assumption \ref{mean_eigenvalue_decay} is also quite standard in the literature on RKHS-based algorithms \cite{devito}. While Assumption~\ref{mean_SRC} captures the smoothness of the target function, i.e., $\mu_0$, Assumption \ref{mean_eigenvalue_decay} controls the spectral complexity of the RKHS through the decay of the eigenvalues of its integral operator. It follows from Assumption \ref{mean_eigenvalue_decay} that 
\begin{equation*}\label{mean_effective_dimenstion}
    \begin{split}
        \mathcal{N}_{1}(\lambda) :=  \emph{trace}((\Lambda_{1}+\lambda I)^{-1}\Lambda_{1})
        = \sum_{i}\frac{\lambda_{i}}{\lambda_{i}+\lambda} = \sum_{i}\frac{1}{1+\frac{\lambda}{\lambda_{i}}}
        \lesssim  \sum_{i}\frac{i^{-b}}{i^{-b}+\lambda} \lesssim \lambda^{-\frac{1}{b}},
    \end{split}
\end{equation*}
where we used $\lambda_{i} \lesssim i^{-b}$ in the penultimate step. This quantity captures the intrinsic dimensionality of the RKHS, which grows to infinity as $\lambda\rightarrow 0$. On the other hand, if the RKHS is finite dimensional, say $d$-dimensional, then only $d$ eigenvalues are positive, which implies $\mathcal{N}_{1}(\lambda)\le d$.\vspace{1mm}\\
(iv) Assumption~\ref{mean_embedding} provides a continuous embedding of the interpolation space $[\mathcal{H}_{1}]^{\alpha}$ into $L_{\infty}(T)$, which has been used as a regularity condition in kernel-based nonparametric estimation \cite{Steinwart_2020_sobolev_norm}. Such embeddings hold for a broad class of commonly used kernels, including Sobolev, spline, Mat\'{e}rn, and Gaussian kernels on compact domains. \cite[Theorem~9]{Steinwart_2020_sobolev_norm} shows that for $\alpha > 0$,
$$\|[\mathcal{H}_{1}]^{\alpha} \hookrightarrow L_{\infty}(T)\|_{\emph{op}} = \|k^{\alpha}\|_{\infty},$$
where $k^{\alpha}(x,x') = \sum_{i}\lambda_{i}^{2 \alpha} \psi_{i}(x) \psi_{i}(x')$ is the reproducing kernel of the $\alpha$-power RKHS, $[\mathcal{H}_{1}]^{\alpha}$. Therefore, $\Vert k^\alpha\Vert_\infty=\sup_{x\in T}\sum_i\lambda^{2\alpha}_i\psi^2_i(x)$, which means Assumption~\ref{mean_embedding} implies that there exists a constant $Z$ such that $\sum_{i}\lambda_{i}^{2 \alpha}\psi_{i}^2(x) \leq Z^2$ for almost all $x \in T$. 
\vspace{1mm}\\
(v) Cai and Yuan \cite{Cai_2011_mean} require the sample paths of the underlying stochastic process to lie in a specified RKHS (a Sobolev space in their setting), which in turn forces the mean function to belong to the same RKHS. We relax this requirement by instead imposing a regularity condition directly on the mean function alongside the moment condition $\sup_{t \in T}\mathbb{E}[X^2(t)] < \infty$. When the kernel under consideration is a Mercer kernel, the assumptions of \cite{Cai_2011_mean} are strictly stronger than ours. To see this, consider the process $X(t) = t + W(t)$ for $t \in [0,1]$, where $W(t)$ is a standard Brownian motion, paired with the Sobolev RKHS $\mathcal{W}_{2}^{1}[0,1] = \{g : [0,1] \to \mathbb{R} \mid g \text{ is absolutely continuous and } g^{(1)} \in L^2[0,1]\}$. The sample paths of $X$ are almost surely nowhere differentiable and therefore do not belong to $\mathcal{W}_2^1[0,1]$, yet the process satisfies $\sup_{t \in [0,1]}\mathbb{E}[X^2(t)] < \infty$, demonstrating that our conditions are strictly weaker. Another example is $X(t) = t+ \sqrt{t}Z,~ t \in [0,1]$ where $Z \sim \mathcal{N}(0,1)$. It can be easily seen that $\mathbb{E}[X^2(t)] = \mathbb{E}[(t+ \sqrt{t}Z)^2] = \mathbb{E}[t^2+ t Z^2 + 2 t^{3/2}Z] = t^2 + t$, $\sup_{t \in [0,1]} \mathbb{E}[X^2(t)]< \infty$ and $m(t)=\mathbb{E}[X(t)]=t\in \mathcal{W}^1_2[0,1]$. But $X'(t) = 1+ \frac{Z}{2 \sqrt{t}} \not\in L^2[0,1]$ a.s. implies that its sample path does not lie in the RKHS $\mathcal{W}_{2}^{1}[0,1]$. Note that the examples listed above satisfy all the conditions stated in Assumption~\ref{mean_moment}.
\end{remark}

The following result (proved in Section~\ref{subsec:thm-mean-upper}) provides the convergence rate for the proposed mean function estimator.
\begin{theorem}\label{thm:mean}
    Suppose Assumptions~\ref{mean_SRC}, \ref{mean_moment} and \ref{mean_eigenvalue_decay} hold. Let $\nu \geq 1$ be the qualification of the regularization family. Then, for $\alpha \geq \frac{1}{2}$, we have 
    \begin{equation*}
        \|\hat{\mu}_{\lambda}-\mu_{0}\|_{L^2(T)} \lesssim_{p} \left(\frac{1}{n^2}\sum_{i=1}^{n}w_{i}^2 m_{i}(m_{i}-1)\right)^{\frac{1}{2}}+\left(\frac{1}{n^2}\sum_{i=1}^{n}w_{i}^2 m_{i}\right)^{\frac{1}{2}}\lambda^{-\frac{1}{2b}}+ \lambda^{r},~~r = \min\{\alpha, \nu\}.
    \end{equation*}
Furthermore, if Assumption~\ref{mean_embedding} also holds, then for $0 < \alpha \leq\frac{1}{2}$, we have
\begin{equation*}
        \|\hat{\mu}_{\lambda}-\mu_{0}\|_{L^2(T)} \lesssim_{p} \left(\frac{1}{n^2}\sum_{i=1}^{n}w_{i}^2 m_{i}(m_{i}-1)\right)^{\frac{1}{2}}+\left(\frac{1}{n^2}\sum_{i=1}^{n}w_{i}^2 m_{i}\right)^{\frac{1}{2}}\lambda^{-\frac{1}{2b}} + \lambda^{\alpha}.
    \end{equation*}
    \end{theorem}
\begin{remark}\label{rem:saturate}
Cai and Yuan \cite{Cai_2011_mean} study mean estimation under the assumption that the mean function belongs to a Sobolev RKHS. Theorem~\ref{thm:mean} extends this framework in several directions.

\begin{itemize}
\item[(i)] Our analysis is formulated for a general reproducing kernel $k$ and is therefore not restricted to Sobolev kernels.

\item[(ii)] We characterize the regularity of the mean function through the source condition
    $$\mu_0\in\mathcal{R}(\Lambda_1^\alpha), \qquad \alpha>0.$$
Since
$$\mathcal{R}(\Lambda_1^{1/2})=\mathcal{H}_1$$
under the canonical identification of \(\mathcal H_1\) with its image in \(L^2(T)\), the choice \(\alpha=1/2\) recovers the usual RKHS assumption. Thus, the source-space formulation permits a finer description of the regularity of \(\mu_0\), allowing mean functions that are either rougher or smoother than elements of the RKHS according as \(\alpha<1/2\) or \(\alpha>1/2\).

\item[(iii)] While \cite{Cai_2011_mean} considers the penalized least-squares estimator corresponding to Tikhonov regularization, our formulation accommodates general spectral regularization schemes. The role of the qualification becomes particularly relevant for smooth target functions. Indeed, Theorem~\ref{thm:mean} shows that the approximation term is of order
\[
    \lambda^{\min\{\alpha,\nu\}},
\]
where \(\nu\) denotes the qualification of the regularization family. Consequently, Tikhonov regularization, whose qualification is \(\nu=1\), cannot exploit source smoothness beyond \(\alpha=1\) and exhibits saturation in this regime. Regularization schemes with higher, or infinite, qualification can exploit higher-order source conditions and, together with Theorem~\ref{thm:mean-lower}, yield the corresponding minimax rates beyond the Tikhonov saturation threshold; see Remark~\ref{rem:lower-mean}.

\item[(iv)] Finally, Theorem~\ref{thm:mean} allows for general subject-specific weights \(\{w_i\}_{i=1}^n\), thereby covering both subject-level (SUBJ) and observation-level (OBS) weighting schemes. In contrast, the formulation in \cite{Cai_2011_mean} corresponds to the SUBJ weighting scheme.
\end{itemize}
\end{remark}

While Theorem~\ref{thm:mean} holds for any $w_i$, two particular weighting schemes are studied in the literature, i.e., SUBJ where $w_i=\frac{1}{m_i}$, and OBS where $w_i = n/\sum_{i=1}^{n} m_i$.  The following corollaries specialize Theorem~\ref{thm:mean} to these two weighting schemes and provide the corresponding convergence rates.

\begin{corollary}[SUBJ]\label{cor:SUBJ}
Suppose Assumptions~\ref{mean_SRC}, \ref{mean_moment} and \ref{mean_eigenvalue_decay} hold. Let $\nu \geq 1$ be the qualification of the regularization family and $w_{i} = \frac{1}{m_{i}}$. Then, for $\alpha \geq \frac{1}{2}$, we have 
    \begin{equation*}
        \|\hat{\mu}_{\lambda}-\mu_{0}\|_{L^2(T)} \lesssim_{p} \frac{1}{\sqrt{n}}+\frac{1}{\sqrt{n m_{H}}}\lambda^{-\frac{1}{2b}}+ \lambda^{r},~~r = \min\{\alpha, \nu\}.
    \end{equation*}
Furthermore, if Assumption~\ref{mean_embedding} holds, then for $0 < \alpha \leq\frac{1}{2}$, we have
\begin{equation*}
        \|\hat{\mu}_{\lambda}-\mu_{0}\|_{L^2(T)} \lesssim_{p} \frac{1}{\sqrt{n}}+\frac{1}{\sqrt{n m_{H}}}\lambda^{-\frac{1}{2b}} + \lambda^{\alpha},
    \end{equation*}
where $m_{H}= \frac{n}{\sum_{i=1}^{n}\frac{1}{m_{i}}}$ is the harmonic mean of $(m_i)_i$.
\end{corollary}

With the choice of $\lambda = (nm_{H})^{-\frac{b}{1+2 r b}}$, we get

\begin{equation*}
    \|\hat{\mu}_{\lambda}-\mu_{0}\|_{L^2(T)} \lesssim_{p} \frac{1}{\sqrt{n}} + (nm_{H})^{-\frac{r b}{1+2 r b}}\lesssim_p \begin{cases}
    \frac{1}{\sqrt{n}},\qquad\qquad \quad\, m_H\ge n^{\frac{1}{2rb}}\\
    (nm_{H})^{-\frac{r b}{1+2 r b}},\quad\text{otherwise}
    \end{cases}.
\end{equation*}

\begin{corollary}[OBS]\label{cor:OBS}
Suppose Assumptions~\ref{mean_SRC}, \ref{mean_moment} and \ref{mean_eigenvalue_decay} hold. Let $\nu \geq 1$ be the qualification of the regularization family and $w_{i} = \frac{n}{\sum_{i=1}^{n}m_{i}}$. Then, for $\alpha \geq \frac{1}{2}$, we have 
    \begin{equation*}
        \|\hat{\mu}_{\lambda}-\mu_{0}\|_{L^2(T)} \lesssim_{p} \left(\frac{\sum_{i=1}^{n}m_{i}(m_{i}-1)}{(\sum_{i=1}^{n}m_{i})^2}\right)^{\frac{1}{2}}+\frac{1}{\sqrt{n m_{A}}}\lambda^{-\frac{1}{2b}}+ \lambda^{r},~~r = \min\{\alpha, \nu\}.
    \end{equation*}
Furthermore, if Assumption~\ref{mean_embedding} holds, then for $0 < \alpha \leq\frac{1}{2}$, we have
\begin{equation*}
        \|\hat{\mu}_{\lambda}-\mu_{0}\|_{L^2(T)} \lesssim_{p} \left(\frac{\sum_{i=1}^{n}m_{i}(m_{i}-1)}{(\sum_{i=1}^{n}m_{i})^2}\right)^{\frac{1}{2}}+\frac{1}{\sqrt{n m_{A}}}\lambda^{-\frac{1}{2b}} + \lambda^{\alpha},
    \end{equation*}
where $m_{A}= \frac{1}{n}\sum_{i=1}^{n}m_{i}$ is the arithmetic mean of $(m_i)_i$.
\end{corollary}

With the choice of $\lambda = (nm_{A})^{-\frac{b}{1+2 r b}}$, we get
\begin{align*}
    \|\hat{\mu}_{\lambda}-\mu_{0}\|_{L^2(T)} &\lesssim_{p}\left(\frac{\sum_{i=1}^{n}m_{i}(m_{i}-1)}{(\sum_{i=1}^{n}m_{i})^2}\right)^{\frac{1}{2}} +  (nm_{A})^{-\frac{rb}{1+2rb}}\\
    &\stackrel{(*)}{\lesssim_p} \frac{1}{\sqrt{n}}+(nm_{A})^{-\frac{rb}{1+2rb}}\lesssim_p\begin{cases}
    \frac{1}{\sqrt{n}},\qquad\qquad\quad\,m_A\ge n^{\frac{1}{2rb}}\\
    (nm_{A})^{-\frac{rb}{1+2rb}},\quad \text{otherwise}
    \end{cases},
\end{align*}
where $(*)$ holds 
under the assumption that $$\frac{1}{n}\sum_{i=1}^{n}m_{i}^2\le C\left(\frac{1}{n}\sum_{i=1}^{n}m_{i}\right)^2$$ for some positive constant $C$.\\

The following result (proved in Section~\ref{subsec:thm-mean-lower}) establishes lower rate for mean estimation.
\begin{theorem} \label{thm:mean-lower}
    Suppose Assumptions~\ref{mean_SRC},\ref{mean_moment} and \ref{mean_eigenvalue_decay} hold. Then for any $\alpha >0$, we have
\begin{equation*}
        \lim_{a \to 0} \lim_{n \to \infty} \inf_{\hat{\mu}} \sup_{\mu_{0} \in \mathcal{F}_{\alpha, R}} \mathbb{P}\left\{\|\hat{\mu}-\mu_{0}\|_{L^2(T)} \geq a\left( \frac{1}{\sqrt{n}}+  (m_{A}n)^{-\frac{\alpha b}{1+2 \alpha b}}\right)\right\} =1,
    \end{equation*}
where $m_{A} = \frac{1}{n}\sum_{i=1}^{n}m_{i}$ is the arithmetic mean.
\end{theorem}
\begin{remark}\label{rem:lower-mean}
The upper and lower bounds above lead to the following observations.

\begin{itemize}
\item[(i)] Corollaries~\ref{cor:SUBJ} and \ref{cor:OBS} reveal a distinction between the SUBJ and OBS weighting schemes when the sampling frequencies \(\{m_i\}_{i=1}^n\) are heterogeneous. In the dense regime, the SUBJ weighting scheme attains the parametric rate \(n^{-1/2}\), whereas an additional condition on the sampling frequencies is required for the OBS weighting scheme to attain the same rate. In the sparse regime, however, the situation is reversed: the rate obtained under OBS weighting matches the minimax lower bound, whereas the rate under SUBJ weighting is generally suboptimal. This discrepancy arises because the SUBJ rate depends on the harmonic mean of the sampling frequencies, while the minimax lower bound is governed by their arithmetic mean.

\item[(ii)] The convergence rates established by Cai and Yuan \cite{Cai_2011_mean} are based on the SUBJ weighting scheme. When the design is balanced, that is, \(m_i=m\) for all \(i\in[n]\), the arithmetic and harmonic means coincide, and their rates agree with the corresponding minimax lower bounds. For heterogeneous sampling frequencies, however, these two means can differ substantially. Our lower bound shows that the minimax rate is governed by the arithmetic, rather than the harmonic, mean of the sampling frequencies. Consequently, the SUBJ-based rates need not be minimax optimal under heterogeneous designs.

\item[(iii)] A second source of suboptimality is the qualification of the regularization method. For OBS weighting, comparison of the upper and lower bounds shows that the minimax rate is attained whenever
\[
    r=\min\{\alpha,\nu\}=\alpha,
\]
or equivalently, whenever \(\alpha\leq\nu\), where \(\nu\) denotes the qualification of the regularization family. Tikhonov regularization has qualification \(\nu=1\) and therefore cannot exploit source smoothness beyond \(\alpha=1\); its convergence rate saturates once \(\alpha>1\). By contrast, regularization methods with higher qualification continue to exploit additional smoothness and yield faster rates in the regime \(\alpha>1\), up to their respective qualification levels. In particular, regularization schemes with infinite qualification avoid this saturation and attain the minimax rates over the full range of source conditions considered here. This identifies the finite qualification of Tikhonov regularization as an additional limitation of the framework considered in \cite{Cai_2011_mean}.
\end{itemize}
\end{remark}

\section{Covariance Estimation}\label{sec:cov}
In this section, we develop a spectral regularization approach for estimating the covariance function $C_0$. In contrast to analyses that assume the mean function to be known or identically zero, we allow $\mu_0$ to be unknown and account explicitly for its estimation. To avoid additional dependence arising from using the same observations for mean and covariance estimation, we employ sample splitting: one subsample is used to construct the mean estimator, while the other is reserved for covariance estimation. Consequently, conditional on the estimated mean, the observations entering the covariance estimator remain independent of the mean-estimation sample. This construction avoids the additional dependence that would arise from reusing the same observations and allows the contribution of mean estimation to be analyzed separately.\\

\noindent
Let $\mathcal{H}_{2}$ be an RKHS with a continuous and bounded reproducing kernel
$K: (T \times T) \times (T \times T) \to \mathbb{R}$ such that $K((x_{1},y_{1}),(x_{2},y_{2}))= k(x_{1},x_{2})k(y_{1},y_{2})$ for a reproducing kernel $k: T \times T \to \mathbb{R}$. Let $J_{2}: \mathcal{H}_{2} \to L^2(T \times T)$ be the inclusion map and we define $\Lambda_{2}:= J_{2}J_{2}^*: L^2(T\times T) \to L^2(T\times T)$ given as
\begin{equation*}
    \Lambda_{2}F(\cdot, \cdot) = \int_{T \times T} k(s,\cdot) k(t,\cdot) F(s,t)\, ds dt,~~\forall~F \in L^2(T \times T).
\end{equation*}

\noindent
Notice that if $k(x,y) = \sum_{k} \tau_{k} \psi_{k}(x)\psi_{k}(y)$ then we have $$K((x_{1},y_{1}),(x_{2},y_{2})) = \sum_{k,k'}\tau_{k}\tau_{k'}(\psi_{k}\otimes \psi_{k'})(x_{1},y_{1})(\psi_{k}\otimes \psi_{k'})(x_{2},y_{2})$$ and we get
$$\Lambda_{2} = \sum_{k,k'}\tau_{k}\tau_{k'} \langle \cdot, \psi_{k} \otimes \psi_{k'} \rangle_{L^2(T \times T)}\psi_{k} \otimes \psi_{k'}.$$

\noindent
Moreover $\mathcal{H}_{2} = \mathcal{H}\otimes \mathcal{H}$, where $\mathcal{H}$ is RKHS for reproducing kernel $k$.\\

\noindent
As we can see from the given model for all $i \in [n],~j,k \in [m_{i}]$,
\begin{equation*}
\begin{split}
    \mathbb{E}[(Y_{ij}-\mathbb{E}[Y_{ij}])(Y_{ik}-\mathbb{E}[Y_{ik}])|t_{ij},t_{ik}] & =  \mathbb{E}[(X_{i}(t_{ij})-\mu_{0}(t_{ij})+\epsilon_{ij})(X_{i}(t_{ik})-\mu_{0}(t_{ik})+\epsilon_{ik})]\\
    & =  \mathbb{E}[(X_{i}(t_{ij})-\mu_{0}(t_{ij}))(X_{i}(t_{ik})-\mu_{0}(t_{ik}))] + \mathbb{E}[\epsilon_{ij}\epsilon_{ik}]\\
    & =  C_{0}(t_{ij},t_{ik}) + \sigma_{0}^2 \delta_{jk}.
\end{split}
\end{equation*}
Motivated by this, an RKHS based estimator of $C_{0}$ can be given as
\begin{equation}\label{covariance_variational}
    \begin{split}
        \hat{C}_{\eta} := \arg\min_{C \in \mathcal{H}_{2}}\left[\frac{1}{n}\sum_{i=1}^{n}v_{i}\sum_{1\leq j\neq k\leq m_{i}}[(Y_{ij}-\hat{\mu}_{\eta}(t_{ij}))(Y_{ik}-\hat{\mu}_{\eta}(t_{ik}))-C(t_{ij},t_{ik})]^2\right], 
    \end{split}
\end{equation}
where $\hat{\mu}_{\eta}$ is the estimator of the mean function $\mu_{0}$ as given in $(\ref{mean_estimator})$ with $\eta$ being the regularization parameter. The positive weights $\{v_{i}\}_{i=1}^{n}$ are normalized such that $\frac{1}{n}\sum_{i=1}^{n}v_{i} m_{i} (m_{i}-1) = 1$. 

\begin{remark}
\begin{itemize}
\item[(i)] The mean estimator $\hat{\mu}_{\eta}$ is constructed using a subset of the data that is separate from the observations subsequently used for covariance estimation. This sample-splitting device ensures that $\hat{\mu}_{\eta}$ is independent of the observations entering the covariance estimator. Such independence allows us to control the error resulting from replacing the unknown mean $\mu_{0}$ by $\hat{\mu}_{\eta}$ without introducing additional dependence terms. As a consequence, the contribution of mean estimation can be separated from the stochastic error arising in covariance estimation, which is instrumental in obtaining the sharper convergence rates established below.

\item[(ii)] We exclude the diagonal terms $j=k$ from the covariance estimator. Indeed, under the measurement error model $Y_{ij}=X_i(t_{ij})+\varepsilon_{ij}$, the diagonal products contain an additional contribution from the measurement error variance $\sigma_{0}^{2}$, whereas for $j\neq k$ the independence and centering of the measurement errors imply that the corresponding cross-products are free of this contribution. Excluding the diagonal terms therefore permits direct estimation of the covariance function $C_{0}$ without separately estimating $\sigma_{0}^{2}$. The resulting estimator is closely related to that considered by Cai and Yuan \cite{cai2010nonparametric}. Their analysis, however, focuses on the zero-mean setting and employs a squared RKHS-norm penalty, corresponding to Tikhonov regularization. In contrast, our formulation accommodates an unknown, nonzero mean function and allows for a general class of spectral regularization methods.
\end{itemize}
\end{remark}

\noindent
For each $1\leq i\leq n$, define the linear operator
$$ B_i:\mathcal{H}_2\to \mathbb{R}^{m_i(m_i-1)} $$
by
$$
B_iF = \left\{ \sqrt{v_i} \left\langle F,\, k(t_{ij},\cdot)k(t_{ik},\cdot) \right\rangle_{\mathcal{H}_2} \right\}_{1\leq j\neq k\leq m_i} = \left\{
\sqrt{v_i}\,F(t_{ij},t_{ik})
\right\}_{1\leq j\neq k\leq m_i}.$$

Further, let

$$
W_i
=
\left\{
\sqrt{v_i}
\bigl(Y_{ij}-\hat{\mu}_{\eta}(t_{ij})\bigr)
\bigl(Y_{ik}-\hat{\mu}_{\eta}(t_{ik})\bigr)
\right\}_{1\leq j\neq k\leq m_i}.
$$

Then the empirical least-squares estimator can be written as

$$
\hat{C}_{\eta}
:=
\arg\min_{C\in\mathcal{H}_2}
\frac{1}{n}
\sum_{i=1}^{n}
\|W_i-B_iC\|_{\mathbb{R}^{m_i(m_i-1)}}^2.
$$

The corresponding normal equation is

$$ \hat{B}_n\hat{C}_{\eta}=O_1, $$
where $ \hat{B}_n := \frac{1}{n}\sum_{i=1}^{n}B_i^*B_i $ and $ O_1 := \frac{1}{n}\sum_{i=1}^{n}B_i^*W_i.$\\

\noindent
In general, the empirical operator $\hat{B}_n:\mathcal{H}_2\to\mathcal{H}_2$ need not be invertible. We therefore define a regularized estimator of $C_0$ by
\begin{equation}\label{eq:cov}
\hat{C}_{\eta,\lambda} := g_{\lambda}(\hat{B}_n)O_1,
\end{equation}
where $g_{\lambda}$ denotes a spectral regularization function.

\begin{remark}
The product structure
$$ K\big((s_{1},t_{1}),(s_{2},t_{2})\big) = k(s_{1},s_{2})k(t_{1},t_{2}) $$
is chosen to preserve the intrinsic symmetry of the covariance function. In particular, since the empirical covariance term is symmetric in its two arguments and the product kernel treats the two coordinates in an identical manner, the associated empirical operator preserves the subspace of symmetric functions. Consequently, applying a spectral regularization method to this operator yields a symmetric estimator, so that
    $$\hat{C}_{\eta,\lambda}(s,t) = \hat{C}_{\eta,\lambda}(t,s).$$
Thus, symmetry is built directly into the construction of the estimator rather than being imposed through a subsequent symmetrization step.
\end{remark}

Similar to Proposition~\ref{mean_alternate_estimator} as in the mean estimation case, we provide an alternative form of the covariance estimator $\hat{C}_{\eta,\lambda}$ which we will use for our convergence analysis.

\begin{prop}\label{covariance_alternate_estimator}
Let $K:(T\times T)\times(T\times T)\to\mathbb{R}$ be a bounded and continuous kernel of the product form

$$ K\big((x_1,y_1),(x_2,y_2)\big) = k(x_1,x_2)k(y_1,y_2). $$

Then the regularized covariance estimator in \eqref{eq:cov} admits the equivalent representation

$$ \hat{C}_{\eta,\lambda} = g_{\lambda}(\hat{B}_n)O_1 = \Lambda_2^{1/2}g_{\lambda}(\hat{T}_n)O_2, $$

where $\hat{T}_n:L^2(T\times T)\to L^2(T\times T)$ is defined by

$$
\hat{T}_n F
=
\frac{1}{n}
\sum_{i=1}^n v_i
\sum_{1\leq j\neq k\leq m_i}
\left\langle
F,\,
k^{1/2}(t_{ij},\cdot)
k^{1/2}(t_{ik},\cdot)
\right\rangle_{L^2(T\times T)}
k^{1/2}(t_{ij},\cdot)
k^{1/2}(t_{ik},\cdot),
\qquad
F\in L^2(T\times T),
$$

and

$$
O_2
=
\frac{1}{n}
\sum_{i=1}^n v_i
\sum_{1\leq j\neq k\leq m_i}
\bigl(Y_{ij}-\hat{\mu}_{\eta}(t_{ij})\bigr)
\bigl(Y_{ik}-\hat{\mu}_{\eta}(t_{ik})\bigr)
k^{1/2}(t_{ij},\cdot)
k^{1/2}(t_{ik},\cdot).
$$

Here,

$$
k^{1/2}(s,\cdot)
:=
\sum_{\beta\geq1}
\tau_\beta^{1/2}\psi_\beta(s)\psi_\beta(\cdot),
\qquad s\in T,
$$

where ${(\tau_\beta,\psi_\beta)}_{\beta\geq1}$ denotes the collection of eigenvalue--eigenfunction pairs associated with the integral operator induced by $k$. Equivalently,

$$
\left\{
\bigl(\tau_\beta\tau_\gamma,\,
\psi_\beta\otimes\psi_\gamma\bigr)
:\beta,\gamma\geq1
\right\}
$$

is the corresponding collection of eigenvalue--eigenfunction pairs of $\Lambda_2$.
\end{prop}

\noindent
We omit the proof of Proposition~\ref{covariance_alternate_estimator} since it follows by arguments analogous to those used in the proof of Proposition~\ref{mean_alternate_estimator}. 

\begin{remark}
Replacing $g_{\lambda}(\hat{B}_{n})O_{1}$ with 
$\Lambda_{2}^{\frac{1}{2}}g_{\lambda}(\hat{T}_{n})O_{2}$ as an expression for $\hat{C}_{\eta, \lambda}$ simplifies the analysis, since it allows us to work with the empirical operator
$\hat{T}_{n} : L^2(T \times T) \to L^2(T \times T)$ rather than $\hat{B}_{n}: \mathcal{H}_{2}\to \mathcal{H}_{2}$. In particular, $\hat{T}_{n}$ acts on the same space as the population operator $\Lambda_{2}$, which facilitates direct comparison between the empirical and population operators.
\end{remark}

\noindent
Define  $C^{\text{true}}_{\lambda} = \Lambda_{2}^{\frac{1}{2}}g_{\lambda}(\hat{T}_{n})O_{3}$, where 
$$O_{3} = \frac{1}{n}\sum_{i=1}^{n}v_{i}\sum_{1 \leq j \neq k \leq m_{i}}(Y_{ij}-\mu_{0}(t_{ij}))(Y_{ik}-\mu_{0}(t_{ik}))k^{\frac{1}{2}}(t_{ij},\cdot)k^{\frac{1}{2}}(t_{ik},\cdot).$$ 

We use the notation $C_{\lambda}^{\text{true}}$ as we use true mean function in the definition.

\begin{remark}
  Note that the construction of $\hat{C}_{\eta, \lambda}$ relies on an estimator of the mean function rather than on the true mean function. By introducing the intermediate quantity $C^{\emph{true}}_{\lambda}$, we can decouple the analysis, allowing the discrepancy between $\hat{\mu}_{\eta}$ and $\mu_{0}$ to be treated separately in the covariance estimation error.
\end{remark}

\begin{remark}
    Since $O_{1}$ is defined by summing over all ordered pairs $j\neq k$, it is symmetric. Moreover, the space of symmetric functions is invariant under $g_{\lambda}(\hat{B}_{n})$. Consequently, the estimator $\hat{C}_{\eta,\lambda}$ is symmetric. The product-kernel structure is chosen precisely to preserve this symmetry, ensuring that the symmetry of $O_{1}$ is inherited by the resulting estimator. Although $\hat{C}_{\eta,\lambda}$ is symmetric, it is not necessarily positive semi definite (PSD). To obtain a valid covariance estimator, we therefore project $\hat{C}_{\eta,\lambda}$ onto the cone of symmetric positive semi definite functions in $L^{2}(T\times T)$.
\end{remark}

\subsection{Construction of a positive semidefinite estimator}

The estimator $\hat{C}_{\eta,\lambda}$ is symmetric, but need not be positive semidefinite. We therefore construct its positive semidefinite projection. By the finite-dimensional representation of $\hat{C}_{\eta,\lambda}$,
\begin{equation}\label{eq:cov-representation}
\hat{C}_{\eta,\lambda}(s,t) = \sum_{i=1}^{n}
\sum_{1\leq j\neq k\leq m_i} \alpha_{ijk}^{\lambda}
k(t_{ij},s)k(t_{ik},t).
\end{equation}
Let

$$
M=\sum_{i=1}^{n}m_i,
\qquad
r(i,j)=\sum_{q=1}^{i-1}m_q+j,
$$

and let ${s_1,\ldots,s_M}$ denote the collection of all observation points under this indexing. Define

$$
\varphi(\cdot)
=
\bigl(
k(s_1,\cdot),\ldots,k(s_M,\cdot)
\bigr)^{\top}.
$$

Then \eqref{eq:cov-representation} can equivalently be written as
\begin{equation}\label{eq:cov-matrix-representation}
\hat{C}_{\eta,\lambda}(s,t)=  \varphi(s)^{\top} A_{\lambda}\varphi(t),
\end{equation}
where

$$
A_{\lambda}[r(i,j),r(i,k)]
=
\alpha_{ijk}^{\lambda},
\qquad
j\neq k,
$$

and all remaining entries are zero. Since $\hat{C}_{\eta,\lambda}$ is symmetric, we have that $A_{\lambda}$ is symmetric as well. Define the $L^{2}(T)$ Gram matrix $G\in\mathbb{R}^{M\times M}$ by

$$
G_{pq}
=
\left\langle
k(s_p,\cdot),k(s_q,\cdot)
\right\rangle_{L^{2}(T)},
\qquad
1\leq p,q\leq M.
$$

Let $r=\operatorname{rank}(G)$ and consider the reduced eigendecomposition

$$
G
=
U_r\Gamma_rU_r^{\top},
\qquad
\Gamma_r
=
\operatorname{Diag}(\gamma_1,\ldots,\gamma_r),
$$

where $\gamma_\ell>0$ for $\ell=1,\ldots,r$, and
$U_r\in\mathbb{R}^{M\times r}$ has orthonormal columns. For
$\ell=1,\ldots,r$, define
\begin{equation}\label{eq:orthonormal-basis}
\psi_\ell= 
\frac{1}{\sqrt{\gamma_\ell}}
\sum_{j=1}^{M}
(U_r)_{j\ell}k(s_j,\cdot).
\end{equation}
Then $\{\psi_\ell\}_{\ell=1}^{r}$ is an orthonormal basis of

$$
\mathcal{V}_M
:=
\operatorname{span}_{L^{2}}
\left\{
k(s_j,\cdot):1\leq j\leq M
\right\}.
$$

Indeed,

$$
\begin{aligned}
\left\langle
\psi_\ell,\psi_{\ell'}
\right\rangle_{L^{2}(T)}
&=
\frac{1}{\sqrt{\gamma_\ell\gamma_{\ell'}}}
\sum_{j,q=1}^{M}
(U_r)_{j\ell}(U_r)_{q\ell'}
\left\langle
k(s_j,\cdot),k(s_q,\cdot)
\right\rangle_{L^{2}(T)}
\\
&=
\frac{1}{\sqrt{\gamma_\ell\gamma_{\ell'}}}
e_\ell^{\top}U_r^{\top}GU_re_{\ell'} = 
\frac{1}{\sqrt{\gamma_\ell\gamma_{\ell'}}}
e_\ell^{\top}\Gamma_re_{\ell'}
=
\delta_{\ell\ell'}.
\end{aligned}
$$

In the orthonormal basis \eqref{eq:orthonormal-basis}, the integral operator induced by $\hat{C}_{\eta,\lambda}$ is represented by
\begin{equation}\label{eq:S-matrix}
S= \Gamma_r^{1/2}
U_r^{\top}A_{\lambda}U_r
\Gamma_r^{1/2}.
\end{equation}

Since $A_{\lambda}$ is symmetric, so is $S$. Let

$$
S
=
QDQ^{\top},
\qquad
D
=
\operatorname{Diag}(d_1,\ldots,d_r),
$$

be an eigendecomposition of $S$, and define

$$
D^{+}
=
\operatorname{Diag}(d_1^{+},\ldots,d_r^{+}),
\qquad
d_\ell^{+}
=
\max\{d_\ell,0\},
$$

and

$$
S^{+}
=
QD^{+}Q^{\top}.
$$

The corresponding coefficient matrix in the original kernel representation is
\begin{equation}\label{eq:A-plus}
A_{\lambda}^{+}= 
U_r\Gamma_r^{-1/2} S^{+} \Gamma_r^{-1/2}U_r^{\top}.
\end{equation}
We therefore define
\begin{equation}\label{eq:cov-psd-estimator}
\hat{C}_{\eta,\lambda}^{+}(s,t)= \sum_{p,q=1}^{M}
(A_{\lambda}^{+})_{pq} k(s_p,s)k(s_q,t).
\end{equation}

We next show that $\hat{C}_{\eta,\lambda}^{+}$ is the $L^{2}(T\times T)$-metric projection of
$\hat{C}_{\eta,\lambda}$ onto the cone of positive semidefinite kernels. Define
$$
H_M:\mathbb{R}^{M}\to L^{2}(T),
\qquad
H_Ma
=
\sum_{j=1}^{M}a_jk(s_j,\cdot).
$$
Then $G = H_M^{*}H_M. $ Define

$$\Psi = H_MU_r\Gamma_r^{-1/2} : \mathbb{R}^{r}\to L^{2}(T).$$

Since 
$$ \Psi^{*}\Psi = \Gamma_r^{-1/2} U_r^{\top}H_M^{*}H_MU_r \Gamma_r^{-1/2} = I_r,$$
the operator $\Psi$ is an isometry from $\mathbb{R}^{r}$ onto
$\mathcal{V}_M$. Moreover,

$$
\begin{aligned}
T_{\hat{C}_{\eta,\lambda}}
&=
H_MA_{\lambda}H_M^{*} = \Psi \Gamma_r^{1/2} U_r^{\top}A_{\lambda}U_r \Gamma_r^{1/2} \Psi^{*} = \Psi S\Psi^{*}.
\end{aligned}
$$

Similarly,

$$ T_{\hat{C}_{\eta,\lambda}^{+}} =\Psi S^{+}\Psi^{*}. $$

Since $\Psi$ is an isometry,

$$ \left\| \Psi B\Psi^{*} \right\|_{\mathrm{HS}} = \|B\|_{F}, \qquad B\in\mathbb{R}^{r\times r}. $$

Furthermore, for every $C\in L^{2}(T\times T)$,

$$ \|T_C\|_{\mathrm{HS}} = \|C\|_{L^{2}(T\times T)}. $$

Therefore,

\begin{align}
\left\|\hat{C}_{\eta,\lambda}- \hat{C}^{+}_{\eta, \lambda}\right\|_{L^{2}(T\times T)} &= \left\| T_{\hat{C}_{\eta,\lambda}}- T_{\hat{C}^{+}_{\eta, \lambda}}\right\|_{\mathrm{HS}}\\
&=
\left\| \Psi(S-S^{+})\Psi^{*} \right\|_{\mathrm{HS}} \\
&= \|S-S^{+}\|_{F}.
\label{eq:projection-isometry}
\end{align}

Let
$$\mathcal{K}^{+} = \left\{ C\in L^{2}(T\times T):
C(s,t)=C(t,s)\ \text{a.e.}, \quad \left\langle T_Cf,f
\right\rangle_{L^{2}(T)} \geq 0 ~\forall~ f\in L^{2}(T)\right\}. $$

Under the isometric correspondence $C\mapsto T_C$, the set
$\mathcal{K}^{+}$ is identified with the closed convex cone of positive self-adjoint Hilbert--Schmidt operators on $L^{2}(T)$. Since $S^{+}$ is obtained by replacing the negative eigenvalues of $S$ by zero, the operator $\Psi S^{+}\Psi^{*}$ is the positive part of $T_{\hat{C}_{\eta,\lambda}}$. Hence,
\begin{equation}\label{eq:psd-projection}
\hat{C}_{\eta,\lambda}^{+}= \arg\min_{C\in\mathcal{K}^{+}} \left\| C-\hat{C}_{\eta,\lambda}\right\|_{L^{2}(T\times T)}.
\end{equation}

Finally, since $C_0\in\mathcal{K}^{+}$,

$$
\begin{aligned}
\left\|
\hat{C}_{\eta,\lambda}^{+}-C_0
\right\|_{L^{2}(T\times T)} &\leq
\|\hat{C}_{\eta,\lambda}^{+}- \hat{C}_{\eta,\lambda}\|_{L^2(T \times T)} + \|\hat{C}_{\eta,\lambda}- C_{0}\|_{L^2(T \times T)}\\
&\leq
2\left\|
\hat{C}_{\eta,\lambda}-C_0
\right\|_{L^{2}(T\times T)}.
\end{aligned}
$$

Thus projection onto the positive semidefinite cone preserves the
$L^{2}(T\times T)$ convergence rate of the original estimator.

\noindent
We next state the counterparts of Assumptions~\ref{mean_SRC}, \ref{mean_eigenvalue_decay}, and \ref{mean_embedding} for covariance estimation.

\begin{assumption}\label{covariance_src}
    For some $\alpha_{1} > 0$, $C_{0} \in \mathcal{F}^{C}_{\alpha_{1},R_{1}} = \{F, F = (\Lambda_{2}^{\alpha_{1}})H, \|H\|_{L^2(T \times T)} \leq R_{1}\}$.
\end{assumption}

\begin{assumption}\label{covariance_eigen_decay}
There exists constant $b_{1}>1$, such that
\begin{equation*}
    i^{-b_{1}} \lesssim \tau_{i} \lesssim i^{-b_{1}},~~~\forall~~i \in \mathbb{N}.
\end{equation*}
\end{assumption}

\noindent
\begin{assumption}\label{covariance_embedding}
    For $0 < \alpha_{1} \leq \frac{1}{2}$, there is a constant $Z_{1} >0$ such that 
    \begin{equation*}
        \|[\mathcal{H}]^{\alpha_{1}} \hookrightarrow L_{\infty}(T)\|_{\emph{op}} \leq Z_{1},
    \end{equation*}
where $[\mathcal{H}]^{\alpha_{1}}$ is defined in similar way as $[\mathcal{H}_{1}^{\alpha}]$ in Assumption~\ref{mean_embedding}.
\end{assumption}

\noindent
The following result (proved in Section~\ref{subsec:thm-cov-upper}) provides the convergence rate for $\hat{C}_{\eta,\lambda}$.
\begin{theorem}\label{thm:covariance}
    Suppose Assumptions \ref{mean_moment},~\ref{covariance_src}, and \ref{covariance_eigen_decay} hold. Let $\nu_{1} \geq 1$ be the qualification of the regularization family. Then for $ \alpha_{1} \geq \frac{1}{2}$, we have 
    \begin{equation*}
    \begin{split}
        \|\hat{C}_{\eta, \lambda}-C_{0}\|_{L^2(T \times T)} \lesssim_{p} & \|\hat{\mu}_{\eta}- \mu_{0}\|^2_{L^2(T)} + \sqrt{\frac{1}{n^2}\sum_{i=1}^{n}v_{i}^2 m_{i}^{\underline{4}}} + \sqrt{\frac{1}{n^2}\sum_{i=1}^{n}v_{i}^2m_{i}^{\underline{2}} \lambda^{-\frac{1}{b_{1}}} \log\left(\frac{1}{\lambda}\right)}\\
        & \qquad+ \sqrt{\frac{1}{n^2}\sum_{i=1}^{n}v_{i}^2 m_{i}^{\underline{3}} \lambda^{-\frac{1}{b_{1}}}} + \lambda^{r_{1}},~~r_{1} = \min\{\alpha_{1},\nu_{1}\}.
    \end{split}
    \end{equation*}

\noindent
Further, if Assumption~\ref{covariance_embedding} holds, then for $0 < \alpha_{1} \leq \frac{1}{2}$, we have
\begin{equation*}
\begin{split}
        \|\hat{C}_{\eta,\lambda}-C_{0}\|_{L^2(T \times T)} \lesssim_{p} & \|\hat{\mu}_{\eta}- \mu_{0}\|^2_{L^2(T)} + \sqrt{\frac{1}{n^2}\sum_{i=1}^{n}v_{i}^2 m_{i}^{\underline{4}}} + \sqrt{\frac{1}{n^2}\sum_{i=1}^{n}v_{i}^2m_{i}^{\underline{2}} \lambda^{-\frac{1}{b_{1}}} \log\left(\frac{1}{\lambda}\right)}\\
        & \qquad+ \sqrt{\frac{1}{n^2}\sum_{i=1}^{n}v_{i}^2 m_{i}^{\underline{3}} \lambda^{-\frac{1}{b_{1}}}} + \lambda^{\alpha_{1}}.
    \end{split}
    \end{equation*}
\end{theorem}

\begin{remark}
(i) If different regularization families are employed for estimating the mean and covariance functions, then the smoothness parameters $\alpha $ and $\alpha_{1}$ associated with the mean and covariance functions, respectively, must not exceed the minimum qualification of the corresponding regularization families, namely $\nu$ and $\nu_{1}$ to achieve the optimal convergence rates. That is, $\alpha \leq \nu \text{ and } \alpha_{1} \le \nu_{1}.$ Consequently, the effective parameters become $r = \min\{\alpha,\nu\}$  and $r_{1} = \min\{\alpha_{1},\nu_{1}\}$. \vspace{1mm}\\

(ii) Instead of employing separate source conditions, a unified source condition can be employed on the underlying stochastic process, i.e., path of the stochastic process $X$ lies in the range of $\Lambda_{1}^{\alpha}$, where $\Lambda_{1}$ is the integral operator corresponding to a reproducing kernel $k : T \times T \to \mathbb{R}$. Under this assumption, it is easy to prove that $\mu_{0} \in \mathcal{R}(\Lambda_{1}^{\alpha})$ and $C_{0} \in \mathcal{R}(\Lambda_{2}^{\alpha})$, where $\Lambda_{2}$ is the integral operator associated with the  reproducing kernel $k \otimes k: (T \times T) \times (T \times T) \to \mathbb{R}$.
\end{remark}

\begin{corollary}\label{condition_on_process}
    Suppose the sample paths of $X$ lie in $\mathcal{F}_{\alpha,R}$ almost surely, and Assumptions~\ref{mean_moment} and \ref{mean_eigenvalue_decay} hold. Let $\nu\geq 1$ be the qualification of the regularization family. Then for $\alpha \geq \frac{1}{2}$, 
    \begin{equation*}
    \begin{split}   
        \|C_{0}-\hat{C}_{\eta, \lambda}\|_{L^2(T \times T)} \lesssim_p &  \|\mu_{0}-\hat{\mu}_{\eta}\|_{L^2(T)}^2 + \sqrt{\frac{1}{n^2}\sum_{i=1}^{n}v_{i}^2 m_{i}^{\underline{4}}} + \sqrt{\frac{1}{n^2}\sum_{i=1}^{n}v_{i}^2m_{i}^{\underline{2}} \lambda^{-\frac{1}{b}} \log\left(\frac{1}{\lambda}\right)}\\
        & \qquad+ \sqrt{\frac{1}{n^2}\sum_{i=1}^{n}v_{i}^2 m_{i}^{\underline{3}} \lambda^{-\frac{1}{b}}} + \lambda^{r},~~r = \min\{\alpha,\nu\}
     \end{split}
    \end{equation*}
Further, if Assumption~\ref{mean_embedding} holds, then, for $0 <\alpha \leq \frac{1}{2}$, 
\begin{equation*}
\begin{split}
    \|C_{0}-\hat{C}_{\eta, \lambda}\|_{L^2(T \times T)} \lesssim_p &  \|\mu_{0}-\hat{\mu}_{\eta}\|_{L^2(T)}^2 + \sqrt{\frac{1}{n^2}\sum_{i=1}^{n}v_{i}^2 m_{i}^{\underline{4}}} + \sqrt{\frac{1}{n^2}\sum_{i=1}^{n}v_{i}^2m_{i}^{\underline{2}} \lambda^{-\frac{1}{b}} \log\left(\frac{1}{\lambda}\right)}\\
        & \qquad+ \sqrt{\frac{1}{n^2}\sum_{i=1}^{n}v_{i}^2 m_{i}^{\underline{3}} \lambda^{-\frac{1}{b}}} + \lambda^{\alpha}.
\end{split}
\end{equation*}
\end{corollary}

\begin{remark}
As in the mean estimation setting, the general convergence bound in Theorem~\ref{thm:covariance} accommodates a broad class of weighting schemes. The following corollaries specialize this result to the OBS and SUBJ weights, thereby making explicit the dependence of the convergence rates on the sampling frequencies $\{m_i\}_{i=1}^{n}$. These two choices also highlight the distinct behavior of the estimator under observation-level and subject-level weighting, particularly when the sampling frequencies vary across subjects.
\end{remark}

\begin{corollary}[SUBJ]\label{Cov:SUBJ}
    Suppose Assumptions \ref{mean_moment},~\ref{covariance_src}, and \ref{covariance_eigen_decay} hold. Let $\nu_{1} \geq 1$ be the qualification of the regularization family and $v_{i} = \frac{1}{m_{i}(m_{i}-1)}$. Then for $ \alpha_{1} \geq \frac{1}{2}$, we have 
    \begin{equation*}
    \begin{split}
        \|\hat{C}_{\eta, \lambda}-C_{0}\|_{L^2(T \times T)} \lesssim_{p} & \|\hat{\mu}_{\eta}- \mu_{0}\|^2_{L^2(T)} + \frac{1}{\sqrt{n}}+ \sqrt{\frac{1}{n^2}\sum_{i=1}^{n}\frac{1}{m_{i}(m_{i}-1)} \lambda^{-\frac{1}{b_{1}}} \log\left(\frac{1}{\lambda}\right)}\\
        & \qquad+ \sqrt{\frac{1}{n^2}\sum_{i=1}^{n}\frac{1}{m_{i}} \lambda^{-\frac{1}{b_{1}}}} + \lambda^{r_{1}},~~r_{1} = \min\{\alpha_{1},\nu_{1}\}.
    \end{split}
    \end{equation*}
\noindent
Further, if Assumption~\ref{covariance_embedding} holds, then for $0 < \alpha_{1} \leq \frac{1}{2}$, we have
\begin{equation*}
\begin{split}
        \|\hat{C}_{\eta,\lambda}-C_{0}\|_{L^2(T \times T)} \lesssim_{p} & \|\hat{\mu}_{\eta}- \mu_{0}\|^2_{L^2(T)} + \frac{1}{\sqrt{n}}+ \sqrt{\frac{1}{n^2}\sum_{i=1}^{n}\frac{1}{m_{i}(m_{i}-1)} \lambda^{-\frac{1}{b_{1}}} \log\left(\frac{1}{\lambda}\right)}\\
        & \qquad+ \sqrt{\frac{1}{n^2}\sum_{i=1}^{n}\frac{1}{m_{i}} \lambda^{-\frac{1}{b_{1}}}} + \lambda^{\alpha_{1}}.
    \end{split}
    \end{equation*}
\end{corollary}

\begin{remark}\label{cov:subj_comparision}
(i) Corollary~\ref{Cov:SUBJ} shows that, when $\nu_{1}\geq \alpha_{1}$ and $m_i=m$ for all $1\leq i\leq n$, provided that the mean function can be estimated at a sufficiently fast rate, we have
\begin{equation*}
\|\hat{C}_{\eta,\lambda}-C_{0}\|_{L^2(T \times T)}
\lesssim_{p}
\begin{cases}
\displaystyle
\frac{1}{\sqrt{n}},
& m \gtrsim n^{\frac{1}{2\alpha_{1}b_{1}}},\\
\displaystyle
(nm)^{-\frac{\alpha_{1}b_{1}}{1+2\alpha_{1}b_{1}}},
& \displaystyle
n^{\frac{1}{2\alpha_{1}b_{1}}}\gtrsim m
\gtrsim \log(nm^2),
\\
\displaystyle
\left(\frac{\log(nm^2)}{nm^2}\right)^{\frac{\alpha_{1}b_{1}}
{1+2\alpha_{1}b_{1}}},
& \log(nm^2)\gtrsim m.
\end{cases}
\end{equation*}
Thus, the convergence behavior can be divided into three regimes: dense, sparse, and very sparse, respectively.

(ii) Cai and Yuan~\cite{cai2010nonparametric} assume that the sample paths of the underlying stochastic process belong to the RKHS, which, in our framework, implies $C_0\in\mathcal H_2$. Their results can therefore be compared with ours by taking $\alpha_1=1/2$. Under the SUBJ weighting scheme, our analysis further distinguishes the sparse setting into sparse and very sparse regimes. In the very sparse regime, our rate is
$\left(\frac{\log(nm^2)}{nm^2}\right)^{\frac{b_1}{2(1+b_1)}},$
whereas the corresponding sparse-regime rate of Cai and Yuan is
$\left(\frac{\log n}{nm}\right)^{\frac{b_1}{2(1+b_1)}}.$

These rates are of the same order when $m$ remains bounded, while our rate can be faster when $m$ increases with $n$ within the very sparse regime. In the sparse regime, our rate further improves to $(nm)^{-\frac{b_1}{2(1+b_1)}},$ which is minimax optimal. Moreover, the transition from the sparse regime to the parametric rate $n^{-1/2}$ occurs at a smaller sampling frequency $m$ than that required by Cai and Yuan.

\end{remark}

\begin{corollary}[OBS]
    Suppose Assumptions \ref{mean_moment},~\ref{covariance_src}, and \ref{covariance_eigen_decay} hold. Let $\nu_{1} \geq 1$ be the qualification of the regularization family and $v_{i} = \frac{n}{\sum_{i=1}^{n}m_{i}(m_{i}-1)}$. Then for $ \alpha_{1} \geq \frac{1}{2}$, we have 
    \begin{equation*}
    \begin{split}
        \|\hat{C}_{\eta, \lambda}-C_{0}\|_{L^2(T \times T)} \lesssim_{p} & \|\hat{\mu}_{\eta}- \mu_{0}\|^2_{L^2(T)} + \frac{\sqrt{\sum_{i=1}^{n}m_{i}^{\underline{4}}}}{\sum_{i=1}^{n}m_{i}^{\underline{2}}}+ \sqrt{\frac{1}{\sum_{i=1}^{n}m_{i}(m_{i}-1)} \lambda^{-\frac{1}{b_{1}}} \log\left(\frac{1}{\lambda}\right)}\\
        & \qquad+ \frac{\sqrt{\sum_{i=1}^{n}m_{i}^{\underline{3}} \lambda^{-\frac{1}{b_{1}}}}}{\sum_{i=1}^{n}m_{i}^{\underline{2}}} + \lambda^{r_{1}},~~r_{1} = \min\{\alpha_{1},\nu_{1}\}.
    \end{split}
    \end{equation*}
\noindent
Further, if Assumption~\ref{covariance_embedding} holds, then for $0 < \alpha_{1} \leq \frac{1}{2}$, we have
\begin{equation*}
\begin{split}
        \|\hat{C}_{\eta,\lambda}-C_{0}\|_{L^2(T \times T)} \lesssim_{p} & \|\hat{\mu}_{\eta}- \mu_{0}\|^2_{L^2(T)} + \frac{\sqrt{\sum_{i=1}^{n}m_{i}^{\underline{4}}}}{\sum_{i=1}^{n}m_{i}^{\underline{2}}}+ \sqrt{\frac{1}{\sum_{i=1}^{n}m_{i}(m_{i}-1)} \lambda^{-\frac{1}{b_{1}}} \log\left(\frac{1}{\lambda}\right)}\\
        & \qquad+ \frac{\sqrt{\sum_{i=1}^{n}m_{i}^{\underline{3}} \lambda^{-\frac{1}{b_{1}}}}}{\sum_{i=1}^{n}m_{i}^{\underline{2}}} + \lambda^{\alpha_{1}}.
    \end{split}
    \end{equation*}
\end{corollary}

The following results (proved in Sections~\ref{subsec:thm-cov-lower-1} and \ref{subsec:thm-cov-lower-2}) combinedly establish the minimax optimality of the proposed covariance function estimator.

\begin{theorem}
\label{thm:cov-lower-1}
Suppose $\tau_\ell\asymp \ell^{-b_1}$ for some $b_1>1$ and $C_0\in\mathcal{F}^{C}_{\alpha_1,R_{1}}$, $\alpha_1>0$. Then

$$
\inf_{\widehat C}\sup_{P\in\mathcal P}
\mathbb{E}_{P}
\left[
\|\widehat C-C_P\|_{L^2(T\times T)}^2
\right]
\gtrsim
\left(
\inf_{\widehat\mu}\sup_{P\in\mathcal P}
\mathbb{E}_{P}
\left[
\|\widehat\mu-\mu_P\|_{L^2(T)}^2
\right]
\right)^2.
$$

Moreover, any minimax lower bound for covariance estimation under known mean continues to hold when the mean is unknown.

\end{theorem}

To obtain a lower bound for the full parameter class, it suffices to restrict attention to the zero-mean subclass. We therefore derive the following lower bound under $\mu_0=0$.

\begin{theorem}\label{thm:lower_covariance}
    Suppose Assumptions~\ref{covariance_src} and \ref{covariance_eigen_decay} hold. Let $ m_{A}=\frac{1}{n}\sum_{i=1}^{n}m_{i}.$ Then there exists a constant $c>0$, independent of $n$ and $\{m_{i}\}_{i=1}^{n}$, such that
\[
\lim_{a \to 0}\lim_{n \to \infty}\inf_{\hat C}
\sup_{C_{0} \in \mathcal{C}_{\alpha}}
\mathbb{P}
\left\{
\|\hat C-C_{0}\|_{L^{2}(T\times T)}
\geq a\left(\frac{1}{\sqrt{n}}+ 
(n m_{A})^{-\frac{\alpha_{1} b_{1}}{1+2\alpha_{1} b_{1}}}\right)
\right\}
= 1,
\]
where $\mathcal{C}_{\alpha}$ denotes the covariance class under
consideration satisfying the source condition
$C_{0}\in \mathcal{F}^{C}_{\alpha_{1}, R_{1}}$.
\end{theorem}

\section{Proofs}\label{sec:proof}
In this section, we provide the proofs of all the results listed in the previous sections.
\subsection{Proof of Proposition~\ref{mean_alternate_estimator}}\label{subsec:prop}
By Mercer's theorem, we have
$$k(x,y) = \sum_{l} \lambda_{l} \psi_{l}(x) \psi_{l}(y),~~\forall~x,y \in T,$$
where $(\lambda_l,\psi_l)_{l\in\mathbb{N}}$ are the eigenvalue-eigenfunction pairs of $\Lambda_1$, i.e., $\Lambda_{1} = \sum_{l} \lambda_{l} \langle \cdot, \psi_{l} \rangle_{L^2(T)} \psi_{l}.$ 

Define $k^{\frac{1}{2}}(x, \cdot) := \sum_{l} \sqrt{\lambda_{l}}\psi_{l}(x) \psi_{l}(\cdot),\,x\in T$. It is easy to verify that $k^{\frac{1}{2}}(x,\cdot)\in L^2(T)$ for all $x\in T$ and $\Lambda_{1}^{\frac{1}{2}}k^{\frac{1}{2}}(x, \cdot) = \sum_{l} \lambda_{l}\psi_{l}(x) \psi_{l}(\cdot) = k(x,\cdot )$. Moreover for any $f\in L^2(T)$, $(\Lambda^{\frac{1}{2}}_1f)(x)=\langle f,k^{\frac{1}{2}}(x,\cdot)\rangle_{L^2(T)},\,x\in T$. 
Now, using these, we first claim that if $(\lambda, f)$ is an eigenvalue-eigenfunction pair for $\hat{A}_{n}$ with $\lambda>0$, then $(\lambda, \Lambda^{\frac{1}{2}}f)$ is an eigenvalue-eigenfunction pair for $\hat{S}_{n}$. To this end, we have $\hat{A}_{n} f = \lambda f$, so
\begin{equation*}
    \frac{1}{n}\sum_{i=1}^{n}w_{i}\sum_{j=1}^{m_{i}} (\Lambda_{1}^{\frac{1}{2}}f)(t_{ij}) k^{\frac{1}{2}}(\cdot, t_{ij})
    =  \lambda f.
\end{equation*}
Since $\Lambda^{\frac{1}{2}}f\in \mathcal{H}_1$ for $f\in L^2(T)$, we have 
\begin{align*}
    \begin{split}
        \hat{S}_{n}(\Lambda_{1}^{\frac{1}{2}}f)  
         &= \frac{1}{n}\sum_{i=1}^{n}w_{i}
         \sum_{j=1}^{m_{i}}(\Lambda_{1}^{\frac{1}{2}}f)(t_{ij}) k(\cdot, t_{ij})=
         \Lambda_{1}^{\frac{1}{2}}\left(\frac{1}{n}\sum_{i=1}^{n} w_{i}
         \sum_{j=1}^{m_{i}} (\Lambda_{1}^{\frac{1}{2}}f)(t_{ij}) k^{\frac{1}{2}}(\cdot, t_{ij})\right)\\
&        =  \lambda \Lambda_{1}^{\frac{1}{2}}f.
    \end{split}
\end{align*}
Therefore,  
if $\{(\lambda_{n}, f_{n})\}_{n \in \mathbb{N}}$ is the set of eigenvalue-eigenfunction pairs for $\hat{A}_{n}$, then the set of eigenvalue-eigenfunction pairs for $\hat{S}_{n}$ is given by $\{(\lambda_{n}, \Lambda_{1}^{\frac{1}{2}}f_{n})\}_{n \in \mathbb{N}}$. This implies
\begin{align*}
        g_{\lambda}(\hat{S}_{n})V & =  \sum_{k} g_{\lambda}(\lambda_{k}) \langle V, \Lambda_{1}^{\frac{1}{2}}f_{k}\rangle_{\mathcal{H}_{1}}\Lambda_{1}^{\frac{1}{2}}f_{k}\\
        &=  \Lambda_{1}^{\frac{1}{2}}\left(\sum_{k}\frac{g_{\lambda}(\lambda_{k})}{n}\sum_{i=1}^{n}w_{i}\sum_{j=1}^{m_{i}}Y_{ij}\langle k(\cdot, t_{ij}), \Lambda_{1}^{\frac{1}{2}}f_{k} \rangle_{\mathcal{H}_{1}}f_{k}\right)\\
        &=  \Lambda_{1}^{\frac{1}{2}}\left(\sum_{k}\frac{g_{\lambda}(\lambda_{k})}{n}\sum_{i=1}^{n}w_{i}\sum_{j=1}^{m_{i}}Y_{ij}\langle k^{\frac{1}{2}}(\cdot, t_{ij}), f_{k} \rangle_{L^2(T)}f_{k}\right)\\
        & =  \Lambda_{1}^{\frac{1}{2}}\left(\sum_{k}g_{\lambda}(\lambda_{k})\langle V_{1}, f_{k}\rangle_{L^2(T)}f_{k}\right)
        =  \Lambda_{1}^{\frac{1}{2}}g_{\lambda}(\hat{A}_{n})V_{1}
\end{align*}
and the result follows.
\subsection{Proof of Theorem~\ref{thm:mean}}\label{subsec:thm-mean-upper}

\underline{\textit{Case-1:} $0 < \alpha \leq \frac{1}{2}$}.
    We start with the error term
    \begin{align*}
        \|\hat{\mu}_{\lambda}-\mu_{0}\|_{L^2(T)} &=  \|\Lambda_{1}^{\frac{1}{2}}g_{\lambda}(\hat{A}_{n})V_{1}-\mu_{0}\|_{L^2(T)}\\
         &\leq  \underbrace{\|\Lambda_{1}^{\frac{1}{2}}g_{\lambda}(\hat{A}_{n})(V_{1}- \Lambda_{1}^{\frac{1}{2}}\mu_{0})\|_{L^2(T)}}_{\textit{Term-1}} + \underbrace{\|\Lambda_{1}^{\frac{1}{2}}g_{\lambda}(\hat{A}_{n})\Lambda_{1}^{\frac{1}{2}}\mu_{0}-\mu_{0}\|_{L^2(T)}}_{\textit{Term-2}}.
    \end{align*}
\textit{Bounding \textit{Term-1}:}
\begin{align*}
        \|\Lambda_{1}^{\frac{1}{2}}g_{\lambda}(\hat{A}_{n})(V_{1}- \Lambda_{1}^{\frac{1}{2}}\mu_{0})\|_{L^2(T)} &\leq  \|\Lambda_{1}^{\frac{1}{2}}(\Lambda_{1}+\lambda I)^{-\frac{1}{2}}\|_{\text{op}} \|(\Lambda_{1}+\lambda I)^{\frac{1}{2}}(\hat{A}_{n}+\lambda I)^{-\frac{1}{2}}\|^2_{\text{op}}\\
        & \qquad\times \|(\hat{A}_{n}+\lambda I)^{\frac{1}{2}}g_{\lambda}(\hat{A}_{n})(\hat{A}_{n}+\lambda I)^{\frac{1}{2}}\|_{\text{op}}\\
        & \qquad\times \|(\Lambda_{1}+\lambda I)^{-\frac{1}{2}}(V_{1}-\Lambda_{1}^{\frac{1}{2}}\mu_{0})\|_{L^2(T)}.
\end{align*}
Using Lemmas~\ref{mean_empirical_estimation}, \ref{mean_powers_constant_bound} and properties of regularization family, we get

\begin{equation*}
    \|\Lambda_{1}^{\frac{1}{2}}g_{\lambda}(\hat{A}_{n})(V_{1}- \Lambda_{1}^{\frac{1}{2}}\mu_{0})\|_{L^2(T)}\lesssim_{p} \left(\frac{1}{n^2}\sum_{i=1}^{n}w_{i}^2 m_{i}\right)^{\frac{1}{2}}\lambda^{-\frac{1}{2b}} +\left(\frac{1}{n^2}\sum_{i=1}^{n}w_{i}^2 m_{i}(m_{i}-1)\right)^{\frac{1}{2}}.
\end{equation*}
\vspace{0.2cm}
\noindent
\textit{Bounding \textit{Term-2}:} We obtain
\begin{align*}
        \|\Lambda_{1}^{\frac{1}{2}}g_{\lambda}(\hat{A}_{n})\Lambda_{1}^{\frac{1}{2}}\mu_{0}- \mu_{0}\|_{L^2(T)} & \leq \underbrace{\|\Lambda_{1}^{\frac{1}{2}}(g_{\lambda}(\hat{A}_{n})-(\hat{A}_{n}+\lambda I)^{-1})\Lambda_{1}^{\frac{1}{2}}\mu_{0}\|_{L^2(T)}}_{\text{Term-2a}} \\
        & \qquad+ \underbrace{\|\Lambda_{1}^{\frac{1}{2}}((\hat{A}_{n}+\lambda I)^{-1}-(\Lambda_{1}+\lambda I)^{-1})\Lambda_{1}^{\frac{1}{2}}\mu_{0}\|_{L^2(T)}}_{\text{Term-2b}}\\
        &\qquad+ \underbrace{\|(\Lambda_{1}^{\frac{1}{2}}(\Lambda_{1}+\lambda I)^{-1}\Lambda_{1}^{\frac{1}{2}}-I)\mu_{0}\|_{L^2(T)}}_{\text{Term-2c}}.
\end{align*}
\noindent
\underline{Bound for Term-2a:}
\begin{align*}
        &\|\Lambda_{1}^{\frac{1}{2}}(g_{\lambda}(\hat{A}_{n})-(\hat{A}_{n}+\lambda I)^{-1})\Lambda_{1}^{\frac{1}{2}}\mu_{0}\|_{L^2(T)} =  \|\Lambda_{1}^{\frac{1}{2}}(-r_{\lambda}(\hat{A}_{n})+\lambda g_{\lambda}(\hat{A}_{n}))(\hat{A}_{n}+\lambda I)^{-1}\Lambda_{1}^{\frac{1}{2}}\mu_{0}\|_{L^2(T)}\\
        &\leq  \|\Lambda_{1}^{\frac{1}{2}}(\Lambda_{1}+\lambda I)^{-\frac{1}{2}}\|_{\text{op}}\|(\Lambda_{1}+\lambda I)^{\frac{1}{2}}(\hat{A}_{n}+\lambda I)^{-\frac{1}{2}}\|_{\text{op}} \|(\hat{A}_{n}+\lambda I)^{\frac{1}{2}}(-r_{\lambda}(\hat{A}_{n})+\lambda g_{\lambda}(\hat{A}_{n}))\|_{\text{op}}\\
        & \qquad\times\|(\hat{A}_{n}+\lambda I)^{-1}\Lambda_{1}^{\frac{1}{2}}\mu_{0}\|_{L^2(T)}\\
        &\stackrel{(*)} \lesssim_{p}  \sqrt{\lambda}\|(\hat{A}_{n}+\lambda I)^{-1}\Lambda_{1}^{\frac{1}{2}}\mu_{0}\|_{L^2(T)}\\
        &\lesssim \|(\Lambda_{1}+\lambda I)^{-\frac{1}{2}}(\hat{A}_{n}-\Lambda_{1}) (\Lambda_{1}+\lambda I)^{-1}\Lambda_{1}^{\frac{1}{2}}\mu_{0}\|_{L^2(T)} + \sqrt{\lambda}\|(\Lambda_{1}+\lambda I)^{-1}\Lambda_{1}^{\frac{1}{2}}\Lambda_{1}^{\alpha}h\|_{L^2(T)},
\end{align*}
where $(*)$ follows from Lemma~\ref{mean_powers_constant_bound} and properties of the regularization family. Next, from Lemma~\ref{mean_empirical_operator}, we get
\begin{equation*}
\begin{split}
    \|\Lambda_{1}^{\frac{1}{2}}(g_{\lambda}(\hat{A}_{n})-(\hat{A}_{n}+\lambda I)^{-1})\Lambda_{1}^{\frac{1}{2}}\mu_{0}\|_{L^2(T)} \lesssim_{p} &\left(\frac{1}{n^2}\sum_{i=1}^{n}w_{i}^2 m_{i}\right)^{\frac{1}{2}} \lambda^{-\frac{1}{2b}}+\sqrt{\lambda}\sup_{i}\left|\frac{\lambda_{i}^{\alpha+\frac{1}{2}}}{\lambda_{i}+\lambda}\right|\\
    \leq &  \left(\frac{1}{n^2}\sum_{i=1}^{n}w_{i}^2 m_{i}\right)^{\frac{1}{2}} \lambda^{-\frac{1}{2b}} + \lambda^{\alpha}.
\end{split}
\end{equation*}

\noindent
\underline{Bound for Term-2b:}
\begin{align*}
    \|\Lambda_{1}^{\frac{1}{2}}((\hat{A}_{n}+\lambda I)^{-1}-(\Lambda_{1}+\lambda I)^{-1})\Lambda_{1}^{\frac{1}{2}}\mu_{0}\|_{L^2(T)}
    & =  \|\Lambda_{1}^{\frac{1}{2}}(\hat{A}_{n}+\lambda I)^{-1}(\Lambda_{1}-\hat{A}_{n})(\Lambda_{1}+\lambda I)^{-1}\Lambda_{1}^{\frac{1}{2}}\mu_{0}\|_{L^2(T)}\\
    & \stackrel{(\dagger)}{\lesssim_{p}} \|(\Lambda_{1}+\lambda I)^{-\frac{1}{2}}(\Lambda_{1}-\hat{A}_{n})(\Lambda_{1}+\lambda I)^{-1}\Lambda_{1}^{\frac{1}{2}}\mu_{0}\|_{L^2(T)}\\
    &\lesssim_{p}  \left(\frac{1}{n^2}\sum_{i=1}^{n}w_{i}^2 m_{i}\right)^{\frac{1}{2}}\lambda^{-\frac{1}{2b}},
\end{align*}
where we employed Lemma~\ref{mean_powers_constant_bound} in $(\dagger)$ and the last step follows from Lemma~\ref{mean_empirical_operator}.\\

\noindent
\underline{Bound for Term-2c:}
\begin{align*}
&\|(\Lambda_{1}^{\frac{1}{2}}(\Lambda_{1}+\lambda I)^{-1}\Lambda_{1}^{\frac{1}{2}}-I)\mu_{0}\|_{L^2(T)} = \|(\Lambda_{1}^{\frac{1}{2}}(\Lambda_{1}+\lambda I)^{-1}\Lambda_{1}^{\frac{1}{2}}-I)\Lambda_{1}^{\alpha}h\|_{L^2(T)}\\
&\lesssim  \sup_{i}\left|\left(\frac{\lambda_{i}}{\lambda_{i}+\lambda}-1\right)\lambda_{i}^{\alpha}\right| = \sup_{i}\left|\left(\frac{\lambda\lambda_{i}^{\alpha}}{\lambda_{i}+\lambda}\right)\right|
\leq  \lambda^{\alpha}.
\end{align*}
\noindent
Combining Terms 2a-2c, we obtain
\begin{equation*}
    \|\Lambda_{1}^{\frac{1}{2}}g_{\lambda}(\hat{A}_{n})\Lambda_{1}^{\frac{1}{2}}\mu_{0}- \mu_{0}\|_{L^2(T)} \lesssim  \left(\frac{1}{n^2}\sum_{i=1}^{n}w_{i}^2 m_{i}\right)^{\frac{1}{2}}\lambda^{-\frac{1}{2b}} + \lambda^{\alpha},
\end{equation*}
which combined with Term-1 yields
\begin{equation*}
    \|\hat{\mu}_{\lambda}-\mu_{0}\|_{L^2(T)} \lesssim \left(\frac{1}{n^2}\sum_{i=1}^{n}w_{i}^2 m_{i}(m_{i}-1)\right)^{\frac{1}{2}}+\left(\frac{1}{n^2}\sum_{i=1}^{n}w_{i}^2 m_{i}\right)^{\frac{1}{2}}\lambda^{-\frac{1}{2b}} + \lambda^{\alpha}.
\end{equation*}

\noindent
\underline{\textit{Case-2:} $\alpha \geq \frac{1}{2}$} Consider
\begin{align*}
        &\|\hat{\mu}_{\lambda}-\mu_{0}\|_{L^2(T)}  =  \|\Lambda_{1}^{\frac{1}{2}}g_{\lambda}(\hat{A}_{n})V_{1}-\mu_{0}\|_{L^2(T)}\\
        &\leq  \|\Lambda_{1}^{\frac{1}{2}}g_{\lambda}(\hat{A}_{n})(V_{1}-\hat{A}_{n}\Lambda_{1}^{\alpha-\frac{1}{2}}h)\|_{L^2(T)} + \|\Lambda_{1}^{\frac{1}{2}}r_{\lambda}(\hat{A}_{n})\Lambda_{1}^{\alpha-\frac{1}{2}}h\|_{L^2(T)}\\
        & \leq  \|\Lambda_{1}^{\frac{1}{2}}(\Lambda_{1}+\lambda I)^{-\frac{1}{2}}\|_{\text{op}}\|(\Lambda_{1}+\lambda I)^{\frac{1}{2}}(\hat{A}_{n}+\lambda I)^{-\frac{1}{2}}\|^2_{\text{op}}\\
        & \qquad\times\|(\hat{A}_{n}+\lambda I)^{\frac{1}{2}}g_{\lambda}(\hat{A}_{n})(\hat{A}_{n}+\lambda I)^{\frac{1}{2}}\|_{\text{op}}\|(\Lambda_{1}+\lambda I)^{-\frac{1}{2}}(V_{1}-\hat{A}_{n}\Lambda_{1}^{\alpha-\frac{1}{2}}h)\|_{L^2(T)}\\
        & \qquad\qquad+ \|\Lambda_{1}^{\frac{1}{2}}r_{\lambda}(\hat{A}_{n})\Lambda_{1}^{\alpha-\frac{1}{2}}h\|_{L^2(T)}\\
        &\stackrel{(*)}\lesssim_{p}  \left(\frac{1}{n^2}\sum_{i=1}^{n}w_{i}^2 m_{i}(m_{i}-1)\right)^{\frac{1}{2}}+\left(\frac{1}{n^2}\sum_{i=1}^{n}w_{i}^2 m_{i}\right)^{\frac{1}{2}}\lambda^{-\frac{1}{2b}}\\ 
        & +\|(\hat{A}_{n}+\lambda I)^{\frac{1}{2}}r_{\lambda}(\hat{A}_{n})(\hat{A}_{n}+\lambda I)^{\nu-\frac{1}{2}}\|_{\text{op}}\|(\hat{A}_{n}+\lambda I)^{-(\nu-\frac{1}{2})}\Lambda_{1}^{\alpha-\frac{1}{2}}h\|_{L^2(T)}\\
        &\lesssim  \left(\frac{1}{n^2}\sum_{i=1}^{n}w_{i}^2 m_{i}(m_{i}-1)\right)^{\frac{1}{2}}+\left(\frac{1}{n^2}\sum_{i=1}^{n}w_{i}^2 m_{i}\right)^{\frac{1}{2}}\lambda^{-\frac{1}{2b}} + \lambda^{\nu}\|(\hat{A}_{n}+\lambda I)^{-(\nu-\frac{1}{2})}\Lambda_{1}^{\alpha-\frac{1}{2}}h\|_{L^2(T)},
\end{align*}
where $(*)$ follows from Lemmas~\ref{mean_empirical_estimation},~\ref{mean_powers_constant_bound} and properties of the regularization family. To bound the last term, we use the qualification property of the regularization family. Let us define $q = \nu-\frac{1}{2}$ and $\lfloor q \rfloor $ is the greatest integer value of $q$. Now consider the third term
\begin{align*}
        &\|(\hat{A}_{n}+\lambda I)^{-(\nu-\frac{1}{2})}\Lambda_{1}^{\alpha-\frac{1}{2}}h\|_{L^2(T)}  = \|(\hat{A}_{n}+\lambda I)^{-(q-\lfloor q \rfloor)}(\hat{A}_{n}+\lambda I)^{-\lfloor q \rfloor}\Lambda_{1}^{\alpha-\frac{1}{2}}h\|_{L^2(T)}\\
         &\leq  \|(\hat{A}_{n}+\lambda I)^{-(q-\lfloor q \rfloor)}(\Lambda_{1}+\lambda I)^{(q-\lfloor q \rfloor)}\|_{\text{op}}\\
        & \qquad\qquad\times \|(\Lambda_{1}+\lambda I)^{-(q-\lfloor q \rfloor)}(\hat{A}_{n}+\lambda I)^{-\lfloor q \rfloor}\Lambda_{1}^{\alpha-\frac{1}{2}}h\|_{L^2(T)} \\
        &\lesssim_{p}   \|((\Lambda_{1}+\lambda I)^{-(q-\lfloor q \rfloor)}(\hat{A}_{n}+\lambda I)^{-\lfloor q \rfloor}- (\Lambda_{1}+\lambda I)^{-q})\Lambda_{1}^{\alpha-\frac{1}{2}}h\|_{L^2(T)}\\
        & \qquad\qquad+ \|(\Lambda_{1}+\lambda I)^{-q}\Lambda_{1}^{\alpha-\frac{1}{2}}h\|_{L^2(T)}\\
        & \stackrel{(*)}{\lesssim}
          \frac{1}{\lambda^{q}}\|(\Lambda_{1}+\lambda I)^{-\frac{1}{2}}(\Lambda_{1}-\hat{A}_{n})(\Lambda_{1}+\lambda I)^{-\frac{1}{2}}\|_{\text{op}}+ \sup_{i} \left|\frac{\lambda_{i}^{\alpha-\frac{1}{2}}}{(\lambda_{i}+\lambda)^{\nu-\frac{1}{2}}}\right|\\
        & \lesssim \frac{1}{\lambda^{q}}\|(\Lambda_{1}+\lambda I)^{-\frac{1}{2}}(\Lambda_{1}-\hat{A}_{n})(\Lambda_{1}+\lambda I)^{-\frac{1}{2}}\|_{\text{op}}+ \sup_{i}\left| \frac{i^{-b(\alpha-\frac{1}{2})}}{(i^{-b}+\lambda)^{\nu-\frac{1}{2}}}\right|\\         
        & \lesssim \frac{1}{\lambda^{q}}\|(\Lambda_{1}+\lambda I)^{-\frac{1}{2}}(\Lambda_{1}-\hat{A}_{n})(\Lambda_{1}+\lambda I)^{-\frac{1}{2}}\|_{\text{op}} + \lambda^{\alpha-\nu},
\end{align*}
where $(*)$ follows from Lemma~\ref{reducing_power} and for last step we use Lemma~\ref{sup_bound} with $\alpha \leq \nu$. Now by using Lemma~\ref{mean_empirical_operator}, we get
\begin{equation*}
    \|\hat{\mu}_{\lambda}-\mu_{0}\|_{L^2(T)} \lesssim \left(\frac{1}{n^2}\sum_{i=1}^{n}w_{i}^2 m_{i}(m_{i}-1)\right)^{\frac{1}{2}}+\left(\frac{1}{n^2}\sum_{i=1}^{n}w_{i}^2 m_{i}\right)^{\frac{1}{2}}\lambda^{-\frac{1}{2b}}+ \lambda^{\alpha}.
\end{equation*}

\noindent
Note that if $\alpha > \nu$ then
\begin{equation*}
    \|\hat{\mu}_{\lambda}-\mu_{0}\|_{L^2(T)} \lesssim \left(\frac{1}{n^2}\sum_{i=1}^{n}w_{i}^2 m_{i}(m_{i}-1)\right)^{\frac{1}{2}}+\left(\frac{1}{n^2}\sum_{i=1}^{n}w_{i}^2 m_{i}\right)^{\frac{1}{2}}\lambda^{-\frac{1}{2b}}+ \lambda^{\nu}.
\end{equation*}

\begin{remark}
    Note that Assumption~\ref{mean_embedding}, together with the condition $\sup_{t \in T}\mathbb{E}[X^2(t)] < \infty$, can be replaced by the assumptions $\sup_{k}\|\psi_{k}\|_{\infty} < \infty$ and $\mathbb{E}\|X\|^2_{L^2} < \infty$ when $\frac{1}{2b} \leq \alpha \le \frac{1}{2}$, without affecting the convergence rates of $\hat{\mu}_{\eta}$. Adopting this alternative set of assumptions requires following adjustments to the arguments in Lemma~\ref{mean_empirical_estimation} and Lemma~\ref{mean_empirical_operator}. For all $\frac{1}{2b} \leq \alpha \le \frac{1}{2}$,
    \begin{equation*}
        \begin{split}
        \sum_{\beta}\frac{\phi_{\beta}^2(t)}{(\lambda_{\beta}+\lambda)} = \sum_{\beta}\frac{\lambda_\beta\psi_{\beta}^2(t)}{(\lambda_{\beta}+\lambda)} & =\sum_{\beta}\frac{\lambda_{\beta}^{2 \alpha} \lambda_{\beta}^{1-2 \alpha}\psi_{\beta}^2(t)}{(\lambda_{\beta}+\lambda)^{2 \alpha} (\lambda_{\beta}+\lambda)^{1-2 \alpha}}\\
        &\leq  \frac{\sup_{\beta}\|\psi_{\beta}\|_{\infty}}{\lambda^{2 \alpha}}\sum_{\beta}\lambda_{\beta}^{2 \alpha} \lesssim \frac{1}{\lambda^{2\alpha}}.
        \end{split}
    \end{equation*}
\end{remark}

\begin{remark}
For $\alpha\geq \frac{1}{2}$, by the reproducing property of $\mathcal H$ and the relation between the $\mathcal H$- and $L^2(T)$-norm errors, we have

$$
\|\hat{\mu}_{\lambda}-\mu_{0}\|_{\infty}
\lesssim
\|\hat{\mu}_{\lambda}-\mu_{0}\|_{\mathcal H}
\lesssim
\lambda^{-\frac12}
\|\hat{\mu}_{\lambda}-\mu_{0}\|_{L^2(T)}.
$$

\noindent
Hence, with the optimal choice of $\lambda$ in Theorem~\ref{thm:mean}, under an appropriate relation between the $\{m_i\}$'s and $n$, we obtain

$$
\|\hat{\mu}_{\lambda}-\mu_{0}\|_{\infty}
\lesssim
\|\hat{\mu}_{\lambda}-\mu_{0}\|_{\mathcal H}
=O_p(1).
$$

\noindent
On the other hand, when $0<\alpha<\frac12$, under Assumption~\ref{mean_embedding}, we have

$$
\|\hat{\mu}_{\lambda}-\mu_{0}\|_{\infty}
\lesssim
\|\hat{\mu}_{\lambda}-\mu_{0}\|_{[\mathcal H_1]^\alpha}
\lesssim
\lambda^{-\alpha}
\|\hat{\mu}_{\lambda}-\mu_{0}\|_{L^2(T)},
$$

which similarly yields $\|\hat{\mu}_{\lambda}-\mu_{0}\|_{\infty}=O_p(1)$ under an appropriate relation between the $\{m_i\}$'s and $n$.

For instance, in the balanced design $m_i=m$ for all $1\leq i\leq n$, one may take

$$
n^{\frac{1}{2\alpha b}}\lesssim m\lesssim n^{2\alpha-1+\frac1b},
\qquad \alpha\geq\frac12,
$$

whereas for $0<\alpha<\frac12$, it suffices to take

$$
m=O\left(n^{\frac{1}{2\alpha b}}\right).
$$

\end{remark}

\subsection{Proof of Theorem~\ref{thm:mean-lower}}\label{subsec:thm-mean-lower}
Similar to \cite[Theorem~$6$]{cai2010nonparametric} observe that, taking $m_{i} =m=1,~~\forall i$, i.e., constant observations gives us $(mn)^{-\frac{\alpha b}{1+ 2 \alpha b}} = n^{-\frac{\alpha b}{1+ 2 \alpha b}} \geq n^{-\frac{1}{2}}$. So we just need to prove that 
\begin{equation*}
    \lim_{a \to 0} \lim_{n \to \infty} \inf_{\hat{\mu}} \sup_{\mu_{0} \in \mathcal{F}_{\alpha,R}} \mathbb{P}\left\{\|\hat{\mu}-\mu_{0}\|_{L^2(T)} \geq a  (mn)^{-\frac{\alpha b}{1+2 \alpha b}}\right\} =1,
\end{equation*}
where $m = \frac{\sum_{i=1}^{n}m_{i}}{n}$. Let $M = c_{0}(nm)^{\frac{1}{1+ 2\alpha b}}$ for some constant $c_{0}>0$. We assume that $t_{ij} \sim \text{Unif}(T),~\epsilon_{ij}\sim \mathcal{N}(0, \sigma_{\epsilon}^2)$ and $X_{i}(t)= \mu(t)+ Z_{i}(t)$ where $Z_{i}$ is i.i.d. Gaussian process with zero mean and uniformly bounded variance $\sigma^2$. 
For $\theta \in \{0,1\}^{M}$, we define 
    \begin{equation*}
        g_{\theta} = \sum_{k=1}^{M} \theta_{k}M^{-\frac{1}{2}}\psi_{k+M}.
    \end{equation*}
Then, clearly $\mu_{\theta}:= \Lambda_{1}^{\alpha}g_{\theta}\in  \mathcal{R}(\Lambda_{1}^{\alpha}),~\alpha >0$.\\

\noindent
With the use of Lemma~\ref{VGbound}, we have  $\{\theta^{i} \in \{0,1\}^M,~0 \leq i \leq N\}$ such that $\theta^{0} = (0,\cdots,0)$  and Hamming distance between any two elements from the set is greater than $\frac{M}{8}$ with $N \geq 2^{\frac{M}{8}}$. Let $P_{\theta}$ be the joint probability distribution of $\{(t_{ij},Y_{ij}),~1\leq i \leq n,~ 1\leq j \leq m_{i}\}$ for $\mu= \mu_{\theta}$. Then the Kullback-divergence between $P_\theta$ and $P_{\theta'}$ is given by
$$\mathcal{K}(P_{\theta}, P_{\theta'}) = \mathbb{E}_{T}[\mathcal{K}((P_{\theta}|T),(P_{\theta'}|T))].$$

As we have $Y_{ij} = \mu(t_{ij})+Z_{i}(t_{ij}) + \epsilon_{ij}$, we can see that 
for vector $Y_{i} = (Y_{i1},\ldots,Y_{im_{i}})$, $\mathbb{E}[Y_{i}] = \mu_{i} = (\mu(t_{i1},\ldots, t_{im_{i}}))$ and the covariance matrix for $Y_{i}$ is $\Sigma_{i} = \{\text{Cov}(Z_{i}(t_{ij}),Z_{i}(t_{ik}))\}_{j, k} + \sigma_{\epsilon}^2 I_{m_{i}}$.\\

\noindent
As all the curves are independent, we can see that
\begin{align*}
\begin{split}
    \mathcal{K}(P_{\theta}, P_{\theta'})& =  \frac{1}{2}\sum_{i=1}^{n}
 \mathbb{E}_{T}\left[(\mu_{i,\theta}-\mu_{i,\theta'})^{\top}\Sigma_{i}^{-1}(\mu_{i,\theta}-\mu_{i,\theta'})\right]
    \end{split}
\end{align*}
where $\mu_{i,\theta}$ and $\mu_{i,\theta'}$ is $\mathbb{E}[Y_{i}]$ when $\mu = \mu_{\theta}$ and $\mu_{\theta'}$ respectively. Moreover, note that
$\Sigma_{i} \succeq \sigma_{\epsilon}^2 I_{m_{i}}$ and hence we get
\begin{align*}
\begin{split}
    \mathcal{K}(P_{\theta}, P_{\theta'})& =  \frac{1}{2}\sum_{i=1}^{n}
 \mathbb{E}_{T}\left[(\mu_{i,\theta}-\mu_{i,\theta'})^{\top}\Sigma_{i}^{-1}(\mu_{i,\theta}-\mu_{i,\theta'})\right]\\
 & \leq \frac{1}{2\sigma_{\epsilon}^2}\sum_{i=1}^{n}\mathbb{E}_{T}\sum_{j=1}^{m_{i}}(\mu_{i,\theta}(t_{ij})-\mu_{i,\theta'}(t_{ij}))^2\\
 & = \frac{1}{2\sigma_{\epsilon}^2}\sum_{i=1}^{n}m_{i}\|\mu_{\theta}-\mu_{\theta'}\|^2_{L^2(T)} = \frac{nm}{2 \sigma_{\epsilon}^2} \|\mu_{\theta}-\mu_{\theta'}\|^2_{L^2(T)}.
    \end{split}
\end{align*}

\noindent
As
\begin{align*}
    \begin{split}
        \|\mu_{\theta}-\mu_{\theta'}\|^2_{L^2(T)} = & \sum_{k=1}^{M}(\theta_{k}-\theta_{k}')^2 M^{-1} \lambda_{k+M}^{2 \alpha}\\
        \leq & c_{1} M^{-1} M^{-2\alpha b} H(\theta,\theta') \leq c_{1} M^{-2 \alpha b},
    \end{split}
\end{align*}
 we get

\begin{align*}
    \mathcal{K}(P_{\theta}, P_{\theta'}) \leq & \frac{nm}{2 \sigma_{\epsilon}^2}c_{1} M^{-2\alpha b} = \frac{c_{1}}{2\sigma_{\epsilon}^2} M^{-2 \alpha b} \left(\frac{M}{c_{0}}\right)^{1+2\alpha b}\\
    = & \frac{c_{1}}{2 \sigma_{\epsilon}^2 c_{0}^{1+2\alpha b}} M \leq \frac{8}{\log 2} \frac{c_{1}}{2 \sigma_{\epsilon}^2 c_{0}^{1+2\alpha b}} \log N \leq \log N,
\end{align*}
where $c_{0}$ has been chosen appropriately large to keep constant less than $1$.\\

\noindent
Next, we consider lower bound on the error term
\begin{align*}
    \|\mu_{\theta}-\mu_{\theta'}\|^2_{L^2(T)} = & \sum_{k=1}^{M}(\theta_{k}-\theta_{k}')^2 M^{-1} \lambda_{k+M}^{2 \alpha}\\
    \gtrsim & M^{-2 \alpha b } M^{-1} H(\theta, \theta') \geq \frac{M^{-2 \alpha b}}{8} \gtrsim (nm)^{-\frac{2 \alpha b}{1 +2 \alpha b}}.
\end{align*}

\noindent
Therefore, the result follows from \cite[Theorem~$2.5$]{tsyback2009lb}.

\subsection{Proof of Theorem~\ref{thm:covariance}}\label{subsec:thm-cov-upper}
We start with the error term
\begin{align*}
    \begin{split}    \|\hat{C}_{\eta,\lambda}-C_{0}\|_{L^2(T \times T)} & =  \|\hat{C}_{\eta, \lambda}- C^{\text{true}}_{\lambda} + C^{\text{true}}_{\lambda} - C_{0} \|_{L^2(T \times T)}\\
        & \leq  \underbrace{\|\hat{C}_{\eta,\lambda}- C^{\text{true}}_{\lambda}\|_{L^2(T \times T)}}_{\textit{Term-3}} + \underbrace{\|C^{\text{true}}_{\lambda} - C_{0} \|_{L^2(T \times T)}}_{\textit{Term-4}}
    \end{split}.
\end{align*}

\noindent
\textit{Bounding \textit{Term-3}:}
\begin{align*}
    \begin{split}
        &\|\hat{C}_{\eta, \lambda}- C^{\text{true}}_{\lambda}\|_{L^2(T \times T)} = \|\Lambda^{\frac{1}{2}}_{2}g_{\lambda}(\hat{T}_{n})(O_{2}-O_{3})\|_{L^2(T \times T)}\\
        & \leq \|(\Lambda_{2}+\lambda I)^{\frac{1}{2}}(\hat{T}_{n}+\lambda I)^{-\frac{1}{2}}\|_{\text{op}}
         \|(\hat{T}_{n}+\lambda I)^{\frac{1}{2}}g_{\lambda}(\hat{T}_{n})(\hat{T}_{n}+\lambda I)^{\frac{1}{2}}\|_{\text{op}}\\
        & \qquad\qquad\times\|\Lambda^{\frac{1}{2}}_{2}(\Lambda_{2}+\lambda I)^{-\frac{1}{2}}\|_{\text{op}}\|(\hat{T}_{n}+\lambda I)^{-\frac{1}{2}}(O_{2}-O_{3})\|_{L^2(T \times T)}\\
        & \lesssim_{p} \|(\hat{T}_{n}+\lambda I)^{-\frac{1}{2}}(O_{2}-O_{3})\|_{L^2(T \times T)},
    \end{split}
\end{align*}
where the last step follows from Lemma~\ref{covariance_power_bound} and properties of the regularization family.\\

\noindent
Observe that for all $i \in [n], j \in [m_{i}]$
\begin{align*}
\begin{split}
   & O_{2}- O_{3} \\ &= \frac{1}{n}\sum_{i=1}^{n}v_{i}\sum_{1 \leq j \neq k \leq m_{i}}[(Y_{ij}-\hat{\mu}_{\eta}(t_{ij}))(Y_{ik}-\hat{\mu}_{\eta}(t_{ik}))-(Y_{ij}-\mu_{0}(t_{ij}))(Y_{ik}-\mu_{0}(t_{ik}))]k^{\frac{1}{2}}(t_{ij},\cdot)k^{\frac{1}{2}}(t_{ik},\cdot)\\
    & =  \frac{1}{n}\sum_{i=1}^{n}v_{i}\sum_{1 \leq j \neq k \leq m_{i}}[(A_{ij}+\Delta_{ij})(A_{ik}+\Delta_{ik})-A_{ij}A_{ik}]k^{\frac{1}{2}}(t_{ij},\cdot)k^{\frac{1}{2}}(t_{ik},\cdot)\\
    & =  \frac{1}{n}\sum_{i=1}^{n}v_{i}\sum_{1 \leq j \neq k \leq m_{i}}[A_{ij}\Delta_{ik}+A_{ik}\Delta_{ij}+\Delta_{ij}\Delta_{ik}]k^{\frac{1}{2}}(t_{ij},\cdot)k^{\frac{1}{2}}(t_{ik},\cdot)\\
    &=  V_{1} + V_{2} + V_{3},
\end{split}
\end{align*}
where $A_{ij} = Y_{ij}-\mu_{0}(t_{ij})$ and $\Delta_{ij} = \mu_{0}(t_{ij})-\hat{\mu}_{\eta}(t_{ij}),~\forall~1\leq i\leq n,~ 1\leq j \leq m_{i}$. From the last step it follows that $$V_{1} = \frac{1}{n}\sum_{i=1}^{n}v_{i}\sum_{1 \leq j \neq k \leq m_{i}}A_{ij}\Delta_{ik}k^{\frac{1}{2}}(t_{ij},\cdot)k^{\frac{1}{2}}(t_{ik},\cdot)~~,V_{2} = \frac{1}{n}\sum_{i=1}^{n}v_{i}\sum_{1 \leq j \neq k \leq m_{i}}A_{ik}\Delta_{ij}k^{\frac{1}{2}}(t_{ij},\cdot)k^{\frac{1}{2}}(t_{ik},\cdot),$$ and $$V_{3} = \frac{1}{n}\sum_{i=1}^{n}v_{i}\sum_{1 \leq j \neq k \leq m_{i}}\Delta_{ij}\Delta_{ik}k^{\frac{1}{2}}(t_{ij},\cdot)k^{\frac{1}{2}}(t_{ik},\cdot).$$
Therefore,
\begin{align*}
    \begin{split}
\|\hat{C}_{\eta, \lambda}- C^{\text{true}}_{\lambda}\|_{L^2(T \times T)}
& \lesssim_{p} \|(\hat{T}_{n}+\lambda I)^{-\frac{1}{2}}V_{1}\|_{L^2(T \times T)} + \|(\hat{T}_{n}+\lambda I)^{-\frac{1}{2}}V_{1}\|_{L^2(T \times T)}\\
        &\qquad\qquad + \|(\hat{T}_{n}+\lambda I)^{-\frac{1}{2}}V_{3}\|_{L^2(T \times T)}\\
        & \lesssim_{p}  \underbrace{\|(\Lambda_{2}+\lambda I)^{-\frac{1}{2}}V_{1}\|_{L^2(T \times T)}}_{\text{Term-$3a$}}+\underbrace{\|(\Lambda_{2}+\lambda I)^{-\frac{1}{2}}V_{1}\|_{L^2(T \times T)}}_{\text{Term-$3b$}}\\
        & \qquad\qquad+ \underbrace{\|(\Lambda_{2}+\lambda I)^{-\frac{1}{2}}V_{3}\|_{L^2(T \times T)}}_{\text{Term-$3c$}}.
    \end{split}
\end{align*}
We now bound Term-$3a$, Term-$3b$ and Term-$3c$ as follows.\vspace{1mm}\\
\noindent
\underline{\text{Term-$3a$:}}
\begin{equation*}
    \begin{split}
     \|(\Lambda_{2}+\lambda I)^{-\frac{1}{2}}V_{1}\|_{L^2(T \times T)} = \left\|\frac{1}{n}\sum_{i=1}^{n}v_{i}(\Lambda_{2}+\lambda I)^{-\frac{1}{2}}U_{i}\right\|,
    \end{split}
\end{equation*}
where $U_{i} = \sum_{1 \leq j\neq k \leq m_{i}}A_{ij}\Delta_{ik} k^{\frac{1}{2}}(t_{ij},\cdot)k^{\frac{1}{2}}(t_{ik},\cdot)$. So applying Chebyshev's inequality gives us that for any $t>0$,
\begin{equation*}
    \mathbb{P}\{\|(\Lambda_{2}+\lambda I)^{-\frac{1}{2}}V_{1}\|_{L^2(T \times T)} \geq t\} \leq \frac{\mathbb{E}\|\frac{1}{n}\sum_{i=1}^{n}v_{i}(\Lambda_{2}+\lambda I)^{-\frac{1}{2}}U_{i}\|_{L^2(T \times T)}^2}{t^2}.
\end{equation*}
As the construction of $\hat{\mu}_{\eta}$ has been done using a independent first subsample and $U_{i}'$s are constructed second subsample. Then, conditional on $\hat{\mu}_{\eta}$, we have $\mathbb{E}[U_{i}] = 0$ and $U_{i}'$s are conditional independent.
Hence $\mathbb{E}\|\frac{1}{n}\sum_{i=1}^{n}v_{i}(\Lambda_{2}+\lambda I)^{-\frac{1}{2}}U_{i}\|^2_{L^2(T \times T)} = \frac{1}{n^2}\sum_{i=1}^{n}v_{i}^2\mathbb{E}\|(\Lambda_{2}+\lambda I)^{-\frac{1}{2}}U_{i}\|^2_{L^2(T \times T)}$. So we consider
\begin{align*}
    \begin{split}
        &\mathbb{E}\|(\Lambda_{2}+\lambda I)^{-\frac{1}{2}}U_{i}\|^2_{L^2(T \times T)} =  \sum_{\beta,\beta'}\frac{\mathbb{E}\langle U_{i}, \psi_{\beta}\otimes \psi_{\beta'}\rangle^2_{L^2(T \times T)}}{(\tau_{\beta}\tau_{\beta'}+\lambda)}
         = \sum_{\beta,\beta'} \frac{\text{Var}\langle U_{i}, \psi_{\beta}\otimes \psi_{\beta'}\rangle_{L^2(T \times T)}}{(\tau_{\beta}\tau_{\beta'}+\lambda)}\\
        = & \sum_{\beta,\beta'} \frac{\text{Var}_{T}[\mathbb{E}_{X|T}\langle U_{i}, \psi_{\beta}\otimes \psi_{\beta'}\rangle_{L^2(T \times T)}] + \mathbb{E}_{T}[\text{Var}_{X|T}\langle U_{i}, \psi_{\beta}\otimes \psi_{\beta'}\rangle_{L^2(T \times T)}]}{(\tau_{\beta}\tau_{\beta'}+\lambda)}\\
        = & \sum_{\beta,\beta'} \frac{\mathbb{E}_{T}[\text{Var}_{X|T}\langle U_{i}, \psi_{\beta}\otimes \psi_{\beta'}\rangle_{L^2(T \times T)}]}{(\tau_{\beta}\tau_{\beta'}+\lambda)}.
    \end{split}
\end{align*}
To bound $\mathbb{E}_{T}[\text{Var}_{X|T}\langle U_{i}, \psi_{\beta}\otimes \psi_{\beta'}\rangle_{L^2(T \times T)}]$, we see that
\begin{align*}
    \begin{split}
         \mathbb{E}_{T}[\text{Var}_{X|T}\langle U_{i}, \psi_{\beta}\otimes \psi_{\beta'}\rangle_{L^2(T \times T)}]
        &=  \mathbb{E}_{T}\left[\text{Var}_{X|T}\left[\sum_{1 \leq j \neq k \leq m_{i}}(X_{i}(t_{ij})+\epsilon_{ij}-\mu_{0}(t_{ij}))\Delta_{ik}\phi_{\beta}(t_{ij})\phi_{\beta'}(t_{ik})\right]\right]\\
        &\leq  \mathbb{E}\left[\sum_{1 \leq j \neq k \leq m_{i}}\sum_{1 \leq j' \neq k' \leq m_{i}}(X_{i}(t_{ij})+\epsilon_{ij}-\mu_{0}(t_{ij}))\Delta_{ik}\phi_{\beta}(t_{ij})\phi_{\beta'}(t_{ik})\right.\\
        &\qquad\qquad\left.\times(X_{i}(t_{ij'})+\epsilon_{ij'}-\mu_{0}(t_{ij'}))\Delta_{ik'}\phi_{\beta}(t_{ij'})\phi_{\beta'}(t_{ik'})\right].
    \end{split}
\end{align*}

\noindent
We bound the above term by dividing into cases over indices.\\

\noindent
\textit{\underline{Case-1: $(\{j,k\}\cap\{j',k'\} = \emptyset)$}}
\begin{align*}
    \begin{split}
        &\mathbb{E}[(X_{i}(t_{ij})+\epsilon_{ij}-\mu_{0}(t_{ij}))\Delta_{ik}\phi_{\beta}(t_{ij})\phi_{\beta'}(t_{ik})(X_{i}(t_{ij'})+\epsilon_{ij'}-\mu_{0}(t_{ij'}))\Delta_{ik'}\phi_{\beta}(t_{ij'})\phi_{\beta'}(t_{ik'})]\\
    &=  \mathbb{E}_{X}\langle X_{i}-\mu_{0}, \phi_{\beta}\rangle^2_{L^2(T)}\langle(\mu_{0}-\hat{\mu}_{\eta}), \phi_{\beta'}\rangle^2_{L^2(T)}
    \end{split}
\end{align*}

\noindent
\textit{\underline{Case-2: $((j,k) = (j',k'))$}}
\begin{align*}
    \begin{split}
        &\mathbb{E}[(X_{i}(t_{ij})+\epsilon_{ij}-\mu_{0}(t_{ij}))\Delta_{ik}\phi_{\beta}(t_{ij})\phi_{\beta'}(t_{ik})(X_{i}(t_{ij'})+\epsilon_{ij'}-\mu_{0}(t_{ij'}))\Delta_{ik'}\phi_{\beta}(t_{ij'})\phi_{\beta'}(t_{ik'})]\\
        &=  \mathbb{E}\int_{T}(X_{i}(s)+\epsilon_{ij}-\mu_{0}(s))^2 \phi_{\beta}^2(s)ds \int_{T}(\mu_{0}(t)-\hat{\mu}_{\eta}(t))^2 \phi_{\beta'}^2(t)\, dt\\
        &\lesssim  \tau_{\beta} \tau_{\beta'}.
    \end{split}
\end{align*}
For the last step we use the fact that $\sup_{t\in T}\mathbb{E}[X^2(t)] < \infty$ implies $\sup_{t \in T}\mu_{0}^2(t)< \infty$ and $\|\hat{\mu}_{\eta}\|_{\infty}\leq c < \infty$ holds from Remark~\ref{}.\\

\noindent
\textit{\underline{Case-3: $(j= j' , k \neq k')$}}
\begin{align*}
    \begin{split}
        &\mathbb{E}[(X_{i}(t_{ij})+\epsilon_{ij}-\mu_{0}(t_{ij}))\Delta_{ik}\phi_{\beta}(t_{ij})\phi_{\beta'}(t_{ik})(X_{i}(t_{ij'})+\epsilon_{ij'}-\mu_{0}(t_{ij'}))\Delta_{ik'}\phi_{\beta}(t_{ij'})\phi_{\beta'}(t_{ik'})]\\
        &\leq  \mathbb{E}[(X_{i}(t_{ij})+\epsilon_{ij}-\mu_{0}(t_{ij}))^2\Delta_{ik}\phi^2_{\beta}(t_{ij})\phi_{\beta'}(t_{ik}) \Delta_{ik'}\phi_{\beta'}(t_{ik'})].
    \end{split}
\end{align*}

\noindent
Putting things together, we get

\begin{align*}
    \begin{split}
        \mathbb{E}\left\|\frac{1}{n}\sum_{i=1}^{n}(\Lambda_{2}+\lambda I)^{-\frac{1}{2}}U_{i}\right\|_{L^2(T \times T)}^2  
        \lesssim  \frac{1}{n^2}\sum_{i=1}^{n}v_{i}^2m_{i}^{\underline{4}}+ \frac{1}{n^2}\sum_{i=1}^{n}v_{i}^2m_{i}^{\underline{2}}\lambda^{-\frac{1}{b_{1}}}\log\left(\frac{1}{\lambda}\right)+ \frac{1}{n^2}\sum_{i=1}^{n} v_{i}^2m_{i}^{\underline{3}}\lambda^{-\frac{1}{b_{1}}}.
    \end{split}
\end{align*}

\noindent
Bound for \text{Term-$3b$} is exactly same as \text{Term-$3a$}, so we move on to bound \text{Term-$3c$}.\\

\noindent
\underline{\textit{Term-$3c$:}}
\begin{align*}
    \begin{split}
        \|(\Lambda_{2}+\lambda I)^{-\frac{1}{2}}V_{3}\|_{L^2(T \times T)} &=  \left\|\frac{1}{n}\sum_{i=1}^{n}v_{i}(\Lambda_{2}+\lambda I)^{-\frac{1}{2}}U_{i}\right\|_{L^2(T \times T)},
    \end{split}
\end{align*}
where $U_{i} = \sum_{1 \leq j \neq k \leq m_{i}}\Delta_{ij}\Delta_{ik}k^{\frac{1}{2}}(t_{ij},\cdot)k^{\frac{1}{2}}(t_{ik},\cdot)$.
\noindent
Similar to previous bound we consider
\begin{align*}
    \begin{split}
        \mathbb{E}\left\|\frac{1}{n}\sum_{i=1}^{n}v_{i}(\Lambda_{2}+\lambda I)^{-\frac{1}{2}}U_{i}\right\|^2_{L^2(T \times T)} =  \mathbb{E} \left[\frac{1}{n^2}\sum_{i=1}^{n}\sum_{i'=1}^{n}v_{i}v_{i'} \langle (\Lambda_{2}+\lambda I)^{-\frac{1}{2}}U_{i}, (\Lambda_{2}+\lambda I)^{-\frac{1}{2}}U_{i'}\rangle_{L^2(T \times T)}\right].
    \end{split}
\end{align*}
The case when $i=i'$ follows similarly as previous case, so we work with the scenario when $i\neq i'$. So we have

\begin{align*}
   & \mathbb{E} \left[\frac{1}{n^2}\sum_{1 \leq i \neq i' \leq n}v_{i}v_{i'} \langle (\Lambda_{2}+\lambda I)^{-\frac{1}{2}}U_{i}, (\Lambda_{2}+\lambda I)^{-\frac{1}{2}}U_{i'}\rangle_{L^2(T \times T)}\right]\\
   = & \frac{1}{n^2}\sum_{1 \leq i \neq i' \leq n}v_{i}v_{i'} m_{i}(m_{i}-1)m_{i'}(m_{i'}-1)\sum_{\beta,\beta'} \frac{\langle \mu_{0}-\hat{\mu}_{\eta}, \phi_{\beta} \rangle_{L^2(T)}^2\langle \mu_{0}-\hat{\mu}_{\eta}, \phi_{\beta'} \rangle_{L^2(T)}^2}{(\tau_{\beta}\tau_{\beta'}+\lambda)}\\
   \leq &  \|\mu_{0}-\hat{\mu}_{\eta}\|_{L^2(T)}^4.
\end{align*}

\noindent
Therefore,
\begin{equation*}
\begin{split}
    \|(\hat{T}_{n}+\lambda I)^{-\frac{1}{2}}V_{3}\|_{L^2(T \times T)} \lesssim_p & \|\mu_{0}-\hat{\mu}_{\eta}\|^2_{L^2(T)} + \sqrt{\frac{1}{n^2}\sum_{i=1}^{n}v_{i}^2m_{i}^{\underline{4}}}\\
    & + \sqrt{\frac{1}{n^2}\sum_{i=1}^{n}v_{i}^2m_{i}^{\underline{2}}\lambda^{-\frac{1}{b_{1}}}\log\left(\frac{1}{\lambda}\right)}+ \sqrt{\frac{1}{n^2}\sum_{i=1}^{n} v_{i}^2m_{i}^{\underline{3}}\lambda^{-\frac{1}{b_{1}}}}.
\end{split}
\end{equation*}

\noindent
Combining all the bounds together, we obtain
\begin{equation*}
\begin{split}
    \|\hat{C}_{\eta,\lambda}-C_{\lambda}^{\text{true}}\|_{L^2(T \times T)} \lesssim_p & \|\mu_{0}-\hat{\mu}_{\eta}\|^2_{L^2(T)} + \sqrt{\frac{1}{n^2}\sum_{i=1}^{n}v_{i}^2m_{i}^{\underline{4}}}\\
    & + \sqrt{\frac{1}{n^2}\sum_{i=1}^{n}v_{i}^2m_{i}^{\underline{2}}\lambda^{-\frac{1}{b_{1}}}\log\left(\frac{1}{\lambda}\right)}+ \sqrt{\frac{1}{n^2}\sum_{i=1}^{n} v_{i}^2m_{i}^{\underline{3}}\lambda^{-\frac{1}{b_{1}}}}.
\end{split}
\end{equation*}

\noindent
\textit{Bounding \textit{Term-4}:}
Estimation of this term will be carried out separately for $0 < \alpha_{1} \leq \frac{1}{2}$ and $\alpha_{1} \geq \frac{1}{2}$ cases.\\

\noindent
\underline{\textit{Case-1:} $0 < \alpha_{1} \leq \frac{1}{2}$}
\begin{align*}
    \begin{split}
        &\|C^{\text{true}}_{\lambda}-C_{0}\|_{L^2(T \times T)}=  \|\Lambda_{2}^{\frac{1}{2}}g_{\lambda}(\hat{T}_{n})O_{3}-C_{0}\|_{L^2(T \times T)}\\
        &\leq \underbrace{\|\Lambda_{2}^{\frac{1}{2}}g_{\lambda}(\hat{T}_{n})(O_{3}-\Lambda_{2}^{\frac{1}{2}}C_{0})\|_{L^2(T \times T)}}_{\textit{Term-5}} + \underbrace{\|\Lambda_{2}^{\frac{1}{2}}g_{\lambda}(\hat{T}_{n})\Lambda_{2}^{\frac{1}{2}}C_{0}-C_{0}\|_{L^2(T \times T)}}_{\textit{Term-6}}.
    \end{split}
\end{align*}

\noindent
\textit{Bounding \textit{Term-5}:}
\begin{align*}
    \begin{split}
        &\|\Lambda_{2}^{\frac{1}{2}}g_{\lambda}(\hat{T}_{n})(O_{3}-\Lambda_{2}^{\frac{1}{2}}C_{0})\|_{L^2(T \times T)}\\ &\leq  \|(\Lambda_{2}+\lambda I)^{\frac{1}{2}}(\hat{T}_{n}+\lambda I)^{-\frac{1}{2}}\|_{\text{op}}^2
          \|(\hat{T}_{n}+\lambda I)^{\frac{1}{2}}g_{\lambda}(\hat{T}_{n})(\hat{T}_{n}+\lambda I)^{\frac{1}{2}}\|_{\text{op}}\\
        & \qquad\qquad\times \|\Lambda_{2}^{\frac{1}{2}}(\Lambda_{2}+\lambda I)^{-\frac{1}{2}}\|_{\text{op}}\|(\Lambda_{2}+\lambda I)^{-\frac{1}{2}}(O_{3}-\Lambda_{2}^{\frac{1}{2}}C_{0})\|_{L^2(T \times T)}\\
         &\stackrel{(*)}\lesssim_{p}  \|(\Lambda_{2}+\lambda I)^{-\frac{1}{2}}(O_{3}-\Lambda_{2}^{\frac{1}{2}}C_{0})\|\\
         &\lesssim_{p}  \sqrt{\frac{1}{n^2}\sum_{i=1}^{n}v_{i}^2 m_{i}^{\underline{4}}} + \sqrt{\frac{1}{n^2}\sum_{i=1}^{n}v_{i}^2m_{i}^{\underline{2}} \lambda^{-\frac{1}{b_{1}}} \log\left(\frac{1}{\lambda}\right)} + \sqrt{\frac{1}{n^2}\sum_{i=1}^{n}v_{i}^2 m_{i}^{\underline{3}} \lambda^{-\frac{1}{b_{1}}}},
    \end{split}
\end{align*}
where $(*)$ follows from Lemma~\ref{covariance_power_bound} and properties of the regularization family, and the last step follows from Lemma~\ref{covariance_empirical}.\\

\noindent
\textit{Bounding \textit{Term-6}:}
\begin{align*}
    \begin{split}
        &\|\Lambda_{2}^{\frac{1}{2}}g_{\lambda}(\hat{T}_{n})\Lambda_{2}^{\frac{1}{2}}C_{0}-C_{0}\|_{L^2(T \times T)} =  \underbrace{\|\Lambda_{2}^{\frac{1}{2}}(g_{\lambda}(\hat{T}_{n})-(\hat{T}_{n}+\lambda I)^{-1})\Lambda_{2}^{\frac{1}{2}}C_{0}\|_{L^2(T \times T)}}_{\text{Term-$6a$}}\\
        & \qquad + \underbrace{\|\Lambda_{2}^{\frac{1}{2}}((\hat{T}_{n}+\lambda I)^{-1}-(\Lambda_{2}+\lambda I)^{-1})\Lambda_{2}^{\frac{1}{2}}C_{0}\|_{L^2(T \times T)}}_{\text{Term-$6b$}}\\
        & \qquad\qquad+ \underbrace{\|\Lambda_{2}^{\frac{1}{2}}(\Lambda_{2}+\lambda I)^{-1}\Lambda_{2}^{\frac{1}{2}}C_{0}-C_{0}\|_{L^2(T \times T)}}_{\text{Term-$6c$}}.
    \end{split}
\end{align*}

\noindent
\underline{Bound for Term-$6a$:}
\begin{align*}
    \begin{split}
        &\|\Lambda_{2}^{\frac{1}{2}}(g_{\lambda}(\hat{T}_{n})-(\hat{T}_{n}+\lambda I)^{-1})\Lambda_{2}^{\frac{1}{2}}C_{0}\|_{L^2(T \times T)}\\
        &=  \|\Lambda_{2}^{\frac{1}{2}}(-r_{\lambda}(\hat{T}_{n})+\lambda g_{\lambda}(\hat{T}_{n}))(\hat{T}_{n}+\lambda I)^{-1}\Lambda_{2}^{\frac{1}{2}}C_{0}\|_{L^2(T \times T)}\\
        &\lesssim \sqrt{\lambda}\|(\hat{T}_{n}+\lambda I)^{-1}\Lambda_{2}^{\frac{1}{2}}C_{0}\|_{L^2(T \times T)}\\
       & \leq  \sqrt{\lambda}\|((\hat{T}_{n}+\lambda I)^{-1}-(\Lambda_{2}+\lambda I)^{-1})\Lambda_{2}^{\frac{1}{2}}C_{0}\|_{L^2(T \times T)}
         + \sqrt{\lambda}\|(\Lambda_{2}+\lambda I)^{-1}\Lambda_{2}^{\frac{1}{2}}\Lambda_{2}^{\alpha_{1}}H\|_{L^2(T \times T)}\\
       & \lesssim_{p} \|(\Lambda_{2}+\lambda I)^{-\frac{1}{2}}(\hat{T}_{n}-\Lambda_{2})(\Lambda_{2}+\lambda I)^{-1}\Lambda_{2}^{\frac{1}{2}}C_{0}\|_{L^2(T \times T)}+ \lambda^{\alpha_{1}}\\
        &\lesssim_{p}  \sqrt{\frac{1}{n^2}\sum_{i=1}^{n}v_{i}^2m_{i}^{\underline{2}} \lambda^{-\frac{1}{b_{1}}} \log\left(\frac{1}{\lambda}\right)} + \sqrt{\frac{1}{n^2}\sum_{i=1}^{n}v_{i}^2 m_{i}^{\underline{3}} \lambda^{-\frac{1}{b_{1}}}}+ \lambda^{\alpha_{1}},
    \end{split}
\end{align*}
where the last step follows from Lemma~\ref{covariance_empirical_operator_with_embedding}.\\

\noindent
\underline{Bound for Term-$6b$:}
\begin{align*}
    \begin{split}
        \|\Lambda_{2}^{\frac{1}{2}}((\hat{T}_{n}+\lambda I)^{-1}-&(\Lambda_{2}+\lambda I)^{-1})\Lambda_{2}^{\frac{1}{2}}C_{0}\|_{L^2(T \times T)}\\
        &=  \|\Lambda_{2}^{\frac{1}{2}}(\hat{T}_{n}+\lambda I)^{-1}(\Lambda_{2}-\hat{T}_{n})(\Lambda_{2}+\lambda I)^{-1}\Lambda_{2}^{\frac{1}{2}}C_{0}\|_{L^2(T \times T)}\\
        &\lesssim_{p}  \| (\Lambda_{2}+\lambda I)^{-\frac{1}{2}}(\Lambda_{2}-\hat{T}_{n})(\Lambda_{2}+\lambda I)^{-1}\Lambda_{2}^{\frac{1}{2}}C_{0}\|_{L^2(T \times T)}\\
        &\lesssim_{p}  \sqrt{\frac{1}{n^2}\sum_{i=1}^{n}v_{i}^2m_{i}^{\underline{2}} \lambda^{-\frac{1}{b_{1}}} \log\left(\frac{1}{\lambda}\right)} + \sqrt{\frac{1}{n^2}\sum_{i=1}^{n}v_{i}^2 m_{i}^{\underline{3}} \lambda^{-\frac{1}{b_{1}}}},
    \end{split}
\end{align*}
where the last step follows from Lemma~\ref{covariance_empirical_operator_with_embedding}.\\

\noindent
\underline{Bound for Term-$6c$:}
\begin{align*}
    \begin{split}
        \|\Lambda_{2}^{\frac{1}{2}}(\Lambda_{2}+\lambda I)^{-1}\Lambda_{2}^{\frac{1}{2}}C_{0}-C_{0}\|_{L^2(T\times T)} \lesssim  \sup_{\beta, \beta'}\left|\frac{(\tau_{\beta}\tau_{\beta'})^{\alpha_{1}+1}}{\tau_{\beta}\tau_{\beta'}+\lambda}-(\tau_{\beta}\tau_{\beta'})^{\alpha_{1}}\right|
        =  \sup_{\beta, \beta'}\left|\frac{\lambda(\tau_{\beta}\tau_{\beta'})^{\alpha_{1}}}{\tau_{\beta}\tau_{\beta'}+\lambda}\right| \leq \lambda^{\alpha_{1}}.
    \end{split}
\end{align*}

\noindent
Putting everything together will yield
\begin{equation*}
    \|C_{\lambda}^{\text{true}}-C_{0}\|_{L^2(T \times T)} \lesssim_{p} \sqrt{\frac{1}{n^2}\sum_{i=1}^{n}v_{i}^2 m_{i}^{\underline{4}}} + \sqrt{\frac{1}{n^2}\sum_{i=1}^{n}v_{i}^2m_{i}^{\underline{2}} \lambda^{-\frac{1}{b_{1}}} \log\left(\frac{1}{\lambda}\right)} + \sqrt{\frac{1}{n^2}\sum_{i=1}^{n}v_{i}^2 m_{i}^{\underline{3}} \lambda^{-\frac{1}{b_{1}}}}+\lambda^{\alpha_{1}}.
\end{equation*}

\noindent
\underline{\textit{Case-2:} $ \alpha_{1} \geq \frac{1}{2}$}
    \begin{align*}
    \begin{split}
        \|C^{\text{true}}_{\lambda}-C_{0}\|_{L^2(T \times T)}&=  \|\Lambda^{\frac{1}{2}}_{2}g_{\lambda}(\hat{T}_{n})O_{3}-C_{0}\|_{L^2(T \times T)}\\
        &\leq \underbrace{\|\Lambda^{\frac{1}{2}}_{2}g_{\lambda}(\hat{T}_{n})(O_{3}-\hat{T}_{n}\Lambda_{2}^{\alpha_{1}-\frac{1}{2}}H)\|_{L^2(T \times T)}}_{\emph{Term-7}}\\
        &\qquad + \underbrace{\|\Lambda^{\frac{1}{2}}_{2}g_{\lambda}(\hat{T}_{n})\hat{T}_{n}\Lambda_{2}^{\alpha_{1}-\frac{1}{2}}H- C_{0}\|_{L^2(T \times T)}}_{\emph{Term-8}}.
    \end{split}
    \end{align*}
We now bound \emph{Term-7} and \emph{Term-8} as follows.\\

\noindent
\textit{Bounding \textit{Term-7}:}
\begin{align*}
    \begin{split}
        &\|\Lambda^{\frac{1}{2}}_{2}g_{\lambda}(\hat{T}_{n})(O_{3}-\hat{T}_{n}\Lambda_{2}^{\alpha_{1}-\frac{1}{2}}H)\|_{L^2(T \times T)}\\
        &\leq \|\Lambda^{\frac{1}{2}}_{2}(\Lambda_{2}+\lambda I)^{-\frac{1}{2}}\|_{\text{op}} \|(\Lambda_{2}+\lambda I)^{\frac{1}{2}}(\hat{T}_{n}+\lambda I)^{-\frac{1}{2}}\|^2_{\text{op}}\\
        &\qquad\times \|(\hat{T}_{n}+\lambda I)^{\frac{1}{2}}g_{\lambda}(\hat{T}_{n})(\hat{T}_{n}+\lambda I)^{\frac{1}{2}}\|_{\text{op}}\\
        & \qquad\qquad\times \|(\Lambda_{2}+\lambda I)^{-\frac{1}{2}}(O_{3}-\hat{T}_{n}\Lambda_{2}^{\alpha_{1}-\frac{1}{2}}H)\|_{L^2(T \times T)}\\
        &\stackrel{(*)}\lesssim_{p}  \|(\Lambda_{2}+\lambda I)^{-\frac{1}{2}}(O_{3}-\hat{T}_{n}\Lambda_{2}^{\alpha_{1}-\frac{1}{2}}H)\|_{L^2(T \times T)}\\
        &\lesssim_{p}  \sqrt{\frac{1}{n^2}\sum_{i=1}^{n}v_{i}^2 m_{i}^{\underline{4}}} + \sqrt{\frac{1}{n^2}\sum_{i=1}^{n}v_{i}^2m_{i}^{\underline{2}} \lambda^{-\frac{1}{b_{1}}} \log\left(\frac{1}{\lambda}\right)} + \sqrt{\frac{1}{n^2}\sum_{i=1}^{n}v_{i}^2 m_{i}^{\underline{3}} \lambda^{-\frac{1}{b_{1}}}},
    \end{split}
\end{align*}
where $(*)$ follows from Lemma~\ref{covariance_power_bound} and properties of the regularization family, and the last step follows from Lemma~\ref{covariance_empirical}.\\

\noindent
\textit{Bounding \textit{Term-8}:}
\begin{equation*}
    \begin{split}
        &\|\Lambda^{\frac{1}{2}}_{2}g_{\lambda}(\hat{T}_{n})\hat{T}_{n}\Lambda_{2}^{\alpha_{1}-\frac{1}{2}}H- C_{0}\|_{L^2(T \times T)} =  \|\Lambda^{\frac{1}{2}}_{2}r_{\lambda}(\hat{T}_{n})\Lambda_{2}^{\alpha_{1}-\frac{1}{2}}H\|_{L^2(T \times T)}\\
        &\leq  \|\Lambda^{\frac{1}{2}}_{2}(\Lambda_{2} + \lambda I)^{-\frac{1}{2}}\|_{\text{op}} \|(\Lambda_{2} + \lambda I)^{\frac{1}{2}}(\hat{T}_{n}+ \lambda I)^{-\frac{1}{2}}\|_{\text{op}}\\
          & \qquad\times \|(\hat{T}_{n}+\lambda I)^{\frac{1}{2}}r_{\lambda}(\hat{T}_{n})(\hat{T}_{n}+\lambda I)^{\nu-\frac{1}{2}}\|_{\text{op}} \|(\hat{T}_{n}+\lambda I)^{-(\nu-\frac{1}{2})}\Lambda_{2}^{\alpha_{1}-\frac{1}{2}}H\|_{L^2(T \times T)}\\
        &\stackrel{(*)}\lesssim_{p}  \lambda^{\nu} \|(\hat{T}_{n}+\lambda I)^{-(\nu-\frac{1}{2})}\Lambda_{2}^{\alpha_{1}-\frac{1}{2}}H\|_{L^2(T \times T)}\\
        &\lesssim \lambda^{\nu} \|(\hat{T}_{n}+\lambda I)^{-(q-\lfloor q\rfloor)}(\hat{T}_{n}+\lambda I)^{-\lfloor q\rfloor}\Lambda_{2}^{\alpha_{1}-\frac{1}{2}}\|_{\text{op}}\\
& \leq  \lambda^{\nu} \|(\hat{T}_{n}+\lambda I)^{-(q-\lfloor q\rfloor)}(\Lambda_{2}+\lambda I)^{q-\lfloor q\rfloor}\|_{\text{op}}\\
& \qquad\times \|(\Lambda_{2}+\lambda I)^{-(q-\lfloor q\rfloor)}(\hat{T}_{n}+\lambda I)^{-\lfloor q\rfloor}\Lambda_{2}^{\alpha_{1}-\frac{1}{2}}\|_{\text{op}}\\
 &\lesssim  \frac{\lambda^{\nu}}{\lambda^{q-\lfloor q\rfloor}}\|((\hat{T}_{n}+\lambda I)^{-\lfloor q\rfloor}- (\Lambda_{2}+\lambda I)^{-\lfloor q\rfloor})\Lambda_{2}^{\alpha_{1}-\frac{1}{2}}\|_{\text{op}}
  + \lambda^{\nu}\|(\Lambda_{2}+\lambda I)^{-q}\Lambda_{2}^{\alpha_{1}-\frac{1}{2}}\|_{\text{op}},
\end{split}
\end{equation*}
where $q = \nu-\frac{1}{2}$ and $\lfloor q\rfloor$ is greatest integer value less than or equal to $q$. For $(*)$, we use Lemma~\ref{covariance_power_bound} and properties of the regularization family. Now using Lemma~\ref{reducing_power}, we get
\begin{align*}
    \begin{split}
        &\|\Lambda^{\frac{1}{2}}_{2}g_{\lambda}(\hat{T}_{n})\hat{T}_{n}\Lambda_{2}^{\alpha_{1}-\frac{1}{2}}H- C_{0}\|_{L^2(T \times T)}  \lesssim \sqrt{\lambda}\|(\Lambda_{2}+\lambda I)^{-\frac{1}{2}}(\Lambda_{2}-\hat{T}_{n})(\Lambda_{2}+\lambda I)^{-\frac{1}{2}}\|+ \lambda^{\alpha_{1}}\\
        \lesssim_{p} & \sqrt{\frac{1}{n^2}\sum_{i=1}^{n}v_{i}^2 m_{i}^{\underline{3}}\lambda^{-\frac{1}{b_{1}}}} + \sqrt{\frac{1}{n^2}\sum_{i=1}^{n}v_{i}^2 m_{i}^{\underline{2}}\lambda^{-\frac{1}{b_{1}}}\log\left(\frac{1}{\lambda}\right)} +\lambda^{\alpha_{1}},
    \end{split}
\end{align*}
where the last step follows from Lemma~\ref{covarianve_empirical_operator_without_embedding}. Combining everything, we have
\begin{align*}\label{covariance_final_bound}
    \begin{split}
        \|\hat{C}_{\eta,\lambda} - C_{0}\|_{L^2(T \times T)} & \lesssim_{p} \|\hat{\mu}_{\eta}- \mu_{0}\|^2_{L^2(T)} + \sqrt{\frac{1}{n^2}\sum_{i=1}^{n}v_{i}^2 m_{i}^{\underline{4}}} + \sqrt{\frac{1}{n^2}\sum_{i=1}^{n}v_{i}^2m_{i}^{\underline{2}} \lambda^{-\frac{1}{b_{1}}} \log\left(\frac{1}{\lambda}\right)}\\
        & \qquad+ \sqrt{\frac{1}{n^2}\sum_{i=1}^{n}v_{i}^2 m_{i}^{\underline{3}} \lambda^{-\frac{1}{b_{1}}}} + \lambda^{\alpha_{1}}.
    \end{split}
\end{align*}
Note that if $\alpha_{1} \geq \nu$, then
\begin{equation*}
\begin{split}
    \|\hat{C}_{\eta,\lambda} - C_{0}\|_{L^2(T \times T)} \lesssim_{p} &  \|\hat{\mu}_{\eta}- \mu_{0}\|^2_{L^2(T)} + \sqrt{\frac{1}{n^2}\sum_{i=1}^{n}v_{i}^2 m_{i}^{\underline{4}}} + \sqrt{\frac{1}{n^2}\sum_{i=1}^{n}v_{i}^2m_{i}^{\underline{2}} \lambda^{-\frac{1}{b_{1}}} \log\left(\frac{1}{\lambda}\right)}\\
        & \qquad+ \sqrt{\frac{1}{n^2}\sum_{i=1}^{n}v_{i}^2 m_{i}^{\underline{3}} \lambda^{-\frac{1}{b_{1}}}} + \lambda^{\nu}.
\end{split}
\end{equation*}

\subsection{Proof of Theorem~\ref{thm:cov-lower-1}}\label{subsec:thm-cov-lower-1}
Let $\{\tau_{\ell},\psi_\ell\}_{\ell\geq1}$ be eigenvalue-eigenfunction pair of integral operator $\Lambda$.
For an integer $M\geq1$, let

$$
\omega^0,\omega^1,\ldots,\omega^N\in\{0,1\}^M
$$

be a constant-weight Varshamov--Gilbert packing such that

$$
N\geq 2^{c_1M},
\qquad
|\omega^k|=qM,
$$

for some fixed $q\in(0,1/2)$, and

$$
H(\omega^k,\omega^{k'})
\geq c_2M,
\qquad k\neq k',
$$

where $c_1,c_2>0$ are constants. We choose $q$ and the packing so that
$c_2<2q$.

For each $\omega$, define $ f_\omega = \frac{r_{n}}{\sqrt{qM}} \sum_{\ell=1}^M \omega_\ell\psi_{\ell+M}. $ Since $|\omega|=qM$ and the $\psi_\ell$'s are orthonormal,

$$ \|f_\omega\|_{L^2(T)}^2 = \frac{r_{n}^2}{qM}|\omega| = r_n^2. $$

Hence $\|f_\omega\|_{L^2(T)}=r_n.$ We write $f_{\omega} = \Lambda_{1}^{\alpha_{1}}h_{\omega}$, where 

$$h_{\omega} = \frac{r_{n}}{\sqrt{qM}}\sum_{\ell = 1}^{M}\omega_{\ell}\tau_{M+\ell}^{-\alpha_{1}}\psi_{M+\ell}.$$

We can see that $\|h\|_{L^2(T)} \leq c < \infty$ given that $r_{n}M^{\alpha_{1}b_{1}} \lesssim 1$. For two distinct elements $\omega,\omega'$,

$$
\begin{aligned}
\|f_\omega-f_{\omega'}\|_{L^2(T)}^2
&=
\frac{r_n^2}{qM}
\sum_{\ell=1}^M
(\omega_\ell-\omega'_\ell)^2 =
\frac{r_n^2}{qM}
d_H(\omega,\omega') \geq
\frac{c_2}{q}r_n^2.
\end{aligned}
$$

Therefore,

$$
\|f_\omega-f_{\omega'}\|_{L^2(T)}
\gtrsim r_n.
$$

Now let us consider $ Z_i\stackrel{\mathrm{i.i.d.}}{\sim}N(0,1)$ and define

$$
X_{i,\omega}(t)
=
(1+Z_i)f_\omega(t).
$$
The variables $Z_i$ are independent of the design points and the
measurement errors. Since $E[Z_i]=0$,
$$ \mu_\omega(t) = E[X_{i,\omega}(t)] = f_\omega(t)\qquad\text{and}\qquad C_\omega=f_\omega\otimes f_\omega.$$

We next see that for any $f,g\in L^2(T)$,

$$
\|f\otimes f-g\otimes g\|_{L^2(T\times T)}^2 = \|f\|_{L^2(T)}^4 + \|g\|_{L^2(T)}^4 - 2\langle f,g\rangle_{L^2(T)}^2.$$

In our construction, $ \|f_\omega\|_{L^2(T)} = \|f_{\omega'}\|_{L^2(T)} = r_n.$

Moreover,
$$
\langle f_\omega,f_{\omega'}\rangle_{L^2(T)} = \frac{r_n^2}{qM}\sum_{\ell}\omega_{\ell}\omega'_{\ell}.
$$

Since both vectors have Hamming weight $qM$,

$$
H(\omega,\omega') = 2qM -2\sum_{\ell}\omega_{\ell}\omega'_{\ell}.
$$
Consequently,
$$
\begin{aligned}
\langle f_\omega,f_{\omega'}\rangle_{L^2(T)} &= r_n^2 \left( 1 -\frac{H(\omega,\omega')}{2qM} \right) \leq
r_n^2 \left( 1-\frac{c_2}{2q} \right).
\end{aligned}
$$

Set $ \rho = 1-\frac{c_2}{2q}.$ By construction, $0\leq\rho<1$. Hence
$$
\langle f_\omega,f_{\omega'}\rangle_{L^2(T)} \leq \rho r_n^2.
$$

Therefore,

$$
\begin{aligned}
\|C_\omega-C_{\omega'}\|_{L^2(T\times T)}^2 &= 2r_n^4 - 2\langle f_\omega,f_{\omega'}\rangle_{L^2(T)}^2\\ 
&\geq 2(1-\rho^2)r_n^4.
\end{aligned}
$$

Thus

$$
\|C_\omega-C_{\omega'}\|_{L^2(T\times T)} \geq \sqrt{2(1-\rho^2)}\,r_n^2.
$$

Next we control the Kullback--Leibler divergence. Let $P_{\omega}$ be the distribution for corresponding observed data when input is $X_{i,\omega}$. Conditional on the design points
and latent variables $Z=(Z_1,\ldots,Z_n)$,

$$ Y_{ij} = (1+Z_i)f_\omega(t_{ij})+\epsilon_{ij}. $$

Thus, conditional on $(T,Z)$, the observations are independent
Gaussian random variables with variance $\sigma^2$. For any $\omega$, we have

$$
\mathcal K(P_\omega^{Y,T}\,\|\,P_0^{Y,T})
\leq
\mathcal K(P_\omega^{Y,T,Z}\,\|\,P_0^{Y,T,Z}).
$$

Since the distribution of $(T,Z)$ is the same under the two models,

$$
\begin{aligned}
&\mathcal K(P_\omega^{Y,T,Z}\,\|\,P_0^{Y,T,Z})\\
&\quad=
\frac{1}{2\sigma^2}
\sum_{i=1}^n\sum_{j=1}^{m_i}
\mathbb{E}\left[
(1+Z_i)^2
\{f_\omega(t_{ij})-f_{\omega^0}(t_{ij})\}^2
\right].
\end{aligned}
$$

Since the design points are uniformly distributed on $T$,
\begin{align*}
    \mathcal K(P_\omega^{Y,T}\,\|\,P_0^{Y,T})   \leq & \frac{1}{\sigma^2} \left(\sum_{i=1}^nm_i\right) \|f_\omega-f_{\omega^0}\|_{L^2(T)}^2\\
    \lesssim & \left(\sum_{i=1}^nm_i\right)r_n^2.
\end{align*}

We choose $M$ so that $\left(\sum_{i=1}^nm_i\right)r_n^2 \lesssim M. $ Note that under the smoothness condition on mean function and decay condition on eigenvalues we can see from Theorem~\ref{thm:mean-lower} that $r_{n} = (nm_{A})^{-\frac{\alpha_{1}b_{1}}{1+2\alpha_{1}b_{1}}}$ and taking $M \asymp (nm_{A})^{\frac{1}{1+2\alpha_{1}b_{1}}}$ does satisfy $r_{n}M^{-\alpha_{1}b_{1}} \lesssim 1$ as well as $\left(\sum_{i=1}^nm_i\right)r_n^2 \lesssim M.$

Since the Varshamov--Gilbert packing satisfies $ \log N\gtrsim M,$ the constant in the construction can be chosen sufficiently small so
that

$$
\frac1N
\sum_{k=1}^N
\mathcal K(P_{\omega^k}^{Y,T}\,\|\,P_0^{Y,T})
\leq
\kappa\log N
$$

for a sufficiently small constant $\kappa>0$.\\

\noindent
 Fano's inequality consequently
implies that there exist constants $c,c_0>0$ such that

$$
\inf_{\widehat C}
\sup_{\omega}
P_\omega
\left\{
\|\widehat C-C_\omega\|_{L^2(T\times T)}
\geq c r_n^2
\right\}
\geq c_0.
$$

Since the constructed family is contained in $\mathcal P$,

$$
\inf_{\widehat C}
\sup_{P\in\mathcal P}
P\left\{
\|\widehat C-C_P\|_{L^2(T\times T)}
\geq c r_n^2
\right\}
\geq c_0.
$$

It is trivially true that estimating covariance with unknown mean is always harder than estimating with known mean. So we get the desired result.

\subsection{Proof of Theorem~\ref{thm:lower_covariance}}\label{subsec:thm-cov-lower-2}

Let $\beta$ be a random variable satisfying
$$\mathbb{P}(\beta=1)
=
\mathbb{P}(\beta=-1)
=
\frac12,$$
and consider the stochastic process
$$X=\beta f,
\qquad
f \in \mathcal{R}(\Lambda)^{\alpha_{1}}~~ \text{such that}~~\|\Lambda^{-\alpha_{1}}f\|_{L^2(T)} \leq R ,$$

where $\Lambda$ is the integral operator associated with kernel $k$.\\

\noindent
Then $\mathbb{E}X=0$ and its covariance function is
$$C_{f} = \mathbb{E}[X\otimes X] = f\otimes f.$$
By Lemma~\ref{lem:covariance-range},
$$C_{f}\in\mathcal{F}^{C}_{\alpha_{1},R_{1}}.$$

Let $M>8$. By the Varshamov--Gilbert bound, there exist
$\theta^{0},\ldots,\theta^{N}\in\{0,1\}^{M}$, with
$\theta^{0}=0$, such that $N\geq 2^{M/8}$
and
$$H(\theta^{k},\theta^{k'})
\geq \frac{M}{8},
\qquad
0\leq k<k'\leq N,$$
where $H(\cdot,\cdot)$ denotes the Hamming distance.\\

\noindent
For $\theta\in\{0,1\}^{M}$ and $\delta >0$, define
$$f_{\theta} = M^{-1/2} \delta \sum_{\ell=M+1}^{2M}
\theta_{\ell-M}\tau_{\ell}^{\alpha_{1}}\psi_{\ell}.$$
Choose a fixed function $f_{0}\in
\mathcal{R}(\Lambda^{\alpha_{1}})$ such that $\|\Lambda^{-\alpha_{1}}f_{0}\|_{L^2(T)} \leq R$ and define
$$f_{k} = f_{0}+f_{\theta^{k}}, \qquad k=0,\ldots,N.$$
Let $\Pi_{k}$ denote the joint distribution of the observed data
$\{(Y_{ij},t_{ij}):1\leq i\leq n,\,
1\leq j\leq m_{i}\}$ under $f=f_{k}$.

By conditioning additionally on $\beta$ and
using the data-processing inequality for the Kullback--Leibler
divergence, we obtain
\begin{align*}
\mathcal{K}(\Pi_{k}\,\|\,\Pi_{0})
&\leq
\frac{1}{2\sigma^{2}}
\sum_{i=1}^{n}\sum_{j=1}^{m_{i}}
\mathbb{E}_{T}
\left[
\{f_{k}(t_{ij})-f_{0}(t_{ij})\}^{2}
\right] \\
&=
\frac{n\bar m}{2\sigma^{2}}
\|f_{k}-f_{0}\|_{L^{2}(T)}^{2}=
\frac{n\bar m}{2\sigma^{2}}
\|f_{\theta^{k}}\|_{L^{2}(T)}^{2}.
\end{align*}
Since $\tau_{\ell}\asymp\ell^{-b_{1}}$ for
$\ell\in\{M+1,\ldots,2M\}$,
\begin{align*}
\|f_{\theta^{k}}\|_{L^{2}(T)}^{2} &= \frac{1}{M}
\sum_{\ell=M+1}^{2M}
(\theta_{\ell-M}^{k})^{2}\tau_{\ell}^{2\alpha_{1}} \lesssim
M^{-2\alpha_{1} b_{1}}.
\end{align*}
Consequently,
$$\mathcal{K}(\Pi_{k}\,\|\,\Pi_{0})
\lesssim
n\bar m\,M^{-2\alpha_{1} b_{1}}.$$

\noindent
Choose
$$M\asymp
(n\bar m)^{\frac{1}{1+2\alpha_{1} b_{1}}}.$$
Then
$$\mathcal{K}(\Pi_{k}\,\|\,\Pi_{0})
\lesssim M
\lesssim \log N,$$
where the multiplicative constant in the construction may be chosen
sufficiently small.

Further we see that for $k\neq k'$, we have
$$C_{k}-C_{k'}
=
 f_{k}\otimes f_{k}
-
f_{k'}\otimes f_{k'}.$$
The construction can be chosen so that
$$\|C_{k}-C_{k'}\|_{L^{2}(T\times T)}^{2}
\gtrsim
\|f_{\theta^{k}} - f_{\theta^{k'}}\|_{L^{2}(T)}^{2}.$$
Moreover, by the Varshamov--Gilbert separation,
\begin{align*}
\|f_{\theta^{k}} - f_{\theta^{k'}}\|_{L^{2}(T)}^{2}
&=
\frac{1}{M}
\sum_{\ell=M+1}^{2M}
(\theta_{\ell-M}^{k}
-\theta_{\ell-M}^{k'})^{2}
\tau_{\ell}^{2\alpha_{1}}\\
&\gtrsim
\frac{H(\theta^{k},\theta^{k'})}{M}
M^{-2\alpha_{1} b_{1}} \gtrsim
M^{-2\alpha_{1} b_{1}}.
\end{align*}
Therefore,
\[
\|C_{k}-C_{k'}\|_{L^{2}(T\times T)}^{2}
\gtrsim
M^{-2\alpha_{1} b_{1}}
\asymp
(n\bar m)^{-\frac{2\alpha_{1} b_{1}}{1+2\alpha_{1} b_{1}}}.
\]
The conclusion now follows from
\cite[Theorem~2.5]{tsyback2009lb}.\\

\noindent
Finally, the $n^{-1/2}$ term follows by considering the fully observed model, which is more informative than the discretely observed model. Indeed, restricting the fully observed model to the one-dimensional subfamily

$$
X_i=\sqrt{\theta}Z_i\psi,\qquad Z_i\stackrel{\mathrm{i.i.d.}}{\sim}N(0,1),
$$

where $\|\psi\|_{L^2(T)}=1$, gives $C_\theta=\theta\psi\otimes\psi$. Estimating $C_\theta$ is therefore equivalent to estimating the variance parameter $\theta$ from $n$ independent Gaussian observations, whose minimax rate is $n^{-1/2}$. Hence, the covariance estimation error is bounded from below by $n^{-1/2}$ even under full observation, and consequently also under discrete noisy observation.

\bibliographystyle{acm}  
\bibliography{References}

@article {ramsay1991some,
    AUTHOR = {Ramsay, J. O. and Dalzell, C. J.},
     TITLE = {Some tools for functional data analysis},
   JOURNAL = {J. Roy. Statist. Soc. Ser. B},
  FJOURNAL = {Journal of the Royal Statistical Society. Series B.
              Methodological},
    VOLUME = {53},
      YEAR = {1991},
    NUMBER = {3},
     PAGES = {539--572},
      ISSN = {0035-9246},
   MRCLASS = {62H25},
  MRNUMBER = {1125714},
MRREVIEWER = {P.\ A. K. Covey-Crump},
       URL =
              {http://links.jstor.org/sici?sici=0035-9246(1991)53:3<539:STFFDA>2.0.CO;2-W&origin=MSN},
}

@article {cuevas2002,
    AUTHOR = {Cuevas, Antonio and Febrero, Manuel and Fraiman, Ricardo},
     TITLE = {Linear functional regression: The case of fixed design and
              functional response},
   JOURNAL = {Canad. J. Statist.},
  FJOURNAL = {The Canadian Journal of Statistics. La Revue Canadienne de
              Statistique},
    VOLUME = {30},
      YEAR = {2002},
    NUMBER = {2},
     PAGES = {285--300},
      ISSN = {0319-5724,1708-945X},
   MRCLASS = {62G08 (62J02)},
  MRNUMBER = {1926066},
MRREVIEWER = {Walter\ Schlee},
       DOI = {10.2307/3315952},
       URL = {https://doi.org/10.2307/3315952},
}

@article {clustering_2003_fd,
    AUTHOR = {James, Gareth M. and Sugar, Catherine A.},
     TITLE = {Clustering for sparsely sampled functional data},
   JOURNAL = {J. Amer. Statist. Assoc.},
  FJOURNAL = {Journal of the American Statistical Association},
    VOLUME = {98},
      YEAR = {2003},
    NUMBER = {462},
     PAGES = {397--408},
      ISSN = {0162-1459,1537-274X},
   MRCLASS = {62H30},
  MRNUMBER = {1995716},
MRREVIEWER = {Kalev\ P\"arna},
       DOI = {10.1198/016214503000189},
       URL = {https://doi.org/10.1198/016214503000189},
}

@article {muller2005generalized,
    AUTHOR = {M\"{u}ller, Hans-Georg and Stadtm\"{u}ller, Ulrich},
     TITLE = {Generalized functional linear models},
   JOURNAL = {Ann. Statist.},
  FJOURNAL = {The Annals of Statistics},
    VOLUME = {33},
      YEAR = {2005},
    NUMBER = {2},
     PAGES = {774--805},
      ISSN = {0090-5364,2168-8966},
   MRCLASS = {62G05 (62G20 62H30 62M09)},
  MRNUMBER = {2163159},
MRREVIEWER = {Arnak\ S.\ Dalalyan},
       DOI = {10.1214/009053604000001156},
       URL = {https://doi.org/10.1214/009053604000001156},
}

@article {yao2005sparsetodense,
    AUTHOR = {Yao, Fang and M\"uller, Hans-Georg and Wang, Jane-Ling},
     TITLE = {Functional data analysis for sparse longitudinal data},
   JOURNAL = {J. Amer. Statist. Assoc.},
  FJOURNAL = {Journal of the American Statistical Association},
    VOLUME = {100},
      YEAR = {2005},
    NUMBER = {470},
     PAGES = {577--590},
      ISSN = {0162-1459,1537-274X},
   MRCLASS = {62H25 (62G05)},
  MRNUMBER = {2160561},
MRREVIEWER = {M.\ Riedel},
       DOI = {10.1198/016214504000001745},
       URL = {https://doi.org/10.1198/016214504000001745},
}

@article {hall_2006_principle_compo,
    AUTHOR = {Hall, Peter and M\"uller, Hans-Georg and Wang, Jane-Ling},
     TITLE = {Properties of principal component methods for functional and
              longitudinal data analysis},
   JOURNAL = {Ann. Statist.},
  FJOURNAL = {The Annals of Statistics},
    VOLUME = {34},
      YEAR = {2006},
    NUMBER = {3},
     PAGES = {1493--1517},
      ISSN = {0090-5364,2168-8966},
   MRCLASS = {62G08 (62H25 62M09)},
  MRNUMBER = {2278365},
MRREVIEWER = {Stephan\ Morgenthaler},
       DOI = {10.1214/009053606000000272},
       URL = {https://doi.org/10.1214/009053606000000272},
}

@article {devito,
    AUTHOR = {Caponnetto, A. and De Vito, E.},
     TITLE = {Optimal rates for the regularized least-squares algorithm},
   JOURNAL = {Found. Comput. Math.},
  FJOURNAL = {Foundations of Computational Mathematics. The Journal of the
              Society for the Foundations of Computational Mathematics},
    VOLUME = {7},
      YEAR = {2007},
    NUMBER = {3},
     PAGES = {331--368},
      ISSN = {1615-3375,1615-3383},
   MRCLASS = {68T05 (62J02 65K10 93E24)},
  MRNUMBER = {2335249},
       DOI = {10.1007/s10208-006-0196-8},
       URL = {https://doi.org/10.1007/s10208-006-0196-8},
}

@article {zhang_2007_statstical_inference,
    AUTHOR = {Zhang, Jin-Ting and Chen, Jianwei},
     TITLE = {Statistical inferences for functional data},
   JOURNAL = {Ann. Statist.},
  FJOURNAL = {The Annals of Statistics},
    VOLUME = {35},
      YEAR = {2007},
    NUMBER = {3},
     PAGES = {1052--1079},
      ISSN = {0090-5364,2168-8966},
   MRCLASS = {62G07 (62G10 62J12)},
  MRNUMBER = {2341698},
       DOI = {10.1214/009053606000001505},
       URL = {https://doi.org/10.1214/009053606000001505},
}

@article{cai2010nonparametric,
  title={Nonparametric covariance function estimation for functional and longitudinal data, Technical Report},
  author={T. Tony Cai and Ming Yuan},
 JOURNAL = {Univ. Pennsylvania, Philadelphia, PA}
}

@article {hsing_2010_non_parametric,
    AUTHOR = {Li, Yehua and Hsing, Tailen},
     TITLE = {Uniform convergence rates for nonparametric regression and
              principal component analysis in functional/longitudinal data},
   JOURNAL = {Ann. Statist.},
  FJOURNAL = {The Annals of Statistics},
    VOLUME = {38},
      YEAR = {2010},
    NUMBER = {6},
     PAGES = {3321--3351},
      ISSN = {0090-5364,2168-8966},
   MRCLASS = {62J05 (60F15 62G08 62G20 62H25 62M20)},
  MRNUMBER = {2766854},
MRREVIEWER = {Daniela\ Rodriguez},
       DOI = {10.1214/10-AOS813},
       URL = {https://doi.org/10.1214/10-AOS813},
}

@article {Cai_2011_mean,
    AUTHOR = {Cai, T. Tony and Yuan, Ming},
     TITLE = {Optimal estimation of the mean function based on discretely
              sampled functional data: {P}hase transition},
   JOURNAL = {Ann. Statist.},
  FJOURNAL = {The Annals of Statistics},
    VOLUME = {39},
      YEAR = {2011},
    NUMBER = {5},
     PAGES = {2330--2355},
      ISSN = {0090-5364,2168-8966},
   MRCLASS = {62G05 (62G08)},
  MRNUMBER = {2906870},
MRREVIEWER = {Mounir\ Arfi},
       DOI = {10.1214/11-AOS898},
       URL = {https://doi.org/10.1214/11-AOS898},
}

@article{wang2016functional,
  title={Functional data analysis},
  author={Wang, Jane-Ling and Chiou, Jeng-Min and M{\"u}ller, Hans-Georg},
  journal={Annual Review of Statistics and Its Application},
  volume={3},
  pages={257--295},
  year={2016},
  publisher={Annual Reviews}
}

@article {Wang&zhang_2016_sparse_to_dense,
    AUTHOR = {Zhang, Xiaoke and Wang, Jane-Ling},
     TITLE = {From sparse to dense functional data and beyond},
   JOURNAL = {Ann. Statist.},
  FJOURNAL = {The Annals of Statistics},
    VOLUME = {44},
      YEAR = {2016},
    NUMBER = {5},
     PAGES = {2281--2321},
      ISSN = {0090-5364,2168-8966},
   MRCLASS = {62G20 (62G05 62G08)},
  MRNUMBER = {3546451},
       DOI = {10.1214/16-AOS1446},
       URL = {https://doi.org/10.1214/16-AOS1446},
}

@article {Medina_fixed_design,
    AUTHOR = {Ruiz-Medina, M. D.},
     TITLE = {Functional analysis of variance for {H}ilbert-valued
              multivariate fixed effect models},
   JOURNAL = {Statistics},
  FJOURNAL = {Statistics. A Journal of Theoretical and Applied Statistics},
    VOLUME = {50},
      YEAR = {2016},
    NUMBER = {3},
     PAGES = {689--715},
      ISSN = {0233-1888,1029-4910},
   MRCLASS = {60G15 (62H10 62H15 62J10)},
  MRNUMBER = {3506664},
MRREVIEWER = {Jos\'e\ R.\ Berrendero},
       DOI = {10.1080/02331888.2015.1094069},
       URL = {https://doi.org/10.1080/02331888.2015.1094069},
}

@article {wang&zhang_2018_longitudinal_and_functional,
    AUTHOR = {Zhang, Xiaoke and Wang, Jane-Ling},
     TITLE = {Optimal weighting schemes for longitudinal and functional
              data},
   JOURNAL = {Statist. Probab. Lett.},
  FJOURNAL = {Statistics \& Probability Letters},
    VOLUME = {138},
      YEAR = {2018},
     PAGES = {165--170},
      ISSN = {0167-7152,1879-2103},
   MRCLASS = {62G20 (62G05 62G08)},
  MRNUMBER = {3788733},
       DOI = {10.1016/j.spl.2018.03.007},
       URL = {https://doi.org/10.1016/j.spl.2018.03.007},
}

@article {Wong&Zhand_2019_covariance,
    AUTHOR = {Wong, Raymond K. W. and Zhang, Xiaoke},
     TITLE = {Nonparametric operator-regularized covariance function
              estimation for functional data},
   JOURNAL = {Comput. Statist. Data Anal.},
  FJOURNAL = {Computational Statistics \& Data Analysis},
    VOLUME = {131},
      YEAR = {2019},
     PAGES = {131--144},
      ISSN = {0167-9473,1872-7352},
   MRCLASS = {62G05 (62G20)},
  MRNUMBER = {3906800},
MRREVIEWER = {Zhenghui\ Feng},
       DOI = {10.1016/j.csda.2018.05.013},
       URL = {https://doi.org/10.1016/j.csda.2018.05.013},
}

@article {Steinwart_2020_sobolev_norm,
    AUTHOR = {Fischer, Simon and Steinwart, Ingo},
     TITLE = {Sobolev norm learning rates for regularized least-squares
              algorithms},
   JOURNAL = {J. Mach. Learn. Res.},
  FJOURNAL = {Journal of Machine Learning Research (JMLR)},
    VOLUME = {21},
      YEAR = {2020},
     PAGES = {1--38},
      ISSN = {1532-4435,1533-7928},
   MRCLASS = {62G08 (62G20)},
  MRNUMBER = {4209491},
MRREVIEWER = {Yongdai\ Kim},
}

@article {wang_2022_low_rank,
    AUTHOR = {Wang, Jiayi and Wong, Raymond K. W. and Zhang, Xiaoke},
     TITLE = {Low-rank covariance function estimation for multidimensional
              functional data},
   JOURNAL = {J. Amer. Statist. Assoc.},
  FJOURNAL = {Journal of the American Statistical Association},
    VOLUME = {117},
      YEAR = {2022},
    NUMBER = {538},
     PAGES = {809--822},
      ISSN = {0162-1459,1537-274X},
   MRCLASS = {99-01},
  MRNUMBER = {4436314},
       DOI = {10.1080/01621459.2020.1820344},
       URL = {https://doi.org/10.1080/01621459.2020.1820344},
}

@article {gupta2025optimal,
    AUTHOR = {Naveen Gupta and S. Sivananthan and Bharath K. Sriperumbudur},
     TITLE = {Optimal rates for functional linear regression with general regularization},
   JOURNAL = {Applied and Computational Harmonic Analysis},
    VOLUME = {76},
      YEAR = {2025},
     PAGES = {101745},
      ISSN = {1063-5203},
       DOI = {https://doi.org/10.1016/j.acha.2024.101745},
       URL = {https://www.sciencedirect.com/science/article/pii/S1063520324001222},
}

@book {cordes1987,
    AUTHOR = {Cordes, H. O.},
     TITLE = {Spectral Theory of Linear Differential Operators and
              Comparison Algebras},
    SERIES = {London Mathematical Society Lecture Note Series},
    VOLUME = {76},
 PUBLISHER = {Cambridge University Press, Cambridge},
      YEAR = {1987},
     PAGES = {x+342},
      ISBN = {0-521-28443-0},
   MRCLASS = {35Pxx (35S05 46L99 47E05 47F05 47G05 58G05 58G12)},
  MRNUMBER = {890743},
MRREVIEWER = {V.\ S.\ Rabinovich},
       DOI = {10.1017/CBO9780511662836},
       URL = {https://doi.org/10.1017/CBO9780511662836},
}

@book {englmartinbook,
    AUTHOR = {Engl, Heinz W. and Hanke, Martin and Neubauer, Andreas},
     TITLE = {Regularization of {I}nverse {P}roblems},
    SERIES = {Mathematics and its Applications},
    VOLUME = {375},
 PUBLISHER = {Kluwer Academic Publishers Group, Dordrecht},
      YEAR = {1996},
     PAGES = {viii+321},
      ISBN = {0-7923-4157-0},
   MRCLASS = {65J20 (35R30 47A50 47H17)},
  MRNUMBER = {1408680},
MRREVIEWER = {Ulrich\ Tautenhahn},
}

@book {ramsay2002afda,
    AUTHOR = {Ramsay, J. O. and Silverman, B. W.},
     TITLE = {Applied Functional Data Analysis},
    SERIES = {Springer Series in Statistics},
 PUBLISHER = {Springer-Verlag, New York},
      YEAR = {2002},
     PAGES = {x+190},
      ISBN = {0-387-95414-7},
   MRCLASS = {62-02 (62G07 62H25 62H30)},
  MRNUMBER = {1910407},
       DOI = {10.1007/b98886},
       URL = {https://doi.org/10.1007/b98886},
}

@book {tsyback2009lb,
    AUTHOR = {Tsybakov, Alexandre B.},
     TITLE = {Introduction to Nonparametric Estimation},
    SERIES = {Springer Series in Statistics},
 PUBLISHER = {Springer, New York},
      YEAR = {2009},
     PAGES = {xii+214},
      ISBN = {978-0-387-79051-0},
   MRCLASS = {62-01 (62G05 62G07 62G08 62G20)},
  MRNUMBER = {2724359},
       DOI = {10.1007/b13794},
       URL = {https://doi.org/10.1007/b13794},
}

@book {sergei2013,
    AUTHOR = {Lu, Shuai and Pereverzev, Sergei V.},
     TITLE = {Regularization Theory for Ill-posed Problems},
    SERIES = {Inverse and Ill-posed Problems Series},
    VOLUME = {58},
 PUBLISHER = {De Gruyter, Berlin},
      YEAR = {2013},
     PAGES = {xiv+289},
      ISBN = {978-3-11-028646-5; 978-3-11-028649-6},
   MRCLASS = {35-02 (35B65 35R25 35R30 60H30 62F15 62G05 65J20)},
  MRNUMBER = {3114700},
       DOI = {10.1515/9783110286496},
       URL = {https://doi.org/10.1515/9783110286496},
}

@book {Theoreticalfoundationshsing2015,
    AUTHOR = {Hsing, Tailen and Eubank, Randall},
     TITLE = {Theoretical Foundations of Functional Data Analysis, with an
              Introduction to Linear Operators},
    SERIES = {Wiley Series in Probability and Statistics},
 PUBLISHER = {John Wiley \& Sons, Ltd., Chichester},
      YEAR = {2015},
     PAGES = {xiv+334},
      ISBN = {978-0-470-01691-6},
   MRCLASS = {62-02 (46-01 47-01 60G17 62G05 62H20 62H25 62J05)},
  MRNUMBER = {3379106},
MRREVIEWER = {David\ Benner\ Hitchcock},
       DOI = {10.1002/9781118762547},
       URL = {https://doi.org/10.1002/9781118762547},
}

@book {paulsen2016rkhs,
    AUTHOR = {Paulsen, Vern I. and Raghupathi, Mrinal},
     TITLE = {An Introduction to the Theory of Reproducing Kernel {H}ilbert
              Spaces},
    SERIES = {Cambridge Studies in Advanced Mathematics},
    VOLUME = {152},
 PUBLISHER = {Cambridge University Press, Cambridge},
      YEAR = {2016},
     PAGES = {x+182},
      ISBN = {978-1-107-10409-9},
   MRCLASS = {46-02 (46C05)},
  MRNUMBER = {3526117},
       DOI = {10.1017/CBO9781316219232},
       URL = {https://doi.org/10.1017/CBO9781316219232},
}

@book {kokoszka2017,
    AUTHOR = {Kokoszka, Piotr and Reimherr, Matthew},
     TITLE = {Introduction to Functional Data Analysis},
    SERIES = {Texts in Statistical Science Series},
 PUBLISHER = {CRC Press, Boca Raton, FL},
      YEAR = {2017},
     PAGES = {xvi+290},
      ISBN = {978-1-4987-4634-2},
   MRCLASS = {62-01 (62Gxx 62H25 62Jxx 62M10 62M30)},
  MRNUMBER = {3793167},
MRREVIEWER = {David\ Benner\ Hitchcock},
}

\appendix
\section{Supplementary Results: Mean Function Estimation}\label{app:mean}
\numberwithin{equation}{section}
In this section, we present supplementary results that are needed to prove Theorem~\ref{thm:mean}.
\begin{lemma}\label{mean_empirical_estimation}
    Under Assumptions~\ref{mean_SRC} and \ref{mean_moment}, 
    we have
    $$\|(\Lambda_{1}+\lambda I)^{-\frac{1}{2}}(V_{1}-\Lambda_{1}^{\frac{1}{2}}\mu_{0})\|_{L^2(T)} \lesssim_p \sqrt{{\mathcal{N}_{1}(\lambda)}}\left(\frac{1}{n^2}\sum_{i=1}^{n}w_{i}^2m_{i}\right)^{\frac{1}{2}}+\left(\frac{1}{n^2}\sum_{i=1}^{n}w_{i}^2m_{i}(m_{i}-1)\right)^{\frac{1}{2}},$$
    and for $\alpha \geq \frac{1}{2}$, we have
    $$\|(\Lambda_{1}+\lambda I)^{-\frac{1}{2}}(V_{1}-\hat{A}_{n}\Lambda_{1}^{\alpha-\frac{1}{2}}h)\|_{L^2(T)} \lesssim_p \sqrt{{\mathcal{N}_{1}(\lambda)}}\left(\frac{1}{n^2}\sum_{i=1}^{n}w_{i}^2m_{i}\right)^{\frac{1}{2}}+\left(\frac{1}{n^2}\sum_{i=1}^{n}w_{i}^2m_{i}(m_{i}-1)\right)^{\frac{1}{2}}.$$
\end{lemma}

\begin{proof}
We start with 
\begin{equation*}
    \begin{split}
        \|(\Lambda_{1}+\lambda I)^{-\frac{1}{2}}(V_{1}-\Lambda_{1}^{\frac{1}{2}}\mu_{0})\|_{L^2(T)} = \left\|\frac{1}{n}\sum_{i=1}^{n}w_{i}Z_{i}\right\|_{L^2(T)},
    \end{split}
\end{equation*}
    where $Z_{i} = (\Lambda_{1}+\lambda I)^{-\frac{1}{2}}U_{i}$ and $U_{i}:=\sum_{j=1}^{m_{i}}(Y_{ij}k^{\frac{1}{2}}(t_{ij}, \cdot)-\Lambda_{1}^{\frac{1}{2}}\mu_{0})$. We remind the reader that the form of $V_{1}$ is defined in Proposition~\ref{mean_alternate_estimator}. Now using Chebyshev's  inequality, for any $t>0$, we obtain  
\begin{equation}\label{mean_chebyshev}
    \mathbb{P}\left(\left\|\frac{1}{n}\sum_{i=1}^{n}w_{i}Z_{i}\right\|_{L^2(T)}\geq t\right) \leq \frac{\mathbb{E}\left\|\frac{1}{n}\sum_{i=1}^{n}w_{i}Z_{i}\right\|_{L^2(T)}^2}{t^2}.
\end{equation}
Since $\{Z_{i}\}_{i=1}^{n}$ are independent and $\mathbb{E}[Z_{i}]= 0$, we have
\begin{equation*}
    \mathbb{E}\left\|\frac{1}{n}\sum_{i=1}^{n}w_{i}Z_{i}\right\|_{L^2(T)}^2 = \frac{1}{n^2}\sum_{i=1}^{n}w_{i}^2\mathbb{E}\|Z_{i}\|_{L^2(T)}^2 = \frac{1}{n^2}\sum_{i=1}^{n}w_{i}^2\sum_{k}\frac{\mathbb{E}\langle U_{i}, \psi_{k}\rangle^2_{L^2}}{(\lambda_{k}+\lambda)}.
\end{equation*}

\noindent
Consider 
\begin{align*}
    \mathbb{E}\langle U_{i}, \psi_{k} \rangle^2_{L^2(T)} &=  \mathbb{E}\left[\sum_{j=1}^{m_{i}}[Y_{ij}\phi_{k}(t_{ij})- \langle \Lambda_{1}^{\frac{1}{2}}\mu_{0}, \psi_{k} \rangle_{L^2(T)}]\right]^2\\
    &=  \underbrace{\text{Var}_{T}\left[\mathbb{E}_{X|T}\left[\sum_{j=1}^{m_{i}}\left(Y_{ij}\phi_{k}(t_{ij})-\langle \Lambda^{\frac{1}{2}}\mu_{0}, \psi_{k}\rangle_{L^2(T)}\right)\Big |T\right]\right]}_{\text{Term-I}}\\
    & \qquad+ \underbrace{\mathbb{E}_{T}\left[\text{Var}_{X|T}\left[\sum_{j=1}^{m_{i}}\left(Y_{ij}\phi_{k}(t_{ij})-\langle \Lambda^{\frac{1}{2}}\mu_{0}, \psi_{k}\rangle_{L^2(T)}\right)\Big|T\right]\right]}_{\text{Term-II}}.
\end{align*}
\noindent
\underline{Bound for Term-I:}
\begin{align*}
    &\text{Var}_{T}\left[\mathbb{E}_{X|T}\left[\sum_{j=1}^{m_{i}}(Y_{ij}\phi_{k}(t_{ij})-\langle \Lambda^{\frac{1}{2}}\mu_{0}, \psi_{k}\rangle_{L^2(T)})\Big{|}T\right]\right] \\
    &=  \text{Var}_{T}\left(\sum_{j=1}^{m_{i}}(\mu_{0}(t_{ij})\phi_{k}(t_{ij})-\langle \Lambda^{\frac{1}{2}}\mu_{0}, \psi_{k}\rangle_{L^2(T)})\right)\\
    &\leq \mathbb{E}_{T}\left(\sum_{j=1}^{m_{i}}(\mu_{0}(t_{ij})\phi_{k}(t_{ij})-\langle \Lambda^{\frac{1}{2}}\mu_{0}, \psi_{k}\rangle_{L^2(T)})\right)^2
    \leq  m_{i}\int_{T} \mu_{0}^2(t) \phi_{k}^2(t)dt\lesssim m_{i}\lambda_k,
\end{align*}
where the last inequality follows from the assumption $\sup_{t \in T}\mathbb{E}X^2(t) < \infty$.\\

\noindent
\underline{Bound for Term-II:}
\begin{align}\label{mean_term_II}
    \begin{split}
    &\mathbb{E}_{T}\left[\text{Var}_{X|T}\left[\sum_{j=1}^{m_{i}}(Y_{ij}\phi_{k}(t_{ij})-\langle \Lambda^{\frac{1}{2}}\mu_{0}, \psi_{k}\rangle_{L^2(T)})\Big{|}T\right]\right]\\
    &\leq  \mathbb{E}_{T,X}\left[\sum_{j,k=1}^{m_{i}}Y_{ij}Y_{ik}\phi_{k}(t_{ij})\phi_{k}(t_{ik})\right]\\
         & =  \mathbb{E}_{T}\left[\sum_{j,k=1}^{m_{i}}\left(\mathbb{E}_{X|T}\left[X_{i}(t_{ij})X_{i}(t_{ik})\Big{|}T\right]+\sigma^2_{0} \delta_{jk}\right)\phi_{k}(t_{ij})\phi_{k}(t_{ik})\right]\\
         &\leq  m_{i}(m_{i}-1) \int_{T \times T}  \mathbb{E}_{X}(X_{i}(s)X_{i}(t))\phi_{k}(t)\phi_{k}(s)dsdt + m_{i} \sigma^2_{0} \lambda_{k}\\
         & \qquad\qquad+ m_{i} \int_{T}\mathbb{E}_{X}(X_{i}^2(t))\phi_{k}^2(t)dt\\
         & \lesssim m_{i}\lambda_k+m_{i}(m_{i}-1)\lambda_k\mathbb{E}\langle X,\psi_k\rangle^2_{L^2(T)},
    \end{split}
\end{align}
where we used the assumption $\sup_{t \in T}\mathbb{E}X^2(t) < \infty$ in the third term on the penultimate line. Combining these bounds, we get

\begin{align}\label{mean_variance_bound}
    \begin{split}
        \mathbb{E}\left\|\frac{1}{n}\sum_{i=1}^{n}w_{i}Z_{i}\right\|^2_{L^2(T)} & \lesssim \frac{1}{n^2}\sum_{i=1}^{n}w_{i}^2 m_{i}\sum_{k}\frac{\lambda_{k}}{\lambda_{k}+\lambda} +\frac{1}{n^2}\sum_{i=1}^{n}w_{i}^2 m_{i} (m_{i}-1)\mathbb{E}\sum_{k} \frac{\lambda_{k}}{\lambda_{k}+\lambda} \langle X, \psi_{k}\rangle^2_{L^2(T)}\\
        &\leq  \frac{\mathcal{N}_{1}(\lambda)}{n^2}\sum_{i=1}^{n}w_{i}^2 m_{i} + \frac{\mathbb{E}\|X\|_{L^2(T)}^2}{n^2} \sum_{i=1}^{n}w_{i}^2 m_{i} (m_{i}-1).
    \end{split}
\end{align}
Hence, the result follows from (\ref{mean_chebyshev}) and (\ref{mean_variance_bound}). \\

\noindent
For $\|(\Lambda_{1}+\lambda I)^{-\frac{1}{2}}(V_{1}-\hat{A}_{n}\Lambda_{1}^{\alpha-\frac{1}{2}}h)\|_{L^2(T)} = \left\|\frac{1}{n}\sum_{i=1}^{n}w_{i}Z_{i}\right\|_{L^2(T)}$, take $Z_{i} = (\Lambda_{1}+\lambda I)^{-\frac{1}{2}}U_{i}$ where  $U_{i} = \sum_{j=1}^{m_{i}}(Y_{ij}-\mu_{0}(t_{ij}))k^{\frac{1}{2}}(t_{ij}, \cdot)$, where we note that $$\hat{A}_{n} \Lambda_{1}^{\alpha-\frac{1}{2}}h = \frac{1}{n}\sum_{i=1}^{n}w_{i}\sum_{j=1}^{m_{i}}\mu_{0}(t_{ij})k^{\frac{1}{2}}(t_{ij}, \cdot).$$ Now using Chebyshev's  inequality, for any $t>0$, we obtain  
\begin{equation}\label{mean_chebyshev_1}
    \mathbb{P}\left(\left\|\frac{1}{n}\sum_{i=1}^{n}w_{i}Z_{i}\right\|_{L^2(T)}\geq t\right) \leq \frac{\mathbb{E}\left\|\frac{1}{n}\sum_{i=1}^{n}w_{i}Z_{i}\right\|_{L^2(T)}^2}{t^2}.
\end{equation}
Since $\{Z_{i}\}_{i=1}^{n}$ are i.i.d and $\mathbb{E}[Z_{i}]= 0$, we have
\begin{equation*}
    \mathbb{E}\left\|\frac{1}{n}\sum_{i=1}^{n}w_{i}Z_{i}\right\|_{L^2(T)}^2 = \frac{1}{n^2}\sum_{i=1}^{n}w_{i}^2\mathbb{E}\|Z_{i}\|_{L^2(T)}^2 = \frac{1}{n^2}\sum_{i=1}^{n}w_{i}^2\sum_{k}\frac{\mathbb{E}\langle U_{i}, \psi_{k}\rangle^2_{L^2}}{(\lambda_{k}+\lambda)}.
\end{equation*}
Clearly one can see that $\mathbb{E}\langle U_{i},\psi_{k}\rangle_{L^2(T)}= 0$. So we have
\begin{equation*}
\begin{split}
    \mathbb{E}\left\|\frac{1}{n}\sum_{i=1}^{n}w_{i}Z_{i}\right\|_{L^2(T)}^2 = \frac{1}{n^2}\sum_{i=1}^{n}w_{i}^2\sum_{k}\frac{\text{Var}\langle U_{i}, \psi_{k} \rangle_{L^2(T)}}{(\lambda_{k}+\lambda)}.
\end{split}
\end{equation*}

\noindent
Consider 
\begin{align*}
    \begin{split}
        \text{Var}\langle U_{i}, \psi_{k} \rangle_{L^2(T)} & = \text{Var}_{T}\left[\mathbb{E}_{X|T}\left[\sum_{j=1}^{m_{i}}(Y_{ij}-\mu_{0}(t_{ij}))\phi_{k}(t_{ij})\Big{|}T\right]\right]\\
        &\qquad\qquad + \mathbb{E}_{T}\left[\text{Var}_{X|T}\left[\sum_{j=1}^{m_{i}}(Y_{ij}-\mu_{0}(t_{ij}))\phi_{k}(t_{ij})\Big{|}T\right]\right]\\
        & \leq \mathbb{E}\left(\sum_{j=1}^{m_{i}}(Y_{ij}-\mu_{0}(t_{ij}))\phi_{k}(t_{ij})\right)^2.
        \end{split}
\end{align*}

Using the calculations from (\ref{mean_term_II}), we get
\begin{equation}\label{mean_variance_bound_2}
    \begin{split}
        \left\|\frac{1}{n}\sum_{i=1}^{n}(\Lambda_{1}+\lambda I)^{-\frac{1}{2}}U_{i}\right\|^2_{L^2(T)} \lesssim_{p} \frac{\mathcal{N}_{1}(\lambda)}{n^2}\sum_{i=1}^{n}w_{i}^2 m_{i}+ \frac{1}{n^2}\sum_{i=1}^{n}w_{i}^2 m_{i}(m_{i}-1),
    \end{split}
\end{equation}
and the result follows from equations (\ref{mean_chebyshev_1}) and (\ref{mean_variance_bound_2}).
\end{proof}

\begin{lemma}\label{mean_empirical_operator}
    Under Assumptions~\ref{mean_moment}, \ref{mean_eigenvalue_decay}
    and \ref{mean_embedding}, for $0<\alpha\le\frac{1}{2}$, 
    we have
    \begin{equation*}
        \|(\Lambda_{1}+\lambda I)^{-\frac{1}{2}}(\Lambda_{1}-\hat{A}_{n})(\Lambda_{1}+\lambda I)^{-\frac{1}{2}}\|_{\emph{op}} \lesssim_p \left(\frac{1}{n^2}\sum_{i=1}^{n}w_{i}^2 m_{i}\right)^{\frac{1}{2}}\sqrt{\frac{\lambda^{-\frac{1}{b}}}{\lambda^{2\alpha}}},
    \end{equation*}
    and
    \begin{equation*}
        \|(\Lambda_{1}+\lambda I)^{-\frac{1}{2}}(\Lambda_{1}-\hat{A}_{n})(\Lambda_{1}+\lambda I)^{-1}\Lambda_{1}^{\frac{1}{2}+\alpha}\|_{\emph{op}} \lesssim_p  \left(\frac{1}{n^2}\sum_{i=1}^{n}w_{i}^2 m_{i}\right) ^{\frac{1}{2}}\lambda^{-\frac{1}{2b}}.
    \end{equation*}
\end{lemma}
\begin{proof}
We start with
    \begin{equation*}
        \begin{split}
           & \|(\Lambda_{1}+\lambda I)^{-\frac{1}{2}}(\Lambda_{1}-\hat{A}_{n})(\Lambda_{1}+\lambda I)^{-\frac{1}{2}}\|_{\text{op}}\\
        & = \sup_{h \in L^2(T),\, \|h\|_{L^2(T)} = 1} |\langle h, (\Lambda_{1}+\lambda I)^{-\frac{1}{2}}(\Lambda_{1}-\hat{A}_{n})(\Lambda_{1}+\lambda I)^{-\frac{1}{2}}h\rangle_{L^2(T)}|.
        \end{split}
    \end{equation*}

\noindent
Consider
\begin{align*}
    \begin{split}
       & |\langle h, (\Lambda_{1}+\lambda I)^{-\frac{1}{2}}(\Lambda_{1}-\hat{A}_{n})(\Lambda_{1}+\lambda I)^{-\frac{1}{2}}h\rangle_{L^2(T)}|\\ &=  \left|\sum_{k, \beta} \frac{\langle h, \psi_{k}\rangle_{L^2(T)}\langle h, \psi_{\beta}\rangle_{L^2(T)}}{(\lambda_{k}+\lambda)^{\frac{1}{2}}(\lambda_{\beta}+\lambda)^{\frac{1}{2}}}\langle \psi_{k}, (\Lambda_{1}-\hat{A}_{n})\psi_{\beta}\rangle_{L^2(T)}\right|\\
        &\leq  \left(\sum_{k, \beta}\frac{\langle \psi_{k}, (\Lambda_{1}-\hat{A}_{n})\psi_{\beta}\rangle_{L^2(T)}^2}{(\lambda_{k}+\lambda)(\lambda_{\beta}+\lambda)}\right)^{\frac{1}{2}}\|h\|_{L^2(T)}^2,
    \end{split}
\end{align*}
where the last step follows from Cauchy-Schwarz inequality. This means,
$$\|(\Lambda_{1}+\lambda I)^{-\frac{1}{2}}(\Lambda_{1}-\hat{A}_{n})(\Lambda_{1}+\lambda I)^{-\frac{1}{2}}\|_{\text{op}}\le \left(\sum_{k, \beta}\frac{\langle \psi_{k}, (\Lambda_{1}-\hat{A}_{n})\psi_{\beta}\rangle_{L^2(T)}^2}{(\lambda_{k}+\lambda)(\lambda_{\beta}+\lambda)}\right)^{\frac{1}{2}}.$$
Taking expectation and applying Jensen's inequality gives
\begin{align*}
    \begin{split}
        \mathbb{E}\|(\Lambda_{1}+\lambda I)^{-\frac{1}{2}}(\Lambda_{1}-\hat{A}_{n})(\Lambda_{1}+\lambda I)^{-\frac{1}{2}}\|_{\text{op}}\le
        \left(\sum_{k, \beta}\frac{\mathbb{E}\langle \psi_{k}, (\Lambda_{1}-\hat{A}_{n})\psi_{\beta}\rangle_{L^2(T)}^2}{(\lambda_{k}+\lambda)(\lambda_{\beta}+\lambda)}\right)^{\frac{1}{2}}.
    \end{split}
\end{align*}
Since $\hat{A}_{n} = \frac{1}{n}\sum_{i=1}^{n}w_{i}\sum_{j=1}^{m_{i}}U_{ij}$, where $U_{ij} = k^{\frac{1}{2}}(t_{ij}, \cdot)\otimes_{L^2}k^{\frac{1}{2}}(t_{ij}, \cdot)$, consider
\begin{align*}
    \begin{split}
        \mathbb{E}\langle \psi_{k}, (\Lambda_{1}-\hat{A}_{n})\psi_{\beta}\rangle_{L^2(T)}^2  = & \mathbb{E}\left[\left(\frac{1}{n}\sum_{i=1}^{n}w_{i}\sum_{j=1}^{m_{i}}\langle \psi_{k}, (\Lambda_{1}-U_{ij}) \psi_{\beta}
        \rangle_{L^2(T)}\right)^2\right]\\
        = & \mathbb{E}\left[\frac{1}{n^2}\sum_{i=1}^{n}\sum_{i'=1}^{n}T_{i}T_{i'}\right]\\
        = & \frac{1}{n^2}\sum_{1 \leq i \neq i' \leq n} \mathbb{E}[T_{i}]\mathbb{E}[T_{i'}] + \frac{1}{n^2}\sum_{i=1}^{n}\mathbb{E}[T_{i}^2] = \frac{1}{n^2}\sum_{i=1}^{n}\mathbb{E}[T_{i}^2]\\
        = & \frac{1}{n^2}\sum_{i=1}^{n}\mathbb{E}\left[w_{i}^2\sum_{j,j'=1}^{m_{i}}\langle \psi_{k}, (\Lambda_{1}-U_{ij}) \psi_{\beta}
        \rangle_{L^2(T)}\langle \psi_{k}, (\Lambda_{1}-U_{ij'}) \psi_{\beta}
        \rangle_{L^2(T)}\right]\\
        = & \frac{1}{n^2}\sum_{i=1}^{n}\mathbb{E}\left[w_{i}^2\sum_{j=1}^{m_{i}}\langle \psi_{k}, (\Lambda_{1}-U_{ij}) \psi_{\beta}
        \rangle_{L^2(T)}^2\right]\\
        \leq &  \frac{1}{n^2}\sum_{i=1}^{n}\mathbb{E}\left[w_{i}^2m_{i} \int_{T} \phi_{k}^2(t)\phi_{\beta}^2(t)dt\right],
    \end{split}
\end{align*}
where $T_{i} = w_{i}\sum_{j=1}^{m_{i}}\langle \psi_{k}, (\Lambda_{1}-U_{ij}) \psi_{\beta}\rangle_{L^2(T)}$.\\

\noindent
So we can see that
\begin{align*}
    \begin{split}
        & \mathbb{E}\|(\Lambda_{1}+\lambda I)^{-\frac{1}{2}}(\Lambda_{1}-\hat{A}_{n})(\Lambda_{1}+\lambda I)^{-\frac{1}{2}}\|_{\text{op}}
        \leq \left(\frac{1}{n^2}\sum_{i=1}^{n}w_{i}^2 m_{i}\sum_{k, \beta} \frac{\langle \phi^2_{k}, \phi_{\beta}^2\rangle_{L^2}}{(\lambda_{k}+\lambda)(\lambda_{\beta}+\lambda)} \right)^{\frac{1}{2}}\\
        &=  \left(\frac{1}{n^2}\sum_{i=1}^{n}w_{i}^2 m_{i}\right)^{\frac{1}{2}}\left(\int_T{}\sum_{k, \beta}\frac{\lambda_{\beta}^{1-2\alpha}\phi_{k}^2(t)\lambda_{\beta}^{2\alpha}\psi_{\beta}^2(t)}{(\lambda_{k}+\lambda)(\lambda_{\beta}+\lambda)^{1-2 \alpha}(\lambda_{\beta}+\lambda)^{2\alpha}}\,dt\right)^\frac{1}{2}.
    \end{split}
\end{align*}
Under Assumption~\ref{mean_embedding}, we have $\sum_{\beta}\lambda_{\beta}^{2 \alpha}\psi_{\beta}^2(t)\le Z$ for almost every $t \in T$. Hence, we obtain
\begin{equation*}
    \mathbb{E}\|(\Lambda_{1}+\lambda I)^{-\frac{1}{2}}(\Lambda_{1}-\hat{A}_{n})(\Lambda_{1}+\lambda I)^{-\frac{1}{2}}\|_{\text{op}} \lesssim \left(\frac{1}{n^2}\sum_{i=1}^{n}w_{i}^2 m_{i}\right)^{\frac{1}{2}}\sqrt{\frac{\mathcal{N}_{1}(\lambda)}{\lambda^{2\alpha}}}.
\end{equation*}
from which the result follows via Markov's inequality.\\

\noindent
For $\|(\Lambda_{1}+\lambda I)^{-\frac{1}{2}}(\Lambda_{1}-\hat{A}_{n})(\Lambda_{1}+\lambda I)^{-1}\Lambda_{1}^{\alpha+\frac{1}{2}}\|_{\text{op}}$, following similar steps as above, we obtain 
\begin{align*}
    \begin{split}
        &\mathbb{E}\|(\Lambda_{1}+\lambda I)^{-\frac{1}{2}}(\Lambda_{1}-\hat{A}_{n})(\Lambda_{1}+\lambda I)^{-1}\Lambda_{1}^{\alpha+\frac{1}{2}}\|_{\text{op}}\\
        & \leq \left(\frac{1}{n^2}\sum_{i=1}^{n}w_{i}^2 m_{i}\sum_{k, \beta} \frac{ \lambda_{\beta}^{2\alpha+1}\langle \phi^2_{k}, \phi_{\beta}^2\rangle_{L^2}}{(\lambda_{k}+\lambda)(\lambda_{\beta}+\lambda)^2} \right)^{\frac{1}{2}}\\
         & =\left(\frac{1}{n^2}\sum_{i=1}^{n}w_{i}^2 m_{i}\right)^{\frac{1}{2}}\left(\int_T{}\sum_{k, \beta}\frac{\lambda_{\beta}^{2}\phi_{k}^2(t)\lambda_{\beta}^{2\alpha}\psi_{\beta}^2(t)}{(\lambda_{k}+\lambda)(\lambda_{\beta}+\lambda)^{2}}\right)^\frac{1}{2}\\
        &\lesssim  \left(\frac{1}{n^2}\sum_{i=1}^{n}w_{i}^2 m_{i}\right)^{\frac{1}{2}}\sqrt{\sum_{k}\frac{\lambda_{k}}{\lambda_{k}+\lambda}} = \left(\frac{1}{n^2}\sum_{i=1}^{n}w_{i}^2 m_{i}\right)^{\frac{1}{2}}\sqrt{{\mathcal{N}_{1}(\lambda)}}.
    \end{split}
\end{align*}
and the result follows.
\end{proof}

\begin{remark}\label{mean_remark_1}
    It is easy to observe that if we remove Assumption~\ref{mean_embedding} from the analysis of Lemma~\ref{mean_empirical_operator}, then we get an additional $\sqrt{\lambda}$ factor in the denominator of the final bound using the fact that $\sum_{\beta}\phi_{\beta}^2(t) < \infty$ for almost every $t \in T$. So we can conclude that 
    \begin{equation*}
        \|(\Lambda_{1}+\lambda I)^{-1}(\Lambda_{1}-\hat{A}_{n})\|_{\emph{op}} \lesssim_p \left(\frac{1}{n^2}\sum_{i=1}^{n}w_{i}^2 m_{i}\right)^{\frac{1}{2}}\sqrt{\frac{\mathcal{N}_{1}(\lambda)}{\lambda}} \lesssim \left(\frac{1}{n^2}\sum_{i=1}^{n}w_{i}^2 m_{i}\right)^{\frac{1}{2}}{\lambda^{-\frac{1}{2b}-\frac{1}{2}}},
    \end{equation*}
where the last step follows using Assumption~\ref{mean_eigenvalue_decay}.
\end{remark}

\begin{lemma}\label{mean_powers_constant_bound}
Under Assumptions \ref{mean_moment} and ~\ref{mean_eigenvalue_decay}, we have 
    \begin{equation*}
        \|(\hat{A}_{n}+\lambda I)^{-1}(\Lambda_{1}+\lambda I)\|_{\emph{op}} \leq_p 2,\qquad \forall~ \lambda \gg \left(\frac{1}{n^2}\sum_{i=1}^{n}w_{i}^2 m_{i}\right)^{\frac{b}{1+b}}.
    \end{equation*}
    
\noindent
Further, if Assumption~\ref{mean_embedding} holds, then 
\begin{equation*}
        \|(\hat{A}_{n}+\lambda I)^{-\frac{1}{2}}(\Lambda_{1}+\lambda I)^{\frac{1}{2}}\|_{\emph{op}} \leq_p \sqrt{2},\qquad \forall~ \lambda \gg \left(\frac{1}{n^2}\sum_{i=1}^{n}w_{i}^2 m_{i}\right)^{\frac{b}{1+2\alpha b}}.
    \end{equation*}
\end{lemma}
\begin{proof}
    Observe that 
    \begin{align*}
    \begin{split}
        \|(\hat{A}_{n}+\lambda I)^{-1}(\Lambda_{1}+\lambda I)\|_{\text{op}} & =  \|(I - (\Lambda_{1}+\lambda I)^{-1}(\Lambda_{1}-\hat{A}_{n}))^{-1}\|_{\text{op}}\\
        & \leq  \frac{1}{1-\|(\Lambda_{1}+\lambda I)^{-1}(\Lambda_{1}-\hat{A}_{n})\|_{\text{op}}},
    \end{split}
    \end{align*}
provided that $\|(\Lambda_{1}+\lambda I)^{-1}(\Lambda_{1}-\hat{A}_{n})\|_{\text{op}} < 1$. Using the bound in Remark~\ref{mean_remark_1} with $\lambda \gg \left(\frac{1}{n^2}\sum_{i=1}^{n}w_{i}^2 m_{i}\right)^{\frac{b}{1+b}}$, we obtain
\begin{equation*}
    \|(\Lambda_{1}+\lambda I)^{-1}(\Lambda_{1}-\hat{A}_{n})\|_{\text{op}} \leq \frac{1}{2},
\end{equation*}
and the result follows.\\

\noindent
For $\|(\hat{A}_{n}+\lambda I)^{-\frac{1}{2}}(\Lambda_{1}+\lambda I)^{\frac{1}{2}}\|_{\text{op}}$, we see that
\begin{align*}
    \begin{split}
        \|(\hat{A}_{n}+\lambda I)^{-\frac{1}{2}}(\Lambda_{1}+\lambda I)^{\frac{1}{2}}\|^2_{\text{op}} & =  \|(I-(\Lambda_{1}+\lambda I)^{-\frac{1}{2}}(\Lambda_{1}-\hat{A}_{n})(\Lambda_{1}+\lambda I)^{-\frac{1}{2}})^{-1}\|_{\text{op}}\\
        & \leq  \frac{1}{1-\|(\Lambda_{1}+\lambda I)^{-\frac{1}{2}}(\Lambda_{1}-\hat{A}_{n})(\Lambda_{1}+\lambda I)^{-\frac{1}{2}}\|_{\text{op}}},
    \end{split}
\end{align*}
provided $\|(\Lambda_{1}+\lambda I)^{-\frac{1}{2}}(\Lambda_{1}-\hat{A}_{n})(\Lambda_{1}+\lambda I)^{-\frac{1}{2}}\|_{\text{op}} \leq 1$.
Taking $\lambda \gg \left(\frac{1}{n^2}\sum_{i=1}^{n}w_{i}^2 m_{i}\right)^{-\frac{b}{1+b}}$ in Lemma \ref{mean_empirical_operator} yields the result.
\end{proof}
\begin{remark}
    Note that 
    \begin{equation*}
        \begin{split}
            \|(\Lambda_{1}+\lambda I)^{-1}(\hat{A}_{n}+ \lambda I)\|_{\emph{op}} &=   \|I-(\Lambda_{1}+\lambda I)^{-1}(\Lambda_{1}- \hat{A}_{n})\|_{\emph{op}}\\
            &\leq  1 + \|(\Lambda_{1}+\lambda I)^{-1}(\Lambda_{1}- \hat{A}_{n})\|_{\emph{op}}
            \leq \frac{3}{2},~~\forall~\lambda \gg \left(\frac{1}{n^2}\sum_{i=1}^{n}w_{i}^2 m_{i}\right)^{-\frac{b}{1+b}}.
        \end{split}
    \end{equation*}
\end{remark}

\begin{remark}
    A bound for $\|(\hat{A}_{n}+\lambda I)^{-\frac{1}{2}}(\Lambda_{1}+\lambda I)^{\frac{1}{2}}\|_{\emph{op}}$ can also be obtained using the bound on $\|(\hat{A}_{n}+\lambda I)^{-1}(\Lambda_{1}+\lambda I)\|_{\emph{op}}$ and Corde's inequality \cite{cordes1987}. Nevertheless, we present two separate arguments, as this distinction becomes essential in the mis-specified regime $0 < \alpha \leq \frac{1}{2}$. Under Assumption~\ref{mean_embedding}, one may choose any regularization parameter satisfying $\lambda \gtrsim \left(\frac{1}{n^2}\sum_{i=1}^{n}w_{i}^2 m_{i}\right)^{-\frac{b}{1+2 \alpha b}}$. In contrast, if this assumption is not imposed, the analysis requires $\lambda \gtrsim \left(\frac{1}{n^2}\sum_{i=1}^{n}w_{i}^2 m_{i}\right)^{-\frac{b}{1+b}}$, a condition that is incompatible with the choice $\lambda$ that is needed to attain the optimal convergence rates in the mis-specified setting.
\end{remark}

\section{Supplementary Results: Covariance Function Estimation}\label{app:cov}
\numberwithin{equation}{section}
\begin{lemma}\label{covariance_empirical}
Under Assumptions \ref{mean_moment} and \ref{covariance_eigen_decay}, we have
    \begin{equation*}
    \begin{split}
        & \|(\Lambda_{2}+\lambda I)^{-\frac{1}{2}}(O_{3}-\hat{T}_{n}\Lambda_{2}^{\alpha_{1}-\frac{1}{2}}H)\|_{L^2(T \times T)}\\
        \lesssim_p &  \sqrt{\frac{1}{n^2}\sum_{i=1}^{n}v_{i}^2 m_{i}^{\underline{4}}} + \sqrt{\frac{1}{n^2}\sum_{i=1}^{n}v_{i}^2m_{i}^{\underline{2}} \lambda^{-\frac{1}{b_{1}}} \log\left(\frac{1}{\lambda}\right)} + \sqrt{\frac{1}{n^2}\sum_{i=1}^{n}v_{i}^2 m_{i}^{\underline{3}} \lambda^{-\frac{1}{b_{1}}}},
        \end{split}
    \end{equation*}
    Further, if Assumption~\ref{covariance_embedding} holds, then for $0 < \alpha_{1} \leq \frac{1}{2}$, we have
    \begin{equation*}
    \begin{split}
        & \|(\Lambda_{2}+\lambda I)^{-\frac{1}{2}}(O_{3}-\Lambda^{\frac{1}{2}}_{2}C_{0})\|_{L^2(T \times T)}\\
        \lesssim_p & \sqrt{\frac{1}{n^2}\sum_{i=1}^{n}v_{i}^2 m_{i}^{\underline{4}}} + \sqrt{\frac{1}{n^2}\sum_{i=1}^{n}v_{i}^2m_{i}^{\underline{2}} \lambda^{-\frac{1}{b_{1}}} \log\left(\frac{1}{\lambda}\right)} + \sqrt{\frac{1}{n^2}\sum_{i=1}^{n}v_{i}^2 m_{i}^{\underline{3}} \lambda^{-\frac{1}{b_{1}}}}.
    \end{split}
    \end{equation*}
\end{lemma}
\begin{proof}
Since the proof idea for both the bounds is similar, we prove the first inequality and provide a sketch for the other. To this end, we start with
    \begin{equation}\label{to_apply_chebyshev}
        \begin{split}
            \|(\Lambda_{2}+\lambda I)^{-\frac{1}{2}}(O_{3}-\hat{T}_{n}\Lambda_{2}^{\alpha_{1}-\frac{1}{2}}H)\|_{L^2(T \times T)} = \left\|\frac{1}{n}\sum_{i=1}^{n}v_{i}(\Lambda_{2}+\lambda I)^{-\frac{1}{2}}U_{i}\right\|_{L^2(T \times T)},
        \end{split}
    \end{equation}
where $U_{i} = \sum_{1 \leq j \neq k \leq m_{i}}[(Y_{ij}-\mu_{0}(t_{ij}))(Y_{ik}-\mu_{0}(t_{ik})-C_{0}(t_{ij},t_{ik}))k^{\frac{1}{2}}(t_{ij}, \cdot)k^{\frac{1}{2}}(t_{ik}, \cdot)]$.
\noindent
Consider
\begin{align*}
\begin{split}
    &\mathbb{E}\left\|\frac{1}{n}\sum_{i=1}^{n}v_{i}(\Lambda_{2}+\lambda I)^{-\frac{1}{2}}U_{i}\right\|_{L^2(T \times T)}^2  =  \frac{1}{n^2}\sum_{i=1}^{n}v_{i}^2\mathbb{E}\|(\Lambda_{2}+\lambda I)^{-\frac{1}{2}}U_{i}\|_{L^2(T \times T)}^2\\
    & =   \frac{1}{n^2}\sum_{i=1}^{n}v_{i}^2\sum_{\beta,\beta'}\frac{\mathbb{E}\langle \psi_{\beta}\otimes \psi_{\beta'}, U_{i}\rangle^2_{L^2(T \times T)}}{(\tau_{\beta}\tau_{\beta'}+\lambda)}.
\end{split}    
\end{align*}
Since $\mathbb{E}[U_{i}] = \Lambda_{2}^{\frac{1}{2}}C_{0}- \Lambda_{2}^{\frac{1}{2}}C_{0} = 0$, we have
\begin{align*}
    \begin{split}
       & \mathbb{E}\langle \psi_{\beta}\otimes \psi_{\beta'}, U_{i}\rangle^2_{L^2(T \times T)} = \text{Var}\left(\sum_{1 \leq j \neq k \leq m_{i}}(W_{i}(j,k)-C_{0}(t_{ij},t_{ik}))\phi_{\beta}(t_{ij})\phi_{\beta'}(t_{ik})\right)\\
        &= \text{Var}_{T}\left(\mathbb{E}_{X|T}\left[\sum_{1 \leq j \neq k \leq m_{i}}(W_{i}(j,k)-C_{0}(t_{ij},t_{ik}))\phi_{\beta}(t_{ij})\phi_{\beta'}(t_{ik})\Big|T\right]\right)\\
        &\qquad\qquad+ \mathbb{E}_{T}\left(\text{Var}_{X|T}\left[\sum_{1 \leq j \neq k \leq m_{i}}(W_{i}(j,k)-C_{0}(t_{ij},t_{ik}))\phi_{\beta}(t_{ij})\phi_{\beta'}(t_{ik})\Big|T\right]\right),
    \end{split}
\end{align*}
where $W_{i}(j,k) = (Y_{ij}-\mu_{0}(t_{ij}))(Y_{ik}-\mu_{0}(t_{ik})),~1 \leq j\neq k \leq m_{i}$.
\noindent
Note that
\begin{align*}
    \begin{split}
        & \text{Var}_{T}\left(\mathbb{E}_{X|T}\left[\sum_{1 \leq j \neq k \leq m_{i}}(W_{i}(j,k)-C_{0}(t_{ij},t_{ik}))\phi_{\beta}(t_{ij})\phi_{\beta'}(t_{ik})\Big|T\right]\right)\\
        &=  \text{Var}_{T}\left(\sum_{1 \leq j \neq k \leq m_{i}}(C_{0}(t_{ij},t_{ik})-C_{0}(t_{ij},t_{ik}))\phi_{\beta}(t_{ij})\phi_{\beta'}(t_{ik})\right)
        =  0,
    \end{split}
\end{align*}
and 
\begin{align*}
    \begin{split}
       & \mathbb{E}_{T}\left(\text{Var}_{X|T}\left[\sum_{1 \leq j \neq k \leq m_{i}}(W_{i}(j,k)-C_{0}(t_{ij},t_{ik}))\phi_{\beta}(t_{ij})\phi_{\beta'}(t_{ik})\Big|T\right]\right)\\
       & \leq  \mathbb{E}_{T}\left(\text{Var}_{X|T}\left[\sum_{1 \leq j \neq k \leq m_{i}}W_{i}(j,k)\phi_{\beta}(t_{ij})\phi_{\beta'}(t_{ik})\Big|T\right]\right)\\
       & =   \mathbb{E}_{T}\left(\text{Var}_{X|T}\left[\sum_{1\leq j \neq k \leq m_{i}}(B_{ij}B_{ik}+ \epsilon_{ij}B_{ik}+B_{ij}\epsilon_{ik}+\epsilon_{ij}\epsilon_{ik})\phi_{\beta}(t_{ij})\phi_{\beta'}(t_{ik})\Big|T\right]\right)\\
       & \lesssim  \mathbb{E}[U_{1}^2+U_{2}^2+U_{3}^2+U_{4}^2],
    \end{split}
\end{align*}
where $B_{ij}= (X_{i}(t_{ij})-\mu_{0}(t_{ij}))$, 
$$
U_{1} = \sum_{1\leq j \neq k \leq m_{i}}B_{ij}B_{ik}\phi_{\beta}(t_{ij})\phi_{\beta'}(t_{ik}),\,\, U_{2}= \sum_{1\leq j \neq k \leq m_{i}}\epsilon_{ij}B_{ik}\phi_{\beta}(t_{ij})\phi_{\beta'}(t_{ik})
$$
$$U_{3}= \sum_{1\leq j \neq k \leq m_{i}}B_{ij}\epsilon_{ik}\phi_{\beta}(t_{ij})\phi_{\beta'}(t_{ik}),\,\,\text{and}\,\, U_{4}= \sum_{1\leq j \neq k \leq m_{i}}\epsilon_{ij}\epsilon_{ik}\phi_{\beta}(t_{ij})\phi_{\beta'}(t_{ik}).$$
\noindent
\textbf{Bound for \textit{$U_{1}$:}}
\begin{equation*}
    \begin{split}
        \mathbb{E}[U_{1}^2] = & \mathbb{E}\left[\sum_{1 \leq j \neq k \leq m_{i}}B_{ij}B_{ik}\phi_{\beta}(t_{ij})\phi_{\beta'}(t_{ik})\right]^2\\
        = & \mathbb{E}\left[\sum_{1 \leq j \neq k \leq m_{i}}\sum_{1 \leq j' \neq k' \leq m_{i}}B_{ij}B_{ik}B_{ij'}B_{ik'}\phi_{\beta}(t_{ij})\phi_{\beta'}(t_{ik})\phi_{\beta}(t_{ij'})\phi_{\beta'}(t_{ik'})\right].
    \end{split}
\end{equation*}
We will divide indices in three cases and will bound each case seperately.\\

\noindent
\underline{\textit{Case-1 $(\{j,k\}\cap\{j',k'\} = \emptyset)$}}
\begin{equation*}
    \begin{split}
       & \mathbb{E}[B_{ij}B_{ik}B_{ij'}B_{ik'}\phi_{\beta}(t_{ij})\phi_{\beta'}(t_{ik})\phi_{\beta}(t_{ij'})\phi_{\beta'}(t_{ik'}))]\\
      & =  \mathbb{E}_{X}\left[\langle X_{i}-\mu_{0}, \phi_{\beta}\rangle_{L^2(T)}^2\langle X_{i}-\mu_{0}, \phi_{\beta'}\rangle_{L^2(T)}^2\right].
    \end{split}
\end{equation*}

\noindent
\underline{\textit{Case-2 $((j,k)= (j',k'))$}}
\begin{equation*}
    \begin{split}
        & \mathbb{E}[B_{ij}B_{ik}B_{ij'}B_{ik'}\phi_{\beta}(t_{ij})\phi_{\beta'}(t_{ik})\phi_{\beta}(t_{ij'})\phi_{\beta'}(t_{ik'}))]\\
        & = \mathbb{E}_{X}\left[\int_{T \times T}(X_{i}(s)-\mu_{0}(s))^2(X_{i}(t)-\mu_{0}(t))^2\phi_{\beta}^2(s)\phi_{\beta'}^2(t)\, ds dt\right]\\
        & \leq  \int_{T \times T}(\mathbb{E}[(X_{i}(s)-\mu_{0}(s))^4])^{\frac{1}{2}}(\mathbb{E}[(X_{i}(t)-\mu_{0}(t))^4])^{\frac{1}{2}} \phi_{\beta}^2(s) \phi_{\beta'}^2(t) \,dsdt
         \lesssim  \tau_{\beta} \tau_{\beta'},
    \end{split}
\end{equation*}
where last step uses $\mathbb{E}[X^4(t)] \lesssim \mathbb{E}[X^2(t)]^2$ for a.e. $t \in T$.\\

\noindent
\underline{\textit{Case-3 $(j=j',k\neq k')$}}
\begin{equation*}
    \begin{split}
        & \mathbb{E}[B_{ij}B_{ik}B_{ij'}B_{ik}'\phi_{\beta}(t_{ij})\phi_{\beta'}(t_{ik})\phi_{\beta}(t_{ij'})\phi_{\beta'}(t_{ik'}))]\\
        & \leq  \mathbb{E}_{X}\left[\int_{T}(X_{i}(s)-\mu_{0}(s))^2\phi_{\beta}^2(s)ds \langle X_{i}-\mu_{0},\phi_{\beta'} \rangle_{L^2(T)}^2\right].
    \end{split}
\end{equation*}
All other remaining combinations of indices follows similarly and gives same bounds.\\

\noindent
\textbf{Bound for \textit{$U_{2}$:}}
\begin{equation*}
\begin{split}
    \mathbb{E}[U_{2}^2]  & =  \mathbb{E}\left[\sum_{1\leq j \neq k \leq m_{i}}\epsilon_{ij}B_{ik}\phi_{\beta}(t_{ij})\phi_{\beta'}(t_{ik})\right]^2\\
    & =  \mathbb{E}\left[\sum_{1 \leq j \neq k \leq m_{i}}\sum_{1 \leq j' \neq k' \leq m_{i}}\epsilon_{ij}\epsilon_{ij'}B_{ik}B_{ik}'\phi_{\beta}(t_{ij})\phi_{\beta'}(t_{ik})\phi_{\beta}(t_{ij'})\phi_{\beta'}(t_{ik'})\right].
    \end{split}
\end{equation*}
Similar to $U_{1}$, we will make three cases:\\

\noindent
\underline{\textit{Case-1 $(\{j,k\}\cap\{j',k'\} = \emptyset)$}}
\begin{equation*}
    \begin{split}
       & \mathbb{E}[\epsilon_{ij}\epsilon_{ij'}B_{ik}B_{ik}'\phi_{\beta}(t_{ij})\phi_{\beta'}(t_{ik})\phi_{\beta}(t_{ij'})\phi_{\beta'}(t_{ik'})]\\
       & =  \mathbb{E}_{T}\left[\mathbb{E}_{X}[\epsilon_{ij}\epsilon_{ij'}B_{ik}B_{ik'}]\phi_{\beta}(t_{ij})\phi_{\beta'}(t_{ik})\phi_{\beta}(t_{ij'})\phi_{\beta'}(t_{ik'})\right] = 0. 
    \end{split}
\end{equation*}

\noindent
\underline{\textit{Case-2 $((j,k)=(j',k'))$}}
\begin{equation*}
    \begin{split}
        & \mathbb{E}[\epsilon_{ij}\epsilon_{ij'}B_{ik}B_{ik}'\phi_{\beta}(t_{ij})\phi_{\beta'}(t_{ik})\phi_{\beta}(t_{ij'})\phi_{\beta'}(t_{ik'})]\\
       & =  \mathbb{E}_{X}\left[\int_{T \times T}(X_{i}(s)-\mu_{0}(s))^2\epsilon_{ij}^2\phi^2_{\beta}(t_{ij})\phi^2_{\beta'}(t_{ik})\,dsdt\right]\\
       & \leq  \sigma_{0}^2\int_{T \times T} \mathbb{E}(X_{i}(s)-\mu_{0}(s))^2 \phi^2_{\beta}(t_{ij})\phi^2_{\beta'}(t_{ik})\,dsdt \lesssim \tau_{\beta}\tau_{\beta'}.
    \end{split}
\end{equation*}

\noindent
\underline{\textit{Case-3 $(j=j',k \neq k')$}}
\begin{equation*}
    \begin{split}
        & \mathbb{E}[\epsilon_{ij}\epsilon_{ij'}B_{ik}B_{ik}'\phi_{\beta}(t_{ij})\phi_{\beta'}(t_{ik})\phi_{\beta}(t_{ij'})\phi_{\beta'}(t_{ik'})]\\ 
        & \leq  \mathbb{E}_{X}\left[\int_{T}\epsilon_{ij}^2\phi_{\beta}^2(s)ds \langle X_{i}-\mu_{0}, \phi_{\beta'}\rangle_{L^2(T)}^2\right].
    \end{split}
\end{equation*}

As $\epsilon_{ij}'$s are independent of $X_{i}'$s so all the indices when $j \neq j'$ will give $0$ as result as $\mathbb{E}[\epsilon_{ij}\epsilon_{ij'}] =0$. Since the bound for $U_{3}$ is the same as that of $U_{2}$, we move to bounding $U_{4}$.\\

\noindent
\textbf{Bound for \textit{$U_{4}$:}}
\begin{equation*}
    \begin{split}
        \mathbb{E}[U_{4}^2] =  \mathbb{E}[\sum_{1 \leq j \neq k \leq m_{i}}\sum_{1 \leq j' \neq k' \leq m_{i}}\epsilon_{ij}\epsilon_{ij'}\epsilon_{ik}\epsilon_{ik'}\phi_{\beta}(t_{ij})\phi_{\beta'}(t_{ik})\phi_{\beta}(t_{ij'})\phi_{\beta'}(t_{ik'})]
        \lesssim  m_{i}^{\underline{2}}\tau_{\beta} \tau_{\beta'},
    \end{split}
\end{equation*}
where only paired indices $(j=j', k=k')$ and $(j=k',j'=k)$ survives.

\noindent
Putting all these bounds together, we obtain
\begin{equation*}
    \begin{split}
        & \mathbb{E}\left\|\frac{1}{n}\sum_{i=1}^{n}(\Lambda_{2}+\lambda I)^{-\frac{1}{2}}U_{i}\right\|_{L^2(T \times T)}^2 \\
        & \lesssim  \frac{1}{n^2}\sum_{i=1}^{n}v_{i}^2 m_{i}^{\underline{4}} + \frac{1}{n^2}\sum_{i=1}^{n}v_{i}^2m_{i}^{\underline{2}}\sum_{\beta, \beta'}\frac{\tau_{\beta}\tau_{\beta'}}{(\tau_{\beta}\tau_{\beta'}+\lambda)} + \frac{1}{n^2}\sum_{i=1}^{n}v_{i}^2 m_{i}^{\underline{3}}\sum_{\beta, \beta'}\frac{\tau_{\beta} \tau_{\beta'} \mathbb{E}\langle X_{i}-\mu_{0}, \psi_{\beta'} \rangle_{L^2(T)}^2}{(\tau_{\beta}\tau_{\beta'}+\lambda)}\\
        \lesssim & \frac{1}{n^2}\sum_{i=1}^{n}v_{i}^2 m_{i}^{\underline{4}} + \frac{1}{n^2}\sum_{i=1}^{n}v_{i}^2m_{i}^{\underline{2}} \lambda^{-\frac{1}{b_{1}}} \log\left(\frac{1}{\lambda}\right) + \frac{1}{n^2}\sum_{i=1}^{n}v_{i}^2 m_{i}^{\underline{3}} \lambda^{-\frac{1}{b_{1}}}
    \end{split}
\end{equation*}
Using the Chebyshev's inequality, we get
\begin{equation*}
\begin{split}
   & \|(\Lambda_{2}+\lambda I)^{-\frac{1}{2}}(O_{3}-\hat{T}_{n}\Lambda_{2}^{\alpha_{1}-\frac{1}{2}}H)\|_{L^2(T \times T)}\\
   & \lesssim_p \sqrt{\frac{1}{n^2}\sum_{i=1}^{n}v_{i}^2 m_{i}^{\underline{4}}} + \sqrt{\frac{1}{n^2}\sum_{i=1}^{n}v_{i}^2m_{i}^{\underline{2}} \lambda^{-\frac{1}{b_{1}}} \log\left(\frac{1}{\lambda}\right)} + \sqrt{\frac{1}{n^2}\sum_{i=1}^{n}v_{i}^2 m_{i}^{\underline{3}} \lambda^{-\frac{1}{b_{1}}}}
\end{split}
\end{equation*}
\noindent
For

\begin{equation*}
    \left\|(\Lambda_{2}+\lambda I)^{-\frac{1}{2}}(O_{3}-\Lambda^{\frac{1}{2}}_{2}C_{0})\right\|_{L^2(T \times T)} =\left\|\frac{1}{n}\sum_{i=1}^{n}v_{i}(\Lambda_{2}+\lambda I)^{-\frac{1}{2}}U_{i}\right\|_{L^2(T \times T)},
\end{equation*}
consider $$U_{i} := \sum_{1 \leq j\neq k}^{m_{i}}[(Y_{ij}-\mu_{0}(t_{ij}))(Y_{ik}-\mu_{0}(t_{ik}))k^{\frac{1}{2}}(t_{ij},\cdot)k^{\frac{1}{2}}(t_{ik},\cdot)-\Lambda_{2}^{\frac{1}{2}}C_{0}].$$ Then clearly $\mathbb{E}\langle U_{i}, \psi_{\beta}\otimes\psi_{\beta'}\rangle_{L^2(T \times T)} =0$, and therefore, 
\begin{equation*}
    \begin{split}
        & \mathbb{E}\langle U_{i}, \psi_{\beta}\otimes \psi_{\beta'}\rangle_{L^2(T \times T)}^2 =  \text{Var}\left(\sum_{1 \leq j \neq k \leq m_{i}}(W_{i}(j,k)\phi_{\beta}(t_{ij})\phi_{\beta'}(t_{ik})-\langle \Lambda^{\frac{1}{2}}_{2}C_{0}, \psi_{\beta}\otimes \psi_{\beta'}\rangle_{L^2(T \times T)}) \right)\\
        & =  \text{Var}_{T}\left(\mathbb{E}_{X|T}\left[\sum_{1\leq j\neq k \leq m_{i}}(W_{i}(j,k)\phi_{\beta}(t_{ij})\phi_{\beta'}(t_{ik})-\langle \Lambda^{\frac{1}{2}}_{2}C_{0}, \psi_{\beta}\otimes \psi_{\beta'}\rangle_{L^2(T \times T)})\right]\right)\\
        & \qquad+ \mathbb{E}_{T}\left(\text{Var}_{X|T}\left[\sum_{1\leq j\neq k \leq m_{i}}(W_{i}(j,k)\phi_{\beta}(t_{ij})\phi_{\beta'}(t_{ik})-\langle \Lambda^{\frac{1}{2}}_{2}C_{0}, \psi_{\beta}\otimes \psi_{\beta'}\rangle_{L^2(T \times T)})\right]\right),
    \end{split}
\end{equation*}
where $W_{i}(j,k) = (Y_{ij}-\mu_{0}(t_{ij}))(Y_{ik}-\mu_{0}(t_{ik})),~1\leq j \neq k \leq m_{i}$. The second term follows exactly as that of the first result and so we consider only the first term.
\begin{equation*}
    \begin{split}
        & \text{Var}_{T}\left(\mathbb{E}_{X|T}\left[\sum_{1\leq j\neq k \leq m_{i}}(W_{i}(j,k)\phi_{\beta}(t_{ij})\phi_{\beta'}(t_{ik})-\langle \Lambda^{\frac{1}{2}}_{2}C_{0}, \psi_{\beta}\otimes \psi_{\beta'}\rangle_{L^2(T \times T)})\right]\right)\\
        &=  \text{Var}_{T}\left(\sum_{1\leq j \neq k \leq m_{i}}(C_{0}(t_{ij},t_{ik})\phi_{\beta}(t_{ij})\phi_{\beta'}(t_{ik})-\langle \Lambda^{\frac{1}{2}}_{2}C_{0}, \psi_{\beta}\otimes \psi_{\beta'}\rangle_{L^2(T \times T)})\right)\\
        & \leq   \mathbb{E}_{T}\left(\sum_{1\leq j \neq k \leq m_{i}}(C_{0}(t_{ij},t_{ik})\phi_{\beta}(t_{ij})\phi_{\beta'}(t_{ik})-\langle \Lambda^{\frac{1}{2}}_{2}C_{0}, \psi_{\beta}\otimes \psi_{\beta'}\rangle_{L^2(T \times T)})\right)^2\\
        & \leq  \mathbb{E}_{T}\left(\sum_{1 \leq j \neq k \leq m_{i}}\sum_{1 \leq j'\neq k'\leq m_{i}}C_{0}(t_{ij},t_{ik})C_{0}(t_{ij'},t_{ik'})\phi_{\beta}(t_{ij})\phi_{\beta'}(t_{ik})\phi_{\beta}(t_{ij'})\phi_{\beta'}(t_{ik'})\right).
    \end{split}
\end{equation*}
We consider three cases over indices to bound this term.\\

\noindent
\underline{\textit{Case-1} $ (\{j,k\}\cap\{j',k'\} = \emptyset)$}
\begin{equation*}
    \begin{split}
        \mathbb{E}_{T}[C_{0}(t_{ij},t_{ik})\phi_{\beta}(t_{ij}) \phi_{\beta'}(t_{ik})C_{0}(t_{ij'},t_{ik'})\phi_{\beta}(t_{ij'})\phi_{\beta'}(t_{ik'})] = \langle C_{0}, \phi_{\beta}\otimes\phi_{\beta'}\rangle^2_{L^2(T\times T)}.
    \end{split}
\end{equation*}

\noindent
\underline{\textit{Case-2} $ ((j,k)= (j',k'))$}
\begin{equation*}
    \begin{split}
        \mathbb{E}_{T}[C_{0}(t_{ij},t_{ik})\phi_{\beta}(t_{ij}) \phi_{\beta'}(t_{ik})C_{0}(t_{ij'},t_{ik'})\phi_{\beta}(t_{ij'})\phi_{\beta'}(t_{ik'})] =&  \int_{T \times T}C_{0}^2(s,t) \phi_{\beta}^2(s) \phi_{\beta'}^2(t)\,dsdt\\
        \lesssim \int_{T \times T}\phi_{\beta}^2(s) \phi_{\beta'}^2(t)\,dsdt  = \tau_{\beta} \tau_{\beta'}.
    \end{split}
\end{equation*}

\noindent
\underline{\textit{Case-3} $ (j = j', k \neq k')$}
\begin{equation*}
\begin{split}
    &\mathbb{E}_{T}[C_{0}(t_{ij},t_{ik})\phi_{\beta}(t_{ij}) \phi_{\beta'}(t_{ik})C_{0}(t_{ij'},t_{ik'})\phi_{\beta}(t_{ij'})\phi_{\beta'}(t_{ik'})]] \\
    & = \int_{T\times T \times T}C_{0}(s,t)\phi^2_{\beta}(s)\phi_{\beta'}(t)C_{0}(s,w)\phi_{\beta'}(w)\, ds dt dw.
\end{split}
\end{equation*}
By utilizing the fact that $\esssup_{s \in T} \sum_{k}\tau_{k}^{2\alpha}\psi_{k}^2(s) < \infty$, we get
\begin{equation*}
    \begin{split}
 & \mathbb{E}\left\|\frac{1}{n}\sum_{i=1}^{n}(\Lambda_{2}+\lambda I)^{-\frac{1}{2}}U_{i}\right\|^2_{L^2(T \times T)}\\
 \lesssim &
\frac{1}{n^2}\sum_{i=1}^{n}v_{i}^2 m_{i}^{\underline{4}} + \frac{1}{n^2}\sum_{i=1}^{n}v_{i}^2m_{i}^{\underline{2}} \lambda^{-\frac{1}{b_{1}}} \log\left(\frac{1}{\lambda}\right) + \frac{1}{n^2}\sum_{i=1}^{n}v_{i}^2 m_{i}^{\underline{3}} \lambda^{-\frac{1}{b_{1}}}.
    \end{split}
\end{equation*}
Therefore, with similar calculation as the previous result, the current one follows.

\end{proof}

\begin{lemma}\label{covarianve_empirical_operator_without_embedding}
    Under Assumptions~\ref{mean_moment} and \ref{covariance_eigen_decay}, we have
    \begin{equation*}
        \|(\Lambda_{2}+\lambda I)^{-1}(\hat{T}_{n}-\Lambda_{2})\|_{\emph{op}} \lesssim_p \sqrt{\frac{1}{n^2}\sum_{i=1}^{n}v_{i}^2 m_{i}^{\underline{3}} \lambda^{-1-\frac{1}{b_{1}}}} + \sqrt{\frac{1}{n^2}\sum_{i=1}^{n}v_{i}^2 m_{i}^{\underline{2}} \lambda^{-1-\frac{1}{b_{1}}} \log\left(\frac{1}{\lambda}\right)}.
    \end{equation*}
\end{lemma}
\begin{proof}

We start with
    \begin{align*}
        \begin{split}
           & \|(\Lambda_{2}+\lambda I)^{-1}(\hat{T}_{n}-\Lambda_{2})\|_{\text{op}}\\
           &=  \sup_{g,h \in L^2(T \times T), \|g\|=\|h\|=1} |\langle g, (\Lambda_{2}+\lambda I)^{-1}(\hat{T}_{n}-\Lambda_{2})h \rangle_{L^2(T \times T)}|\\
            &=  \sup_{g,h \in L^2(T \times T), \|g\|=\|h\|=1} \left|\sum_{l, l'}\sum_{\beta, \beta'}\frac{\langle g, \psi_{l}\otimes \psi_{l'}\rangle_{L^2}  \langle h, \psi_{\beta}\otimes \psi_{\beta'}\rangle_{L^2} }{(\tau_{l}\tau_{l'}+\lambda)}\langle \psi_{l}\otimes \psi_{l'}, (\hat{T}_{n}-\Lambda_{2}) \psi_{\beta} \otimes \psi_{\beta'}\rangle_{L^2(T \times T)}\right| .
        \end{split}
    \end{align*}

\noindent
By applying Cauchy-Schwarz, we see that
\begin{align*}
    \begin{split}
       & \|(\Lambda_{2}+\lambda I)^{-1}(\hat{T}_{n}-\Lambda_{2})\|_{\text{op}}\\
       &\leq  \sup_{g,h \in L^2(T \times T),\|g\|=\|h\|=1}  \left(\sum_{l, l'}\sum_{\beta, \beta'}\frac{
        \langle \psi_{l}\otimes \psi_{l'}, (\hat{T}_{n}-\Lambda_{2}) \psi_{\beta}\otimes \psi_{\beta'} \rangle^2_{L^2(T \times T)} }{(\tau_{l} \tau_{l'}+\lambda)^2}\right)^{\frac{1}{2}} \|g\|_{L^2(T\times T)}\|h\|_{L^2(T\times T)}\\
        &\leq  \left(\sum_{l, l'}\sum_{\beta, \beta'}\frac{
        \langle \psi_{l}\otimes \psi_{l'}, (\hat{T}_{n}-\Lambda_{2}) \psi_{\beta}\otimes \psi_{\beta'} \rangle^2_{L^2(T \times T)} }{(\tau_{l} \tau_{l'}+\lambda)^2}\right)^{\frac{1}{2}}.
    \end{split}
\end{align*}

\noindent
Taking expectation on both sides and applying Jensen's inequality, we obtain
\begin{equation*}
    \begin{split}
       \mathbb{E} \|(\Lambda_{2}+\lambda I)^{-1}(\hat{T}_{n}-\Lambda_{2})\|_{\text{op}} \leq \left(\sum_{l, l'}\sum_{\beta, \beta'}\frac{
        \mathbb{E} \langle \psi_{l}\otimes \psi_{l'}, (\hat{T}_{n}-\Lambda_{2}) \psi_{\beta}\otimes \psi_{\beta'} \rangle^2_{L^2(T \times T)} }{(\tau_{l} \tau_{l'}+\lambda)^2}\right)^{\frac{1}{2}}.
    \end{split}
\end{equation*}
Now, consider
\begin{equation*}
    \begin{split}
\mathbb{E}\langle \psi_{l}\otimes \psi_{l'}, (\hat{T}_{n}-\Lambda_{2}) \psi_{\beta}\otimes \psi_{\beta'} \rangle^2_{L^2(T \times T)} & =  \mathbb{E}\left[\frac{1}{n}\sum_{i=1}^{n}v_{i}\langle \psi_{l}\otimes \psi_{l'}, U_{i} (\psi_{\beta}\otimes \psi_{\beta'} )\rangle^2_{L^2(T \times T)}\right]^2\\
& =  \frac{1}{n^2} \sum_{i=1}^{n}v_{i}^2\mathbb{E}\langle \psi_{l}\otimes \psi_{l'}, U_{i} (\psi_{\beta}\otimes \psi_{\beta'}) \rangle^2_{L^2(T \times T)},
    \end{split}
\end{equation*}
where $U_{i}= \sum_{1\leq j \neq k \leq m_{i}} [k^{\frac{1}{2}}(t_{ij},\cdot)k^{\frac{1}{2}}(t_{ik},\cdot) \otimes k^{\frac{1}{2}}(t_{ij},\cdot)k^{\frac{1}{2}}(t_{ik},\cdot)- \Lambda_{2}]$
and 
\begin{equation*}
    \begin{split}
       & \mathbb{E}\langle \psi_{l}\otimes \psi_{l'}, (\hat{T}_{n}-\Lambda_{2}) \psi_{\beta}\otimes \psi_{\beta'} \rangle^2_{L^2(T\times T)}\\
       & \leq  \frac{1}{n^2}\sum_{i=1}^{n}v_{i}^2\mathbb{E}\left[\sum_{1\leq j \neq k \leq m_{i}}\sum_{1\leq j' \neq k' \leq m_{i}}(\phi_{l}(t_{ij})\phi_{l'}(t_{ik}) \phi_{\beta}(t_{ij})\phi_{\beta'}(t_{ik})  - \langle \psi_{l}\otimes \psi_{l'}, \Lambda_{2}\psi_{\beta}\otimes \psi_{\beta'}\rangle_{L^2(T \times T)})\right.\\
       & \qquad \qquad \qquad \qquad\times \left.((\phi_{l}(t_{ij'})\phi_{l'}(t_{ik'}) \phi_{\beta}(t_{ij'})\phi_{\beta'}(t_{ik'})  - \langle \psi_{l}\otimes \psi_{l'}, \Lambda_{2}\psi_{\beta}\otimes \psi_{\beta'}\rangle_{L^2(T \times T)}) \right].
    \end{split}
\end{equation*}
We make three cases over indices to bound each one separately.\\

\noindent
\underline{\textit{Case-1 $(\{j,k\}\cap \{j',k'\} = \emptyset)$}}
\begin{equation*}
    \begin{split}
       & \mathbb{E}\left[(\phi_{l}(t_{ij})\phi_{l'}(t_{ik}) \phi_{\beta}(t_{ij})\phi_{\beta'}(t_{ik})  - \langle \psi_{l}\otimes \psi_{l'}, \Lambda_{2}\psi_{\beta}\otimes \psi_{\beta'}\rangle_{L^2(T \times T)})\right.\\
       &
       \qquad\left.\times (\phi_{l}(t_{ij'})\phi_{l'}(t_{ik'}) \phi_{\beta}(t_{ij'})\phi_{\beta'}(t_{ik'})  - \langle \psi_{l}\otimes \psi_{l'}, \Lambda_{2}\psi_{\beta}\otimes \psi_{\beta'}\rangle_{L^2(T \times T)}) \right] \\
        & =  (\langle \phi_{l}\otimes \phi_{l'}, \phi_{\beta}\otimes \phi_{\beta'} \rangle_{L^2(T \times T)}-\langle \phi_{l}\otimes \phi_{l'}, \phi_{\beta}\otimes \phi_{\beta'} \rangle_{L^2(T \times T)})^2 = 0.
    \end{split}
\end{equation*}

\noindent
\underline{\textit{Case-2 $((j,k)= (j',k'))$}}
\begin{equation*}
    \begin{split}
        & \mathbb{E}\left[(\phi_{l}(t_{ij})\phi_{l'}(t_{ik}) \phi_{\beta}(t_{ij})\phi_{\beta'}(t_{ik})  - \langle \psi_{l}\otimes \psi_{l'}, \Lambda_{2}\psi_{\beta}\otimes \psi_{\beta'}\rangle_{L^2(T \times T)})\right.\\
        &\qquad\left.\times(\phi_{l}(t_{ij'})\phi_{l'}(t_{ik'}) \phi_{\beta}(t_{ij'})\phi_{\beta'}(t_{ik'})  - \langle \psi_{l}\otimes \psi_{l'}, \Lambda_{2}\psi_{\beta}\otimes \psi_{\beta'}\rangle_{L^2(T \times T)}) \right]\\
        & =  \int_{T \times T}\phi_{l}^2(s)\phi_{l'}^2(t)\phi_{\beta}^2(s)\phi_{\beta'}^2(t),\ ds dt- \langle \psi_{l}\otimes \psi_{l'}, \Lambda_{2}\psi_{\beta}\otimes \psi_{\beta'}\rangle^2_{L^2(T \times T)}\\
        & \leq  \int_{T}\phi_{l}^2(s)\phi_{\beta}^2(s)\,ds \int_{T}\phi_{l'}^2(t)\phi_{\beta'}^2(t)dt.
    \end{split}
\end{equation*}

\noindent
\underline{\textit{Case-3 $(j= j', k \neq k')$}}
\begin{equation*}
    \begin{split}
        & \mathbb{E}\left[(\phi_{l}(t_{ij})\phi_{l'}(t_{ik}) \phi_{\beta}(t_{ij})\phi_{\beta'}(t_{ik})  - \langle \psi_{l}\otimes \psi_{l'}, \Lambda_{2}\psi_{\beta}\otimes \psi_{\beta'}\rangle_{L^2(T \times T)})\right.\\
        &\qquad\left.\times(\phi_{l}(t_{ij'})\phi_{l'}(t_{ik'}) \phi_{\beta}(t_{ij'})\phi_{\beta'}(t_{ik'})  - \langle \psi_{l}\otimes \psi_{l'}, \Lambda_{2}\psi_{\beta}\otimes \psi_{\beta'}\rangle_{L^2(T \times T)}) \right]\\
        & \leq \mathbb{E}\left[\phi_{l}(t_{ij})\phi_{l'}(t_{ik}) \phi_{\beta}(t_{ij})\phi_{\beta'}(t_{ik}) \phi_{l}(t_{ij'})\phi_{l'}(t_{ik'}) \phi_{\beta}(t_{ij'})\phi_{\beta'}(t_{ik'}) \right]\\
        & \leq  \int_{T} \phi_{l}^2(s)\phi_{\beta}^2(s)ds \langle \phi_{l'}, \phi_{\beta'}\rangle^2_{L^2(T)}.
    \end{split}
\end{equation*}

\noindent
Combining these bounds, we have
\begin{equation*}
    \begin{split}
        \mathbb{E}\|(\Lambda_{2}+\lambda I)^{-1}(\hat{T}_{n}-\Lambda_{2})\|_{\text{op}} &\lesssim  \left(\frac{1}{n^2}\sum_{i=1}^{n}v_{i}^2m_{i}^{\underline{3}}\sum_{l, l'}\sum_{\beta,\beta'}\frac{\int_{T} \phi_{l}^2(s)\phi_{\beta}^2(s)ds \langle \phi_{l'}, \phi_{\beta'}\rangle^2_{L^2(T)}}{(\tau_{l} \tau_{l'}+\lambda)^2}\right)^{\frac{1}{2}}\\
        & + \left(\frac{1}{n^2}\sum_{i=1}^{n}v_{i}^2 m_{i}^{\underline{2}}\sum_{l,l'}\sum_{\beta, \beta'} \frac{\int_{T}\phi_{l}^2(s)\phi_{\beta}^2(s)\,ds \int_{T}\phi_{l'}^2(t)\phi_{\beta'}^2(t)dt}{(\tau_{l}\tau_{l'}+\lambda)^2}\right)^{\frac{1}{2}}.
    \end{split}
\end{equation*}
Using the fact that $\sum_{\beta}\phi_{\beta}^2(s) < \infty$, we get
\begin{equation*}
    \begin{split}
        \mathbb{E}\|(\Lambda_{2}+\lambda I)^{-1}(\hat{T}_{n}-\Lambda_{2})\|_{\text{op}} \lesssim &  \sqrt{\frac{1}{n^2}\sum_{i=1}^{n}v_{i}^2 m_{i}^{\underline{3}} \lambda^{-1-\frac{1}{b_{1}}}} + \sqrt{\frac{1}{n^2}\sum_{i=1}^{n}v_{i}^2 m_{i}^{\underline{2}} \lambda^{-1-\frac{1}{b_{1}}} \log\left(\frac{1}{\lambda}\right)}.
    \end{split}
\end{equation*}
Then the result follows from Markov's inequality.
\end{proof}

\begin{lemma}\label{covariance_empirical_operator_with_embedding}
Under Assumptions~\ref{mean_moment}, \ref{covariance_eigen_decay} and \ref{covariance_embedding}, we have
    \begin{equation*}
    \begin{split}
        & \|(\Lambda_{2}+\lambda I)^{-\frac{1}{2}}(\Lambda_{2}-\hat{T}_{n})(\Lambda_{2}+\lambda I)^{-\frac{1}{2}}\|_{\emph{op}}\\
        \lesssim_p & \sqrt{\frac{1}{n^2}\sum_{i=1}^{n}v_{i}^2 m_{i}^{\underline{3}}\lambda^{-\frac{1}{b_{1}}-2\alpha_{1}}} + \sqrt{\frac{1}{n^2}\sum_{i=1}^{n}v_{i}^2 m_{i}^{\underline{2}}\lambda^{-\frac{1}{b_{1}}-2\alpha_{1}}\log\left(\frac{1}{\lambda}\right)},
    \end{split}
    \end{equation*}
    and
    \begin{equation*}
    \begin{split}
        & \|(\Lambda_{2}+\lambda I)^{-\frac{1}{2}}(\Lambda_{2}-\hat{T}_{n})(\Lambda_{2}+\lambda I)^{-1}\Lambda_{2}^{\frac{1}{2}+\alpha_{1}}\|_{\emph{op}}\\
        \lesssim_p & \sqrt{\frac{1}{n^2}\sum_{i=1}^{n}v_{i}^2 m_{i}^{\underline{3}}\lambda^{-\frac{1}{b_{1}}}} + \sqrt{\frac{1}{n^2}\sum_{i=1}^{n}v_{i}^2 m_{i}^{\underline{2}}\lambda^{-\frac{1}{b_{1}}}\log\left(\frac{1}{\lambda}\right)}.
            \end{split}
    \end{equation*}
\end{lemma}
\begin{proof}
    Since the proof idea is similar to that of Lemma~\ref{covarianve_empirical_operator_without_embedding}, we only provide necessary details. Note that 
\begin{equation*}
\begin{split}
& \|(\Lambda_{2}+\lambda I)^{-\frac{1}{2}}(\Lambda_{2}-\hat{T}_{n})(\Lambda_{2}+\lambda I)^{-\frac{1}{2}}\|_{\text{op}}\\
&=  \sup_{h \in L^2(T \times T),~ \|h\|=1} \left|\sum_{l,l'}\sum_{\beta, \beta'} \frac{\langle h, \psi_{l}\otimes \psi_{l'} \rangle_{L^2(T \times T)} \langle h, \psi_{\beta} \otimes \psi_{\beta'} \rangle_{L^2(T \times T)} }{(\tau_{l}\tau_{l'}+\lambda )^{\frac{1}{2}}(\tau_{\beta} \tau_{\beta'}+\lambda )^{\frac{1}{2}}} \langle \psi_{l}\otimes \psi_{l'}, (\Lambda_{2}-\hat{T}_{n})\psi_{\beta}\otimes \psi_{\beta'} \rangle_{L^2(T \times T)}\right|. 
\end{split}
\end{equation*}
By applying Cauchy-Schwarz, we get
\begin{align*}
    \begin{split}
        & \|(\Lambda_{2}+\lambda I)^{-\frac{1}{2}}(\Lambda_{2}-\hat{T}_{n})(\Lambda_{2}+\lambda I)^{-\frac{1}{2}}\|_{\text{op}}\\
        &\leq  \sup_{h \in L^2(T \times T), \|h\|= 1} \left(\sum_{l,l'}\sum_{\beta, \beta'} \frac{ \langle \psi_{l}\otimes \psi_{l'}, (\Lambda_{2}-\hat{T}_{n})\psi_{\beta}\otimes \psi_{\beta'} \rangle^2_{L^2(T \times T)}}{(\tau_{l}\tau_{l'}+\lambda )(\tau_{\beta}\tau_{\beta'}+\lambda )} \right)^{\frac{1}{2}} \|h\|_{L^2(T \times T)}\\
        &\leq  \left(\sum_{l,l'}\sum_{\beta, \beta'} \frac{ \langle \psi_{l}\otimes \psi_{l'}, (\Lambda_{2}-\hat{T}_{n})\psi_{\beta}\otimes \psi_{\beta'} \rangle^2_{L^2(T \times T)}}{(\tau_{l}\tau_{l'}+\lambda )(\tau_{\beta} \tau_{\beta'}+\lambda )} \right)^{\frac{1}{2}}.
    \end{split}
\end{align*}

\noindent
Applying Jensen's inequality, we get
\begin{align*}
    \mathbb{E}\|(\Lambda_{2}+\lambda I)^{-\frac{1}{2}}(\Lambda_{2}-\hat{T}_{n})(\Lambda_{2}+\lambda I)^{-\frac{1}{2}}\|_{\text{op}}
    \leq \left(\sum_{l,l'}\sum_{\beta,\beta'} \frac{ \mathbb{E}\langle \psi_{l}\otimes \psi_{l'}, (\Lambda_{2}-\hat{T}_{n})\psi_{\beta}\otimes \psi_{\beta'} \rangle^2_{L^2(T \times T)}}{(\tau_{l} \tau_{l'}+\lambda )(\tau_{\beta} \tau_{\beta'}+\lambda )} \right)^{\frac{1}{2}}.
\end{align*}
It follows from the proof of Lemma~\ref{covarianve_empirical_operator_without_embedding} that
\begin{equation*}
    \begin{split}
        & \mathbb{E}\|(\Lambda_{2}+\lambda I)^{-\frac{1}{2}}(\Lambda_{2}-\hat{T}_{n})(\Lambda_{2}+\lambda I)^{-\frac{1}{2}}\|_{\text{op}}
         \lesssim  \left(\frac{1}{n^2}\sum_{i=1}^{n}v_{i}^2 m_{i}^{\underline{3}}\sum_{l, l'}\sum_{\beta, \beta'} \frac{\int_{T}\phi_{l}^2(s)\phi_{\beta}^2(s)\,ds \langle \phi_{l'}, \phi_{\beta'}\rangle_{L^2(T)}^2}
        {(\tau_{l}\tau_{l'}+\lambda)(\tau_{\beta}\tau_{\beta'}+\lambda)}\right)^{\frac{1}{2}}\\
        & + \left(\frac{1}{n^2}\sum_{i=1}^{n}v_{i}^2 m_{i}^{\underline{2}}\sum_{l,l'}\sum_{\beta, \beta'} \frac{\int_{T}\phi_{l}^2(s)\phi_{\beta}^2(s)\,ds \int_{T}\phi_{l'}^2(t)\phi_{\beta'}^2(t)dt}{(\tau_{l}\tau_{l'}+\lambda)(\tau_{\beta}\tau_{\beta'}+\lambda)}\right)^{\frac{1}{2}}\\
        & \lesssim \sqrt{\frac{1}{n^2}\sum_{i=1}^{n}v_{i}^2 m_{i}^{\underline{3}}\lambda^{-\frac{1}{b_{1}}-2\alpha_{1}}} + \sqrt{\frac{1}{n^2}\sum_{i=1}^{n}v_{i}^2 m_{i}^{\underline{2}}\lambda^{-\frac{1}{b_{1}}-2\alpha_{1}}\log\left(\frac{1}{\lambda}\right)},
    \end{split}
\end{equation*}
and the result follows from Chebyshev's inequality.\\

\noindent
For the second part, observe that
\begin{equation*}
    \begin{split}
       & \mathbb{E}\|(\Lambda_{2}+\lambda I)^{-\frac{1}{2}}(\Lambda_{2}-\hat{T}_{n})(\Lambda_{2}+\lambda I)^{-1}\Lambda_{2}^{\alpha_{1}+\frac{1}{2}}\|_{\text{op}}\\
       \lesssim & \left(\frac{1}{n^2}\sum_{i=1}^{n}v_{i}^2 m_{i}^{\underline{3}}\sum_{l, l'}\sum_{\beta, \beta'} \frac{(\tau_{\beta}\tau_{\beta'})^{2\alpha_{1}}\int_{T}\phi_{l}^2(s)\phi_{\beta}^2(s)\,ds \langle \phi_{l'}, \phi_{\beta'}\rangle_{L^2(T)}^2}
        {(\tau_{l}\tau_{l'}+\lambda)(\tau_{\beta}\tau_{\beta'}+\lambda)}\right)^{\frac{1}{2}}\\
        & + \left(\frac{1}{n^2}\sum_{i=1}^{n}v_{i}^2 m_{i}^{\underline{2}}\sum_{l,l'}\sum_{\beta, \beta'} \frac{(\tau_{\beta}\tau_{\beta'})^{2\alpha_{1}}\int_{T}\phi_{l}^2(s)\phi_{\beta}^2(s)\,ds \int_{T}\phi_{l'}^2(t)\phi_{\beta'}^2(t)dt}{(\tau_{l}\tau_{l'}+\lambda)(\tau_{\beta}\tau_{\beta'}+\lambda)}\right)^{\frac{1}{2}}\\
        \lesssim & \sqrt{\frac{1}{n^2}\sum_{i=1}^{n}v_{i}^2 m_{i}^{\underline{3}}\lambda^{-\frac{1}{b_{1}}}} + \sqrt{\frac{1}{n^2}\sum_{i=1}^{n}v_{i}^2 m_{i}^{\underline{2}}\lambda^{-\frac{1}{b_{1}}}\log\left(\frac{1}{\lambda}\right)},
    \end{split}
\end{equation*}
where the last step follows using Assumption~\ref{covariance_embedding}. Hence the result follows from Markov's inequality.
\end{proof}

\begin{lemma}\label{covariance_power_bound}
    Under Assumptions~\ref{mean_moment} and \ref{covariance_eigen_decay}, 
    we have 
    \begin{equation*}
        \|(\Lambda_{2}+\lambda I)(\hat{T}_{n}+\lambda I)^{-1}\|_{\emph{op}} \le_p 2, ~ \quad \forall~\lambda \gg \max\{\left(\frac{1}{n^2}\sum_{i=1}^{n}v_{i}^2 m_{i}^{\underline{3}}\right)^{\frac{b_{1}}{1+b_{1}}}, \left(\frac{1}{n^2}\sum_{i=1}^{n}v_{i}^2 m_{i}^{\underline{2}}\log\left(\frac{1}{\lambda}\right)\right)^{\frac{b_{1}}{1+b_{1}}}\}.
    \end{equation*}
\noindent
Further, if Assumption~\ref{covariance_embedding} holds, then we obtain
\begin{equation*}
     \|(\Lambda_{2}+\lambda I)^{\frac{1}{2}}(\hat{T}_{n}+\lambda I)^{-\frac{1}{2}}\|_{\emph{op}} \leq_p \sqrt{2}, ~ \quad \forall~\lambda \gg \max\{\left(\frac{1}{n^2}\sum_{i=1}^{n}v_{i}^2 m_{i}^{\underline{3}}\right)^{\frac{b_{1}}{1+2 \alpha_{1}b_{1}}}, \left(\frac{1}{n^2}\sum_{i=1}^{n}v_{i}^2 m_{i}^{\underline{2}}\log\left(\frac{1}{\lambda}\right)\right)^{\frac{b_{1}}{1+2 \alpha_{1}b_{1}}}\}.
\end{equation*}
\end{lemma}
\begin{proof}
   Observe that
   \begin{align*}
   \begin{split}
       \|(\hat{T}_{n}+\lambda I)^{-1}(\Lambda_{2}+\lambda I)\|_{\text{op}} & =  \|(I-(\Lambda_{2}+\lambda I)^{-1}(\Lambda_{2}-\hat{T}_{n}))^{-1}\|_{\text{op}}\\
       & \leq  \frac{1}{1-\|(\Lambda_{2}+\lambda I)^{-1}(\Lambda_{2}-\hat{T}_{n})\|_{\text{op}}},
       \end{split}
   \end{align*}
provided $\|(\Lambda_{2}+\lambda I)^{-1}(\Lambda_{2}-\hat{T}_{n})\|_{\text{op}} <1$, which follows by using condition on regularization parameter and Lemma~\ref{covarianve_empirical_operator_without_embedding}, and the result follows. Similarly, note that
\begin{align*}
    \begin{split}
        \|(\Lambda_{2}+\lambda I)^{\frac{1}{2}}(\hat{T}_{n}+\lambda I)^{-\frac{1}{2}}\|^2_{\text{op}} & =  \|(I-(\Lambda_{2}+\lambda I)^{-\frac{1}{2}}(\Lambda_{2}-\hat{T}_{n})(\Lambda_{2}+\lambda I)^{-\frac{1}{2}})^{-1}\|_{\text{op}}\\
        & \leq  \frac{1}{1-\|(\Lambda_{2}+\lambda I)^{-\frac{1}{2}}(\Lambda_{2}-\hat{T}_{n})(\Lambda_{2}+\lambda I)^{-\frac{1}{2}}\|_{\text{op}}},
    \end{split}
\end{align*}
provided $\|(\Lambda_{2}+\lambda I)^{-\frac{1}{2}}(\Lambda_{2}-\hat{T}_{n})(\Lambda_{2}+\lambda I)^{-\frac{1}{2}}\|_{\text{op}}<1$, which follows by condition over regularization parameter iand Lemma~\ref{covariance_empirical_operator_with_embedding}, and the result follows.
\end{proof}

\begin{lemma}\label{lem:covariance-range}
Let $\alpha_{1}>0$ and suppose that
$$X \in \mathcal{R}(\Lambda^{\alpha_{1}})
\qquad \text{a.s.}, \qquad \text{and}\quad \mathbb{E}\big\|
\Lambda^{-\alpha_{1}}X
\big\|_{L^{2}(T)}^{2}\leq R <\infty.$$

\noindent
Then the mean function
$$
\mu_{0}=\mathbb{E}[X] \in \mathcal{R}(\Lambda^{\alpha_{1}}), \qquad \text{such that} \quad \|\Lambda^{-\alpha_{1}}\mu_{0}\|^2_{L^2(T)} \leq R < \infty,
$$
and the covariance
function
$$C_{0}(s,t)
=
\mathbb{E}\big[
\{X(s)-\mu_{0}(s)\}
\{X(t)-\mu_{0}(t)\}
\big] \in \mathcal{R}(\Lambda^{\alpha_{1}})
\otimes
\mathcal{R}(\Lambda^{\alpha_{1}})
=
\mathcal{R}(\Lambda_{2}^{\alpha_{1}})$$
 such that $\|\Lambda_{2}^{-\alpha_{1}}C_{0}\|_{L^2(T \times T)} \leq R < \infty$
\end{lemma}

\begin{proof}
Let $\{(\tau_{k},\psi_{k})\}_{k\geq1}$ denote the eigenpairs of
$\Lambda$, where $\{\psi_{k}\}_{k\geq1}$ is an orthonormal basis of
$L^{2}(T)$. Since
$X\in\mathcal{R}(\Lambda^{\alpha_{1}})$ almost surely, we have
$$X=\Lambda^{\alpha_{1}}Z, \qquad Z = \sum_{k}
\frac{\langle X,\psi_{k}\rangle_{L^{2}(T)}}
{\tau_{k}^{\alpha_{1}}}\psi_{k},$$
where

$$ \|Z\|_{L^{2}(T)}^{2} = \sum_{k\geq1} \frac{ |\langle X,\psi_{k}\rangle_{L^{2}(T)}|^{2}
}{
\tau_{k}^{2\alpha_{1}}
}
<\infty
\qquad \text{a.s.}$$
Then we can see that 

$$\mu_{0} = \mathbb{E}[X] = \mathbb{E}[\Lambda^{\alpha_{1}}Z] = \Lambda^{\alpha_{1}} \mathbb{E}[Z] \in \mathcal{R}(\Lambda^{\alpha_{1}}).$$

\noindent
Next, using the spectral expansion of $X$, we have
\begin{align*}
\mathbb{E}[X(s)X(t)]
&=
\sum_{j,k\geq1}
\mathbb{E}\left[
\langle X,\psi_{j}\rangle_{L^{2}(T)}
\langle X,\psi_{k}\rangle_{L^{2}(T)}
\right]
\psi_{j}(s)\psi_{k}(t) \\
&=
\Lambda_{2}^{\alpha_{1}}
\left(
\sum_{j,k\geq1}
\mathbb{E}\left[
\frac{
\langle X,\psi_{j}\rangle_{L^{2}(T)}
\langle X,\psi_{k}\rangle_{L^{2}(T)}
}{
\tau_{j}^{\alpha_{1}}\tau_{k}^{\alpha_{1}}
}
\right]
\psi_{j}\otimes\psi_{k}
\right),
\end{align*}
where we have used
$$\Lambda_{2}^{\alpha_{1}}
(\psi_{j}\otimes\psi_{k}) =
(\tau_{j}\tau_{k})^{\alpha_{1}}
(\psi_{j}\otimes\psi_{k})
= \tau_{j}^{\alpha_{1}}\tau_{k}^{\alpha_{1}}
(\psi_{j}\otimes\psi_{k}).$$

It remains to verify that the term inside the parentheses belongs to
$L^{2}(T\times T)$. By the Cauchy--Schwarz inequality,
\begin{align*}
&\left\| \sum_{j,k}
\mathbb{E}\left[
\frac{ \langle X,\psi_{j}\rangle \langle X,\psi_{k}\rangle}{\tau_{j}^{\alpha_{1}}\tau_{k}^{\alpha_{1}}}\right]
\psi_{j}\otimes\psi_{k} \right\|_{L^{2}(T\times T)}^{2} \\
&\quad= \sum_{j,k}
\left| \mathbb{E}\left[ \frac{ \langle X,\psi_{j}\rangle
\langle X,\psi_{k}\rangle}{\tau_{j}^{\alpha_{1}}\tau_{k}^{\alpha_{1}}
}\right]\right|^{2} \\
&\quad\leq
\sum_{j,k}
\mathbb{E}\left[
\frac{|\langle X,\psi_{j}\rangle|^{2}}
{\tau_{j}^{2\alpha_{1}}}
\right]
\mathbb{E}\left[
\frac{|\langle X,\psi_{k}\rangle|^{2}}
{\tau_{k}^{2\alpha_{1}}}
\right] \\
&\quad=
\left\{
\mathbb{E}
\|\Lambda^{-\alpha_{1}}X\|_{L^{2}(T)}^{2}
\right\}^{2}
<\infty.
\end{align*}
Hence
$$\mathbb{E}[X\otimes X] \in
\mathcal{R}(\Lambda_{2}^{\alpha_{1}}).$$

\noindent
Since
$\mu_{0}\in\mathcal{R}(\Lambda^{\alpha_{1}})$, we also have
$$\mu_{0}\otimes\mu_{0} \in
\mathcal{R}(\Lambda^{\alpha_{1}}) \otimes
\mathcal{R}(\Lambda^{\alpha_{1}}) = \mathcal{R}(\Lambda_{2}^{\alpha_{1}}).$$
Therefore,
$$C_{0} = \mathbb{E}[X\otimes X] - \mu_{0}\otimes\mu_{0} \in \mathcal{R}(\Lambda_{2}^{\alpha_{1}}).$$
\end{proof}

\section{Technical Results}\label{app:tech}
\numberwithin{equation}{section}
In this appendix, we collect some technical results that are used in proving the main results of the paper.
\begin{lemma}\cite[Lemma A.3]{gupta2025optimal}
  \label{reducing_power}
  Let $T$ and $\hat{T}$ be positive operators from $H$ to $H$. Then, for any $n \geq 1$, we have
    \begin{align*}
        (\hat{T}+\lambda I )^{-n} - (T+\lambda I )^{-n} = & (\hat{T}+\lambda I )^{-(n-1)}[(\hat{T}+\lambda I )^{-1} - (T+\lambda I )^{-1}]\\
        & + \sum_{i=1}^{n-1}(\hat{T}+\lambda I )^{-i}(T-\hat{T})(T+\lambda I )^{-(n+1-i)}.
\end{align*}
\end{lemma}  

\begin{lemma}\cite[Lemma A.10]{gupta2025optimal}\label{sup_bound}
Let $a,m,p,q$ and $l$ be positive numbers. Then for $r = \min \{p,\frac{lq}{m}+a\}$ and $a< p$, we have
    $$\sup_{i \in \mathbb{N}}\left[\frac{i^{-(p-a)m}}{(i^{-q}+\lambda)^l}\right] \leq \lambda^{\frac{(r-a)m-ql}{q}},~ \forall~ \lambda >0.$$
\end{lemma}

\begin{lemma}\cite[Lemma A.11]{gupta2025optimal}
\label{ch_2:seriessum}
    For $\alpha >1$, $\beta >1,$ and $q \geq \frac{\alpha}{\beta}$, we have 
    $$\sum_{i\in \mathbb{N}}\frac{i^{-\alpha}}{(i^{-\beta}+\lambda)^q} \lesssim \lambda^{-\frac{1+\beta q -\alpha}{\beta}}, ~~ \forall ~\lambda >0.$$
\end{lemma}

\begin{lemma}[Varshamov-Gilbert bound \cite{tsyback2009lb}] 
\label{VGbound}
Let $M \geq 8$. Then there exists a subset $\Theta = \{\theta^{(0)},\ldots,\theta^{(N)}\} \subset \{0,1\}^{M}$ such that $\theta^{(0)}=(0,\cdots,0)$,
\begin{equation*}
    H(\theta,\theta^{'}) > \frac{M}{8}, \quad \forall ~~ \theta \neq \theta^{'} \in \Theta,
\end{equation*}
where $\displaystyle H(\theta, \theta^{'}) = \sum_{i=1}^{M}(\theta_{i}-\theta^{'}_{i})^2 $ is the Hamming distance and $N \geq 2^{\frac{M}{8}}$.
\end{lemma}

\begin{lemma}
    Let $\{\lambda_{i}\}_{i \in \mathbb{N}}$ be a sequence of real numbers such that $c i^{-a} \leq \lambda_{i} \leq C i^{-a}$ for positive values $c, C$ and $a > 1$ and $\{\psi_{i}\}_{i \in \mathbb{N}}$ forms an orthonormal basis of $L^2(T)$. Then for any $f \in L^2(T)$

    \begin{equation*}
        \sum_{j,k}\frac{\lambda_{j}\lambda_{k}}{(\lambda_{j}\lambda_{k}+\lambda)} \langle f, \psi_{k} \rangle_{L^2(T)}^2 \lesssim \lambda^{-\frac{1}{a}}.
    \end{equation*}
\end{lemma}

\begin{proof}
For any fixed $k\in\mathbb N$, consider

$$ S_k:=\sum_{j=1}^{\infty} \frac{\lambda_j\lambda_k} {(\lambda_j\lambda_k+\lambda)}. $$

Since $\lambda_j\asymp j^{-a}$, we have

$$ S_k \lesssim \sum_{j=1}^{\infty} \frac{j^{-a}\lambda_k}{(j^{-a}\lambda_k+\lambda)} = \sum_{j=1}^{\infty}\frac{1}{(1+(\lambda/\lambda_k)j^a)}.
$$

Let $ J_k:=\left(\frac{\lambda_k}{\lambda}\right)^{1/a}.$ Then we have

\begin{align*}
S_k &\lesssim \sum_{j\leq J_{k}}\frac{1}{(1+(\lambda/\lambda_k)j^a)} + \sum_{j>J_{k}}
\frac{1}{(1+(\lambda/\lambda_k)j^a)}\\
&\lesssim \sum_{j \leq J_{k}} 1+ \sum_{j > J_{k}}\frac{\lambda_k}{\lambda}j^{-a}  \lesssim J_{k} + \frac{\lambda_{k}}{\lambda}J_{k}^{1-a}\lesssim \left(\frac{\lambda_k}{\lambda}\right)^{1/a},
\end{align*}

where for $j \leq J_{k}$ we use $(1+(\lambda/\lambda_k)j^a) \geq 1$ and for $j > J_{k}$ we use $(1+(\lambda/\lambda_k)j^a) \geq (\lambda/\lambda_k)j^a$. Therefore,

\begin{align*}
\sum_{j,k} \frac{\lambda_j\lambda_k}
{(\lambda_j\lambda_k+\lambda)} \langle f,\psi_k\rangle_{L^2(T)}^2 &= \sum_k \left( \sum_j \frac{\lambda_j\lambda_k} {(\lambda_j\lambda_k+\lambda)} \right) \langle f,\psi_k\rangle_{L^2(T)}^2\\
&\lesssim \lambda^{-1/a} \sum_k \lambda_{k}^{1/a}\langle f,\psi_k\rangle_{L^2(T)}^2\\
&\lesssim \lambda^{-1/a}\|f\|_{L^2(T)}^2.
\end{align*}
\end{proof}

\begin{lemma}
    Let $\{\lambda_{i}\}_{i \in \mathbb{N}}$ be a sequence of real numbers such that $c i^{-a} \leq \lambda_{i} \leq C i^{-a}$ for positive values $c, C$ and $a > 1$. Then 
    \begin{equation*}
        \sum_{j,k}\frac{\lambda_{j}\lambda_{k}}{(\lambda_{j}\lambda_{k}+\lambda)}  \lesssim \lambda^{-\frac{1}{a}} \log\left(\frac{1}{\lambda}\right).
    \end{equation*}
\end{lemma}

\begin{proof}

Let $N = \lambda^{-\frac{1}{a}}$. We start by splitting the sum according to whether $jk\leq N$ or $jk>N$. Thus,

\begin{align*}
    \sum_{j,k}\frac{\lambda_{j}\lambda_{k}}{(\lambda_{j}\lambda_{k}+\lambda)} \lesssim & \sum_{j,k}\frac{(jk)^{-a}}{((jk)^{-a}+\lambda)}\\
    = & \sum_{jk \leq N} \frac{(jk)^{-a}}{((jk)^{-a}+\lambda)} + \sum_{jk > N} \frac{(jk)^{-a}}{((jk)^{-a}+\lambda)}\\
    \lesssim & \sum_{jk \leq N} 1 + \sum_{jk > N} N^a (jk)^{-a}
\end{align*}

For the first term,

$$ \sum_{jk\leq N}1 = \sum_{j\leq N}\left\lfloor\frac{N}{j}\right\rfloor \leq N\sum_{j\leq N}\frac{1}{j} \lesssim N\log N.$$

For the second term, write
$$ \sum_{jk>N}(jk)^{-a} = \sum_{j\leq N}j^{-a} \sum_{k>N/j}k^{-a} + \sum_{j>N}j^{-a}\sum_{k\geq1}k^{-a}. $$

Since $a>1$, $\sum_{k>x}k^{-a}\lesssim x^{1-a}.$ Consequently,

\begin{align*}
N^a\sum_{j\leq N}j^{-a} \sum_{k>N/j}k^{-a} &\lesssim N^a\sum_{j\leq N}j^{-a} \left(\frac{N}{j}\right)^{1-a} =
N\sum_{j\leq N}\frac{1}{j}\lesssim N\log N.
\end{align*}

Moreover, since $\sum_{k\geq1}k^{-a}<\infty$,

$$
N^a\sum_{j>N}j^{-a}\sum_{k\geq1}k^{-a}
\lesssim
N^a N^{1-a}
=
N.
$$

Combining the above estimates gives

$$
\sum_{j,k}
\frac{\lambda_j\lambda_k}{\lambda_j\lambda_k+\lambda}
\lesssim
N\log N.
$$

Finally, since $N=\lambda^{-1/a}$,

$$
N\log N
=
\frac1a\lambda^{-1/a}\log\left(\frac1\lambda\right),
$$

and hence

$$
\boxed{
\sum_{j,k}
\frac{\lambda_j\lambda_k}
{\lambda_j\lambda_k+\lambda}
\lesssim
\lambda^{-1/a}\log\left(\frac1\lambda\right).
}
$$

\end{proof}

\end{document}